\documentclass{amsart}

\usepackage{amsmath}
\usepackage{amsfonts}
\usepackage{amssymb}
\usepackage{upgreek}
\usepackage{mathtools}
\usepackage{hyperref}
\usepackage{mathrsfs}
\usepackage{stmaryrd}
\usepackage[inline]{enumitem}
\usepackage{cleveref}

\newtheorem{theorem}{Theorem}[section]
\newtheorem{lemma}[theorem]{Lemma}
\newtheorem{corollary}[theorem]{Corollary}

\theoremstyle{definition}
\newtheorem{definition}[theorem]{Definition}

\theoremstyle{remark}
\newtheorem{remark}[theorem]{Remark}

\numberwithin{equation}{section}

\DeclareSymbolFont{matha}{OML}{txmi}{m}{it}
\DeclareMathSymbol{\varv}{\mathord}{matha}{118}
\DeclarePairedDelimiter{\norm}{\lVert}{\rVert}
\newcommand*\diff{\mathop{}\!\mathrm{d}}
\newcommand{\abs}[1]{\lvert#1\rvert}
\newcommand{\ndg}[1]{| \kern -.25mm \|{#1}| \kern -.25mm \|}
\newcommand{\timeint}[1]{\int_{0}^{T} {#1} \diff{t}}
\newcommand{\spaceint}[1]{\int_{\Omega}^{} {#1} \diff{x}}
\newcommand{\spacetimeint}[1]{\int_{\Omega^{}_{T}} {#1} \diff{x} \diff{t}}
\newcommand{\timeintt}[1]{\int_{0}^{\uptau} {#1} \diff{t}}

\newcommand{\intJn}[1]{\int_{t^{}_{n-1}}^{t^{}_{n}} {#1} \diff{t}}

\begin{document}

\title[An optimal control problem for the  Allen-Cahn equation]{Analysis and approximations of an optimal control problem for the Allen-Cahn equation}


\author[ K. Chrysafinos]{KONSTANTINOS CHRYSAFINOS}
\address{}
\curraddr{}
\email{chrysafinos@math.ntua.gr}
\thanks{Department of Mathematics, School of Mathematical and Physical Sciences, National Technical University of Athens, Zografou 15780, Greece and IACM, FORTH, 20013 Heraklion, Crete, Greece.}

\author[D. Plaka]{DIMITRA PLAKA}
\address{}
\curraddr{}
\email{dplaka@central.ntua.gr}
\thanks{Department of Mathematics, School of Mathematical and Physical Sciences, National Technical University of Athens, Zografou 15780, Greece.}


\date{}

\dedicatory{}

\begin{abstract}
The scope of this paper is the analysis and approximation of an optimal control problem related to the Allen-Cahn equation. A tracking functional is minimized subject to the Allen-Cahn equation using distributed controls that satisfy point-wise control constraints. First and second order necessary and sufficient conditions are proved. The lowest order discontinuous Galerkin - in time - scheme is considered for the approximation of the control to state and adjoint state mappings. Under a suitable restriction on maximum size of the temporal and spatial discretization parameters $k$, $h$ respectively in terms of the parameter $\epsilon$ that describes the thickness of the interface layer, a-priori estimates are proved with constants depending polynomially upon $1/ \epsilon$. Unlike to previous works for the uncontrolled Allen-Cahn problem our approach does not rely on a construction of an approximation of the spectral estimate, and as a consequence our estimates are valid under low regularity assumptions imposed by the optimal control setting. These estimates are also valid in cases where the solution and its discrete approximation do not satisfy uniform space-time bounds independent of $\epsilon$. These estimates and a suitable localization technique, via the second order condition (see \cite{Arada-Casas-Troltzsch_2002,Casas-Mateos-Troltzsch_2005,Casas-Raymond_2006,Casas-Mateos-Raymond_2007}), allows to prove error estimates for the difference between local optimal controls and their discrete approximation as well as between the associated state and adjoint state variables and their discrete approximations 
\end{abstract}

\maketitle

\section{Introduction}\label{sec1}

We consider the following distributed optimal control problem that is governed by the Allen-Cahn equation: Minimize
\begin{equation}\label{cost-fun}
\begin{aligned}
   J(u) & = \ \frac{1}{2} \timeint{ \spaceint{ \abs{y^{}_{u}(t,x) - y^{}_d(t,x)}^{2}_{} }} + \frac{\gamma}{2} \spaceint{ \abs{y^{}_{u}(T,x) - y_{\Omega}^{}(x)}^2_{}} \\
  & \quad \  + \frac{{\mu}}{2} \timeint{ \spaceint{ \abs{u(t,x)}^2_{}}},
\end{aligned}
\end{equation}
subject to
\begin{equation}\label{state}
\begin{aligned}
   y_{u,t}^{} - \Delta y_u + \frac{1}{{\epsilon}^2}(y^3_{u} -y_{u}) & = u  & \text{in} & \ \Omega^{}_{T}=\Omega \times (0,T), \\
  \frac{\partial y_u}{\partial n}  &= 0 \quad\text{or} \quad  y_{u} =0  & \text{on} & \ \Sigma^{}_{T}=\partial\Omega \times (0,T), \\
   y_{u}(\cdot,0)&= y_0^{}   & \text{in} & \ \Omega.
\end{aligned}
\end{equation}

The Allen-Cahn equation can be viewed as a prototype equation that describes a phase separation process. The physical meaning of \eqref{state} for the uncontrolled case with zero Neumann boundary data (i.e. when $u=0$ and $\frac{\partial y}{\partial n} = 0$) is well understood. Indeed,  denoting the nonlinear term by $F:\mathbb R \to \mathbb R, \ F(s):=\frac{1}{\epsilon^2} (s^3-s)$, we observe that $F(s)$ represents the derivative of the classical {\it double-well potential}. Hence, at the uncontrolled case,  it is well known that the solution (denoted by $y$) of \eqref{state} develops time-dependent interfaces $\Gamma_t := \{ x \in \Omega : y(t,x) = 0 \} $ that separate regions where $y \approx 1$ from regions where $y \approx -1$. Furthermore, it is well understood that the solution moves from one region to the other very fast, with typical length of ``diffuse interface"  ${\mathcal O}(\epsilon)$ and that the Allen-Cahn equation satisfies a maximum principle in the sense that if the initial data satisfies $\abs{y_0(x)} \leq 1$ then the solution also satisfies $\abs{y(t,x)} \leq 1$. Furthermore, the ``smallness" of the parameter $\epsilon <<1 $ should be taken into consideration in both the analysis as well as the numerical analysis of problems involving the Allen-Cahn equation. We refer the reader to \cite{Du-Feng_2019} for an overview of various phase field models, their analysis and approximations.

Our work is devoted to the analysis and numerical approximation of an optimal control problem related to \eqref{state}. In particular, our goal is to intervene to the dynamics of \eqref{state} by using a control function $u$, in order to guarantee that the solution $y_u$ will be as close as possible to a given target profile $y_d$. It is crucial to observe that in the optimal control setting our goal is to intervene to the dynamics of an initial profile $y_0$ that does not necessary satisfy $\abs{y_0(x)} \leq 1$ and/or does not possess additional smoothness properties. In addition, our  possible target profile $y_d$ is not necessarily $\{+1,-1 \}$ and/or satisfies additional smoothness properties.

In \eqref{cost-fun} $\mu> 0$ denotes the typical Tikhonov regularization term, while the inclusion of the terminal tracking term in the functional (with $\gamma \geq 0$) facilitates the effectivity of approximations near the end point of the time interval.
Throughout this work the set of admissible controls is defined as
\begin{equation*}
    U^{}_{ad} = \Big \lbrace u \in L^{}_2 \left(0,T;L^{}_2(\Omega) \right); \ u^{}_{a} \leq u(t,x) \leq u^{}_{b} \ \  \text{for a.e.} \ (t,x) \in  [0,T] \times \Omega \Big \rbrace,
\end{equation*}
and the optimal control problem is formulated, in the standard reduced functional form, as
\begin{equation}\label{P}
\begin{cases}
  & \min J(u) \\
  & u \in U^{}_{ad}.
\end{cases}
\end{equation}

Note that the above optimal control problem is non-convex. Thus, it is necessary to distinguish between local and global solutions. A control $\bar{u} \in U^{}_{ad}$ is said to be a local optimal control of \eqref{P} in the sense of $L^{}_{2}(0,T;L^{}_2(\Omega))$, if there exists $\alpha >0$ such that $J(\bar{u}) \leq J(u)$ for all $ u \in U^{}_{ad}\cap B^{}_{\alpha}(\bar{u})$, where $ B^{}_{\alpha}(\bar{u})$ is the open ball of $L^{}_{2}(0,T;L^{}_2(\Omega))$ centered at $\bar{u}$ with radius $\alpha$. We shall use the notation $\bar{y}:= y_{\bar u}$ and $\bar{\varphi} := \varphi_{\bar u}$ for the associated state and adjoint state solutions, respectively.

Optimal control problems having states constrained to semi-linear parabolic PDEs have been extensively studied. We refer the reader to \cite{Troltzsch_2010} (see also references within) for an overview of optimal control problems related to classical elliptic and parabolic semi-linear PDEs. In fact, various issues such as existence, first and second order necessary and sufficient conditions have been considered even for nonstandard optimal control problems related to semi-linear parabolic PDEs: see for instance  \cite{Casas-Kruze-Kunisch_2017} (BV controls),  \cite{Casas-Mateos-Rosch_2019} (control problems in absence of Tikhonov regularization term), \cite{Casas-Herzog-Wachsmuth_2017,Casas-Ryll-Troltzsch_2015,Casas-Ryll-Troltzsch_2018,Casas-Troltzsch_2019} (and references within) (control problems with sparse controls),  \cite{Meyer-Susu_2017} (control problems for non smooth semilinear parabolic PDEs).

The numerical analysis of optimal control problems with semi-linear parabolic PDEs as constraints have been considered in  \cite{Neitzel-Vexler_2012} (control constraints, piecewise constants / linear controls / variational discretization approach),  \cite{Chrysafinos-Karatzas_2012} (discontinuous in time schemes, no constraints), and \cite{Casas-Mateos-Rosch_2019} (error estimates in absence of Tikhonov term). In these works, error estimates of fully-discrete approximations have been presented under a monotonicity assumption on the semilinear term. In \cite{Casas-Chrysafinos_2012,Casas-Chrysafinos_2015,Casas-Chrysafinos_2016,Casas-Chrysafinos_2017} the numerical analysis of optimal control problems related to the evolutionary Navier-Stokes equations, including error estimates, was considered. Finally, for related works for other nonlinear parabolic PDEs, we refer the reader to cite \cite{Garcke-Hinze-Kahle_2019} (time-discrete two-phase flows), \cite{Holtmannspotter-Rosch_2021} (a-priori estimates for a coupled semilinear PDE-ODE system).

A common ingredient in the error analysis of fully-discrete schemes for optimal control problems related to nonlinear parabolic PDEs, under control constraints (such as \cite{Neitzel-Vexler_2012,Casas-Mateos-Rosch_2019,Casas-Chrysafinos_2012,Casas-Chrysafinos_2015,Casas-Chrysafinos_2016,Casas-Chrysafinos_2017,Holtmannspotter-Rosch_2021}) is the use of results regarding the Lipschitz  continuity of the control to state and of state to adjoint mappings, the derivation of first and second order necessary and sufficient optimality conditions, and detailed error estimates of the corresponding control to state, and state to adjoint mappings that allow the classical {\it localization} argument of \cite{Arada-Casas-Troltzsch_2002,Casas-Mateos-Troltzsch_2005,Casas-Raymond_2006,Casas-Mateos-Raymond_2007} (developed for error analysis of discretization schemes for semilinear elliptic PDE constrained optimization problems) to work under the prescribed regularity assumptions.

However, to our best knowledge none of the above key results are present in case of the Allen-Cahn equation. For instance, observe that the Allen-Cahn equation involves a nonmonotone nonlinearity that satisfies $\frac{1}{\epsilon^2} F'(s):= \frac{1}{\epsilon^2} (3s^2 - 1) \geq - \frac{1}{\epsilon^2}.$ As a consequence, for realistic values of $\epsilon$ the {\it classical} approaches for proving the Lipschitz continuity of the control to state mapping as well as its numerical analysis, fail since they introduce constants depending exponentially upon $\frac{1}{\epsilon^2}$. For the  numerical analysis of the uncontrolled Allen-Cahn equation, i.e. for $u=0$, this  difficulty was first circumvented in the seminal works of  \cite{Feng-Prohl_2003}, \cite{Kessler-Nochetto-Schmidt_2004}, \cite{Feng-Wu_2005}, \cite{Bartels-Muller-Ortner_2011}, \cite{Bartels-Muller_2011},  where error estimates (a-priori and a-posteriori) were established for the homogeneous Allen-Cahn equation with constants that dependent polynomially upon $\frac{1}{\epsilon}$ based on suitable discrete approximation of the {\it spectral estimate} and a nonstandard continuation argument of the form of a nonlinear Gronwall Lemma. Recall that the {\it spectral estimate} concerns the principle eigenvalue of the linearized Allen-Cahn operator about the solution $y$ of the homogeneous (uncontrolled) Allen-Cahn equation. In particular, for the case of zero Neumann boundary data, it involves the quantity
\begin{equation}\label{spectral_classical}
-\Lambda^{}(t) : = {\displaystyle\inf_{\varv\in H^{1}(\Omega) \setminus \lbrace{0 \rbrace}}} \frac{\norm{\nabla \varv}^2_{L^{}_2(\Omega)} + \epsilon^{-2} \left(( 3y^2_{} -1) \varv,\varv \right)}{\norm{\varv}^2_{L^{}_2(\Omega)}}.
\end{equation}
The celebrated works \cite{Chen_1994,Mottoni-Schatzman_1995,Alikakos-Fusco_1993}  showed that for the homogeneous (uncontrolled) equation, $\Lambda \in L^{}_{\infty}(0,T)$ can be bounded independently of $\epsilon$ for the case of smooth, evolved interfaces. For the nonhomogeneous case, in Section 5, we show that a similar estimate holds, with a different bound that it is still independent of $\epsilon$, provided that the spectral estimate is valid in the homogeneous (uncontrolled) case with the same initial data.

In addition, the numerical analysis of the homogeneous (uncontrolled) problem, cannot be directly employed in the optimal control setting since such analysis typically requires regularity assumptions on the state (as well as to its fully-discrete counterpart) that are not available within the optimal control context. For example, the assumption $y \in L_{\infty}(0,T;H^2(\Omega)) \cap H_{}^{1}(0,T;H_{}^{2}(\Omega)) \cap H^2(0,T;L_2(\Omega))$ (used in the above mentioned works for a-priori estimates of the uncontrolled Allen-Cahn equation) is not present in our optimal control setting. In addition, similar difficulties arise when dealing with analysis and numerical analysis of the state to adjoint mapping. In fact, it turns out that the analysis can be even harder due to the absence of the cubicterm that generates the $L_{4}(0,T;L_4(\Omega))$ norm when testing \eqref{state} with $y.$

Our work aims to present an error estimate for a fully-discrete scheme based on the discontinuous -in time- Galerkin dG(0) framework, (piece-wise constant approximations in time, finite elements based on piecewise linear polynomials approximations in space). First, we present a detailed analysis of the control to state, and of the state to adjoint mappings and we derive first and second order optimality conditions. In particular, under a closeness assumption on the controls, we establish the Lipschitz continuity of the control to state mapping, with Lipschitz constant that is independent of $\epsilon$ in the $L_{\infty}(0,T;L_2(\Omega))$ norm  by exploiting the presence of $L_4(0,T;L_4(\Omega))$ critical term, the {\it spectral estimate} (for the nonohomogeneous case) and the  nonlinear Gronwall Lemma. For the state to adjoint mapping, using similar technical tools we are able to obtain Lipschitz results with constants that depend  polynomially upon $\frac{1}{\epsilon}$ (see Section 3).

The numerical analysis of the control to state under low regularity assumptions is based on a slight but critical modification of the approach by \cite{Bartels-Muller-Ortner_2011,Bartels-Muller_2011} that is the construction of a globally space-time projection as the dG(0) solution of a heat equation with right-hand side $y_t - \Delta y$ similar to earlier works of \cite{Chrysafinos-Walkington_2010}  (uncontrolled Navier-Stokes equations), \cite{Casas-Chrysafinos_2012} (for controlled Navier-Stokes) and \cite{Chrysafinos_2019} (for the uncontrolled Allen-Cahn). In our approach we do not assume any point-wise space-time bound of the fully-discrete solution of the control to state mapping and most crucially we do not construct a discrete approximation of the {\it spectral estimate}, resulting an estimate that is valid under the limited regularity assumptions imposed by our optimal control setting.

For the numerical analysis of the discrete adjoint state mapping, the key difficulty involves the development of discrete stability and error bounds that depend polynomially upon $\frac{1}{\epsilon}$. Indeed, we note that the spectral estimate is not longer valid if we replace $y$ by its discretization and the direct application of the nonlinear Gronwall Lemma will lead to severe restrictions on size of $\|y_u-y_d\|^{}_{L_2(I;L_2(\Omega))}$ and $\|y(T)-y_\Omega\|^{}_{L^{}_2(\Omega)}$ in terms of $\epsilon$ that are not appropriate in the optimal control setting. Our approach is based on a pseudo duality argument that avoids the construction of a discrete approximation of the spectral estimate and the use of a nonlinear Gronwall Lemma resulting to discrete stability estimates. Then, for the derivation of error estimates for the discrete state to adjoint mapping, we employ a {\it boot-strap} argument. Note that the error analysis of the state to adjoint mapping, is of independent interest since it concerns a linear singularly perturbed problem under low regularity assumptions on the given data. Combining these estimates, we are able to proceed similar to \cite{Casas-Chrysafinos_2012} to establish the desired estimates for the difference between local optimal controls and their discrete approximation as well as estimates for the differences between the corresponding state and adjoint state and theirs discrete approximations.

\section{Preliminary setting}
Let $\Omega\subset \mathbb{R}^d$ be a convex, polygonal ($d=2$) or polyhedral ($d=3$) domain of the Euclidean space $\mathbb{R}^d$ and by $I:=[0,T]$ the time interval where $T \in \mathbb{R}^{+}$ is the final time. We denote by $\Omega_T$ the space-time cylinder. We use standard notation for $L_{p}^{}(\Omega)$ spaces with $1 \leq p \leq \infty$ and their norms. We set $W^{k,p}(\Omega)$ for the $k$th order of Sobolev spaces based on $L_{p}^{}(\Omega)$, and we denote by  $H^{k}(\Omega):= W^{k,2}(\Omega)$, $k \geq 0$, for  the Hilbertian space along with the corresponding norms $\norm{ \cdot }^{}_{W^{k,p}(\Omega)}$ and $\norm{ \cdot}^{}_{H^{k}(\Omega)}$, respectively. We denote $H_{0}^{1}(\Omega) := \lbrace \varv \in H_{}^{1}(\Omega) : \varv \vert^{}_{\partial \Omega} = 0 \rbrace$. For a given Banach space  $X$  we denote by $X^*$  its dual. We also consider the Bochner spaces, $L^{}_{p}(I;X)$, endowed with the norms:
\begin{align*}
    \norm{ \varv }^{}_{ L^{}_{p}(I;X) } = \Big( \int_{I} \norm{\varv}^{p}_{X} \diff t \Big)^{1/p},\ p\in[1,+\infty),\quad  \norm{ \varv }^{}_{ L^{}_{\infty}(I;X) } = \operatorname*{ess.sup}_{t\in I} \norm{\varv}^{}_{X}.
\end{align*}
We will abbreviate $\|.\|_{L_{\infty}(I;L_{\infty}(\Omega))} := \|.\|_{_{\infty}}$. We also consider the  space
\begin{equation*}
    W^{2,1}_{p}\left(\Omega^{}_{T}\right) =  \Big \lbrace \varv \in L^{}_{p}\left(\Omega^{}_{T}\right)  : \frac{\partial \varv}{\partial x^{}_{i}}, \frac{\partial^2 \varv}{\partial x^{}_{i} \partial x^{}_{j}}, \frac{\partial \varv}{\partial t} \in L^{}_{p}\left(\Omega^{}_{T}\right) , \ 1 \leq i,j \leq d \Big \rbrace,
\end{equation*}
with the respective norm
\begin{equation*}
\begin{aligned}
   \norm{\varv}^{p}_{ W^{2,1}_{p}\left(  \Omega^{}_{T}\right)} & {=}   \spacetimeint{ \left( \abs{\varv}^p_{} {+}    \Big\lvert \frac{\partial \varv}{\partial t} \Big\rvert^{p}  \right)} {+} \sum_{i=1}^{d} \spacetimeint{ \bigg \lvert \frac{\partial \varv}{\partial x^{}_{i}} \bigg \rvert^{p}} \\
 & {+} \sum_{i,j=1}^{d} \spacetimeint{ \bigg \lvert  \frac{\partial^2 \varv}{\partial x^{}_{i} \partial x^{}_{j}} \bigg \rvert^{p}}.
\end{aligned}
\end{equation*}
We adopt the notation: $H^{2,1}\left(  \Omega^{}_{T}\right) :=  W^{2,1}_{2}\left(  \Omega^{}_{T}\right)$. Furthermore, we consider the set $W(I):= \Big \lbrace \varv \in L^{}_2(I;H^1(\Omega)), \ \varv^{}_t\in L^{}_2(I; (H^1(\Omega))^*) \Big \rbrace.$
We will use extensively the classical \textit{interpolation inequality }  for all   $\varv \in L_4(\Omega),$
\begin{align}\label{inter}
\norm{ \varv}^{3}_{L_3(\Omega)} & \leq \norm{ \varv}_{L_2(\Omega)} \norm{ \varv}^2_{L_4(\Omega)}, \ \ \text{for }  d=2,3,
\end{align}
 and \textit{Gagliardo-Nirenberg-Ladyzhenskaya inequalities (GNL)} for all $\varv \in H^{1}(\Omega)$
\begin{align}
\norm{ \varv}^{}_{L^{}_4(\Omega)} & \leq \tilde{c} \norm{ \varv}^{1/2} _{L^{}_2(\Omega)} \norm{ \varv}^{1/2} _{H^{1}(\Omega)}, \ \text{for } d=2, \label{GNL1} \\
\norm{ \varv}^{}_{L^{}_3(\Omega)} & \leq \tilde{c} \norm{ \varv}^{1/2} _{L^{}_2(\Omega)} \norm{ \varv}^{1/2} _{H^{1}(\Omega)}, \ \text{for } d=3,  \label{GNL2} \\
\norm{ \varv}^{}_{L^{}_4(\Omega)} & \leq \tilde{c} \norm{ \varv}^{1/4} _{L^{}_2(\Omega)} \norm{ \varv}^{3/4} _{H^1(\Omega)}, \  \text{for }d=3, \label{GNL3}
\end{align}
 with $\tilde{c}>0$, independent of $\varv$. Moreover, we consider the \textit{Young's inequality} reading: For any $\delta > 0$, $a,b \geq 0$, $p,q > 1$ and for some $C(p,q)>0$, it holds
\begin{align*}
a b \leq \delta a^p + C(p,q) \delta^{-\frac{q}{p}} b^{q}, \quad  \text{where} \quad  1/p  +  1/ q  = 1.
\end{align*}
Throughout this paper $C$ denotes a positive constant depending only on $\Omega_T$.

Assume that $u \in L_2(I; L_2(\Omega))$, $y_0 \in L_2(\Omega)$, $y^{}_{\Omega} \in H^{1}(\Omega)$. The weak formulation of the state equation \eqref{state} becomes: We seek $y \in W(I) $ such that for a.e. $t \in I$, for all $\varv \in H^1(\Omega),$
\begin{equation}\label{weak-state}
\begin{aligned}
    \langle y_t^{},\varv \rangle + \left( \nabla y, \nabla \varv \right) + \epsilon^{-2} \langle y^3 - y,\varv \rangle  & = \left(u, \varv \right) ,\\
    ( y(0),\varv )  & = ( y_0^{},\varv),
\end{aligned}
\end{equation}
Here, $\left( \cdot, \cdot \right)$ denotes the  standard $L^{}_2(\Omega)$ inner product, and $\langle \cdot , \cdot \rangle$ the associated duality pairing between $(H^{1}(\Omega))^*$ and $H^1(\Omega).$  For any $\epsilon >0$, and $y_0 \in H^1(\Omega)$, it easy to show that  $y \in H^{2,1}_{} \left(\Omega^{}_{T} \right)  \cap C (I;H^{1}(\Omega))$ (see \cite{Temam_1997}). We refer the readers to \cite{Temam_1997, Zeidler_1990} for straightforward techniques that prove the enhanced regularity results.
The following lemma quantifies the dependence upon $\epsilon$ of various norms.
\begin{lemma}\label{stability}
$1.$ Let $u \in L_2(I;L_2(\Omega))$ and $y_0 \in L_2(\Omega).$ Then, there exists a constant $C$ independent of $\epsilon$ such that:
\begin{equation*}
\begin{aligned}
    &\norm{y}^{}_{L_2(I;L_2(\Omega))} {+} \norm{y}^{2}_{L_{4}(I;L_4(\Omega))}  {\leq} C
 \left (  \abs{\Omega_T}^{\frac{1}{2}} {+}\epsilon \norm{y^{}_0}^{}_{L^{}_2(\Omega)} {+}  \epsilon^2 \norm{u}^{}_{L_2(I;L_2(\Omega))} \right ) {=:} C^{}_{\text{st},1},\\
    &\norm{y}^{}_{L_{\infty}(I;L_2(\Omega))}{+} \norm{y}^{}_{L_{2}(I;H^1(\Omega))}  \leq 
C_{\text{st},1} \epsilon^{-1} .
\end{aligned}
\end{equation*}
$2.$ Let $u \in L^{}_2(I;L^{}_2(\Omega))$ and $y_0 \in H^1(\Omega)$. Then, there exists a constant $C$ independent of $\epsilon$,
such that the following estimates hold:
\begin{align*}
&   \norm{y}^{}_{L^{}_{\infty}(I;H^{1}_{}(\Omega))} {+} \norm{y^{}_t}^{}_{L^{}_2(I;L^{}_{2}(\Omega))} {+} \frac{1}{2\epsilon} \|(y^2-1)^2\|^{1/2}_{L_{\infty}(I;L^{}_1(\Omega))}  \\
& \leq  C \left ( \| \nabla y_0 \|^{}_{L^{}_2(\Omega)} {+} \frac{1}{2\epsilon} \|(y^2_0-1)^2\|^{1/2}_{L^{}_1(\Omega)} {+} \|u\|_{L_2(I;L_2(\Omega))} \right ) {:=}  C^{}_{\text{st},2},\\
&    \norm{y}^{}_{L^{}_2(I;H^{2}_{}(\Omega))} \leq  C \left ( T^{1/2} C^{}_{\text{st},2} \epsilon^{-1} {+} \|\nabla y_0\|_{L^{}_2(\Omega)} {+} \|u\|_{L_2(I;L_2(\Omega))} \right ) {:=} C^{}_{\text{st},3}.
\end{align*}
\end{lemma}
\begin{proof}  These estimates are standard (see e.g. \cite{Feng-Prohl_2003,Bartels_2016,Chrysafinos_2019}). For completeness we state the main steps that lead to the estimates. For the first result testing \eqref{weak-state} with $y$, and using Young's inequalities, we obtain
\begin{align*}
& \frac{1}{2} \frac{d}{dt} \|y\|^2_{L^{}_2(\Omega)} + \| \nabla y \|^2_{L_2(\Omega)} + \frac{1}{\epsilon^2} \|y\|^4_{L^{}_4(\Omega)}  \leq \frac{2}{\epsilon^2} \|y\|^2_{L^{}_2(\Omega)} + \frac{\epsilon^2}{4} \|u\|^2_{L^{}_2(\Omega)} \\
& \leq \frac{1}{2 \epsilon^2} \|y\|^4_{L^{}_4(\Omega)} + \frac{2}{\epsilon^2} \abs{\Omega} + \frac{\epsilon^2}{4} \|u\|^2_{L^{}_2(\Omega)}.
\end{align*}
Integrating from $0$   to $T$ and  standard algebra we get the estimate on  $\|{\cdot}\|^{}_{L_{4}(I;L_4(\Omega))}.$ Integrating from $0$ to $t$, and using the bound on $\| {\cdot} \|^{}_{L_{4}(I;L_4(\Omega))}$, we deduce the estimate on $\| {\cdot} \|^{}_{L^{}_{\infty}(I;L^{}_2(\Omega))}$. Note that H\"older and Young's  inequalities yield,
$$\| y\|^{}_{L_2(I;L_2(\Omega))} \leq \abs{\Omega}^{1/4} T^{1/4} \|y\|^{}_{L_{4}(I;L_4(\Omega))}  \leq  C \left ( \abs{ \Omega }^{1/2} T^{1/2} + \|y\|^2_{L_{4}(I;L_4(\Omega))} \right ).$$
The estimate $\|y\|_{L_2(0,T;H^1(\Omega))}$ follows by triangle inequality.
Now we test (\ref{weak-state}) with $y_t$, and we use Young's inequality to get
\begin{align*}
& \|y_t\|^2_{L^{}_2(\Omega)} + \frac{\diff}{\diff{t}} \Big( \frac{1}{2} \| \nabla y\|^2_{L^{}_2(\Omega)} +  \frac{1}{4\epsilon^{2}} \|(y^2{-}1)^2\|_{L^{}_1(\Omega)} \Big) \leq\frac{1}{2} \|y_t\|^2_{L^{}_2(\Omega)} + \frac{1}{2} \|u\|^2_{L^{}_2(\Omega)} .
\end{align*}
Integrating with respect to $t\in(0,\uptau)$,
\begin{align} \label{imp}
&  \|y_t\|^{2}_{L^{}_2(0,\uptau;L_2(\Omega))} +  \| \nabla y(\uptau) \|^2_{L^{}_2(\Omega)} + \frac{1}{2 \epsilon^2} \|(y^2(\uptau){-}1)^2\|_{^{}L_1(\Omega)} \\
& \ \leq  \|u\|^2_{L^{}_2(0,\uptau;L_2(\Omega))} +  \| \nabla y(0)\|^2_{L^{}_2(\Omega)} + \frac{1}{2 \epsilon^2} \|(y^2(0){-}1)^2 \|_{L^{}_1(\Omega)},  \nonumber
\end{align}
which implies the desired estimate. Multiplying \eqref{state} with $- \Delta y,$ integrating by parts with respect to space, and Young's inequality yield,
\begin{align*} \label{imp2}
& \frac{1}{2} \frac{\diff}{\diff{t}} \|\nabla y\|^2_{L^{}_2(\Omega)} {+} \| \Delta y \|^2_{L^{}_2(\Omega)} {+} \frac{3}{\epsilon^2} \|y \nabla y \|^2_{L^{}_2(\Omega)} \\
& \ \leq  \frac{1}{2} \left(\| \Delta y\|^2_{L^{}_2(\Omega)}  {+}  \|u\|^2_{L^{}_2(\Omega))}\right) {+}  \frac{1}{\epsilon^2} \| \nabla  y\|^2_{L^{}_2(\Omega)} , \nonumber
\end{align*}
The estimate is completed after integrating with respect to time and substituting the bounds of  $\|\nabla y\|^2_{L_{\infty}(I;L_2(\Omega))}$.
\end{proof}
\begin{remark} \label{first} 1. If $\epsilon \|\nabla y_0\|_{L_2(\Omega)} + \|y_0\|^2_{L_4(\Omega)} \leq C$ (with $C$  bounded independent of $\epsilon$) then Lemma \ref{stability} implies that $\|y\|^2_{L_{\infty}(I;L_4(\Omega))}$ (and hence $\|.\|^2_{L_{\infty}(I;L_2(\Omega))}$) is bounded independent of $\epsilon$ by $C + \epsilon \|u\|_{L_2(I;L_2(\Omega))}$. In addition, $\|y\|_{L_{\infty}(I;H^1(\Omega))} \leq ( C\epsilon^{-1} + \|u\|_{L_2(I;L_2(\Omega))} )$, and $\|y\|_{H^{2,1}(\Omega_T)} \leq \tilde C \epsilon^{-2}$, where $\tilde C$ is a constant (independent of $\epsilon$), depending on $C$,$\|.\|_{L_2(I,L_2(\Omega))}$ and $T$. \\
2.  If the more stringent condition  $\| \nabla y(0)\|^{}_{L^{}_2(\Omega)} {+}  \frac{1}{2 \epsilon} \|(y^2(0){-1})^2 \|^{1/2}_{L^{}_1(\Omega)} \leq D$ (with $D$ independent of $\epsilon$) holds, then $C_{st,2}$ is bounded independent of $\epsilon$ and hence we deduce $\|y\|^{}_{H^{2,1}(\Omega_T)} \leq C \epsilon^{-1} ( D{+} \|u\|^{}_{L_2(I;L_2(\Omega))} )$  where $C$ is an algebraic constant depending only on the domain. In addition we have 
$\|y\|_{L_{\infty}(I;H^1(\Omega))} \leq C (  D + \|u\|_{L_2(I;L_2(\Omega)}).$\\
3. The condition $\| \nabla y(0)\|^2_{L^{}_2(\Omega)} {+} \frac{1}{4 \epsilon^2} \|(y^2(0){-1})^2 \|^{}_{L^{}_1(\Omega)} \leq D^2_{}$ relates to the assumption that the associated Ginzburg-Landau energy  is bounded at $0$ (see for instance \cite{Feng-Prohl_2003}) and it is commonly used in the literature. Indeed, for the uncontrolled problem with zero Neumann boundary data and  initial data $y_0 \in H^1(\Omega)$, the solution of \eqref{state} satisfies, for any $t \geq 0$,
\begin{equation} \label{en2} \frac{d}{dt} E(t) + \|y_t(t)\|^2_{L^{}_2(\Omega)} = 0
\end{equation}
where $E(t)$ denotes the associated Ginzburg-Landau energy i.e., for a.e. $t \in (0,T]$,
$$E(t) = \int_{\Omega} \left ( \frac{1}{2} \abs{\nabla y}^2 + \frac{1}{4\epsilon^2} (y^2-1)^2  \right ) dx.$$
As a consequence, \eqref{en2}, implies that $E(s) \leq E(0), \ \forall s \in (0,T].$
\end{remark}
We conclude our preliminary section by stating three key results that are necessary for the analysis as well as the numerical analysis of the Allen-Cahn equation. A crucial idea is the \emph{spectral estimate} of the principal eigenvalue of the linearized Allen-Cahn operator about the state solution, $y_u$,
\begin{equation}\label{spectral_classical}
-\lambda^{}(t) : = {\displaystyle\inf_{ \varv\in H^{1}(\Omega) \setminus \lbrace{0 \rbrace}}} \frac{\norm{\nabla \varv}^2_{L^{}_2(\Omega)} + \epsilon^{-2} \left(( 3y^2_{u} -1) \varv,\varv \right)}{\norm{\varv}^2_{L^{}_2(\Omega)}}.
\end{equation}
For the uncontrolled (homogenous) case,  \cite{Chen_1994,Mottoni-Schatzman_1995,Alikakos-Fusco_1993}  showed that  $\lambda \in L^{}_{\infty}(I)$ can be bounded independently of $\epsilon$ for the case of smooth, evolved interfaces. At Section 5 we provide a short proof for the nonhomogeneous case.
We state generalizations of the continuous and discrete Gronwall lemmas that are required for the upcoming analysis on Section 3 (see e.g. \cite{Bartels-Muller-Ortner_2011,Bartels_2016}).
\begin{lemma}\label{continuous_GL}
Let the nonnegative functions $w^{}_1\in C(I)$, $w{}_2, w^{}_3 \in L^{}_1(I)$, $\alpha \in L^{}_{\infty}(I)$ and the real number $A \geq 0$ satisfy, for all $t \in I$,
\begin{equation*}
    w^{}_1(t) + \int^{t}_{0} w^{}_2(s) \diff{s} \leq A + \int^{t}_{0} \alpha(s) w^{}_1(s) \diff{s} + \int^{t}_{0} w^{}_3(s) \diff{s}.
\end{equation*}
Assume also that for $B \geq 0$, $\beta >0$ and for every $t \in I$, it holds,
\begin{equation*}
    \int^{t}_{0} w^{}_3(s) \diff{s} \leq B \sup^{}_{s \in [0,t]} w^{\beta}_{1}(s) \int^{t}_{0}  \left( w^{}_1(s) + w^{}_2(s) \right) \diff{s}.
\end{equation*}
If in addition, $4A E \leq \left( 4B(T+1)E \right)^{-1/\beta},$ with $E: = \exp{\left( \int_{I} \alpha(t) \diff{t} \right)} ,$ it holds,
\begin{equation*}
    \sup^{}_{t \in I} w^{}_1(t) + \int_{I} w^{}_2(t) \diff{t} \leq 4 A E.
\end{equation*}
\end{lemma}
\begin{lemma}\label{discrete_GL}
Let $k>0$  and suppose that the nonnegative real sequences $\lbrace w^{n}_j \rbrace^{N}_{n=0}$, $j=1,2,3$, $\lbrace \alpha^{n}_{} \rbrace^{N}_{n=0}$ and the real number $A\geq 0$ satisfy
\begin{align*}
    w^{m}_1+ k \sum^{m}_{n=1}  w^{n}_2 \leq A +  k \sum^{m}_{n=1} \alpha^{n}_{}  w^{n}_1  + k \sum^{m-1}_{n=1}  w^{n}_3,
\end{align*}
for all $m=1,\ldots,N$. Let $\sup^{}_{n=1,\ldots,N} k \alpha^{n}_{} \leq 1/2$ and $Nk \leq T$. Assume also that for $B \geq  0$, $\beta >0$ and for every $m=1,\ldots,N$ we have,
\begin{align*}
     k \sum^{m-1}_{n=1}  w^{n}_3 \leq B \sup^{}_{n=1,\ldots,m-1} \left(  w^{n}_1 \right)^{\beta}_{} k \sum^{m-1}_{n=1} \left( w^{n}_1 + w^{n}_2\right).
\end{align*}
If in addition, $4AE \leq \left( 4B(T+1)E \right)^{-1/{\beta}}$ where $E:=\exp \big( 2 k \sum^{N}_{n=1} \alpha^{n}_{} \big),$ then it holds,
\begin{align*}
    \sup^{}_{n=1,\ldots,N} w^{n}_1 +  k \sum^{N}_{n=1}  w^{n}_2 \leq  4AE.
\end{align*}
\end{lemma}
\begin{remark} \label{homogeneous} In the remaining we will focus on the zero Neumann boundary data case, but all results also hold in case that zero Dirichlet boundary data.
\end{remark}

\section{Optimality conditions}

This section is devoted to the analysis of optimal control problem $\eqref{P}$. Since we are dealing with a non-convex problem, the first order necessary optimality conditions are no longer sufficient. 
\subsection{Continuity}
We begin by  studying the continuity of the relation between the control and the state. The emphasis here is on quantifying the dependence of various Lipschitz constants upon $\epsilon$.
\begin{definition} \label{controltostate}
Suppose that $y_0 \in H^1(\Omega) \cap L_{\infty}(\Omega)$. The mapping $G: L_2(I;L_2(\Omega)) \rightarrow H^{2,1}_{} \left(\Omega^{}_{T} \right)  \cap C(I;H^{1}(\Omega))$ that assigns each control function $u$ to the corresponding state $y^{}_u = y(u) = G(u)$, is called control to state operator.
\end{definition}
Throughout the analysis we  use the abbreviation: $F(y) := y^3{-}y$ and for $u^{}_i \in L_2(I;L_2(\Omega))$, $i=1,2$, we denote by $y^{}_i = G(u_{i}) : = y^{}_{u^{}_i}$.
\begin{theorem}\label{lipcon-1}
For $d=2,3$, assume that it holds,
\begin{equation}\label{cond_1}
\begin{aligned}
& \norm{u^{}_1 {-} u^{}_2}^{}_{L^{}_2(I:L^{}_2(\Omega))} \leq  C \epsilon^{d+1} \norm{y_1}^{-1}_{_{\infty}}}{ (\tilde{c} (T {+}1 ))^{-1}  E^{-(d/2)}_{d},  
 \end{aligned}
\end{equation} where  $E_{d}{:=} \exp \left ( \int_{I} 2\lambda(t) (1 {-} \epsilon^2) {+} (6{-}d) + 2 \epsilon^2 \diff{t}  \right ),$ and $C$ is an algebraic constant. Then, with $L^{}_1{:=}2E^{1/2}_d$ it holds,
\begin{equation*}\label{LC_1}
\begin{aligned}
    & \sup_{t\in I} \norm{ y^{}_1 {-} y^{}_2}^{}_{L^{}_2(\Omega)} {+} \epsilon \norm{y^{}_1 {-} y^{}_2}^{}_{L^{}_2(I;H^{1}(\Omega))}  {+} \epsilon^{-1} \norm{y^{}_1 {-} y^{}_2}^{2}_{L^{}_4(I;L^{}_4(\Omega))} \\
& \   \leq L^{}_1 \norm{u^{}_1 {-} u^{}_2}^{}_{L^{}_2(I;L^{}_2(\Omega))}.
\end{aligned}
\end{equation*}
\end{theorem}
\begin{proof}
Subtracting the equations satisfied by $y^{}_1$ and $y^{}_2$, it yields $\forall v \in H^1(\Omega)$,
\begin{equation}\label{subtract1}
\begin{aligned}
  \langle  \left( y^{}_1 {-} y^{}_2 \right)^{}_t, v \rangle  {+} \left ( \nabla \left( y^{}_1 {-} y^{}_2 \right), \nabla v \right ) {+} \epsilon^{-2} \langle F(y^{}_1) {-}  F(y^{}_2)  , v \rangle & = (u^{}_1 {-} u^{}_2,v) \\
     y^{}_1(0) {-} y^{}_2(0) & = 0.
\end{aligned}
\end{equation}
Rewriting  the nonlinear part,
\begin{equation*}
\begin{aligned}
    F(y^{}_1) {-}   F(y^{}_2) & =  y^{3}_1 {-} y^{3}_2 {-} (y^{}_1 {-} y^{}_2) {=} (y^{}_1 {-} y^{}_2)^3 {+} 3y^{}_1y^{}_2(y^{}_1 {-} y^{}_2) {-} (y^{}_1 {-} y^{}_2) \\
    & = (y^{}_1 {-} y^{}_2)^3 {+} (3y^{2}_1 {-}1)(y^{}_1 {-} y^{}_2) {-} 3y^{}_1(y^{}_1 {-} y^{}_2)^2,
\end{aligned}
\end{equation*}
 relation $\eqref{subtract1}$ becomes
\begin{equation}\label{subtract2}
\begin{aligned}
    & \langle \left( y^{}_1 {-} y^{}_2 \right)^{}_t ,v \rangle {+} \left ( \nabla \left( y^{}_1 {-} y^{}_2 \right), \nabla v \right ) + \epsilon^{-2}  \langle (y^{}_1 {-} y^{}_2)^3 {+} (3y^{2}_1 {-}1)(y^{}_1 {-} y^{}_2) , v \rangle \\
    & \ = (u^{}_1 {-} u^{}_2,v ) {+}   3 \epsilon^{-2}  \langle  y^{}_1(y^{}_1 {-} y^{}_2)^2, v \rangle .
\end{aligned}
\end{equation}
Testing with  $y^{}_1 {-} y^{}_2$, and noting that $F'(y_1) = 3y^2_1 {-}1$,
\begin{equation*}\label{testsubtract2}
\begin{aligned}
    & \frac{1}{2} \frac{\diff }{\diff t} \norm{y^{}_1 {-} y^{}_2}^{2}_{L^{}_2(\Omega)} {+} \norm{\nabla (y^{}_1 {-} y^{}_2)}^{2}_{L^{}_2(\Omega)} {+} \epsilon^{-2} \norm{y^{}_1 {-} y^{}_2}^{4}_{L^{}_4(\Omega)} \\
    & + \epsilon^{-2} \big(  F'(y^{}_1)(y^{}_1 {-} y^{}_2), y^{}_1 {-} y^{}_2 \big) =  \left(  u^{}_1 {-} u^{}_2, y^{}_1 {-} y^{}_2 \right) {+} 3 \epsilon^{-2}   \left( y^{}_1(y^{}_1 {-} y^{}_2)^2,   y^{}_1 {-} y^{}_2 \right).
\end{aligned}
\end{equation*}
Recall the spectral estimate \eqref{spectral_classical} for $\varv =  y^{}_1 {-} y^{}_2 \in H^1(\Omega)$ to deduce that
\begin{equation*}
\begin{aligned}
    &  \norm{\nabla ( y^{}_1 {-} y^{}_2)}^2_{L^{}_2(\Omega)}  + \epsilon^{-2} \big( F'(y^{}_1)   (y^{}_1 {-} y^{}_2),  y^{}_1 {-} y^{}_2 \big)   \\
    & \geq  {-} {\lambda}(t) (1 {-} \epsilon^2) \norm{  y^{}_1 {-} y^{}_2}^2_{L^{}_2(\Omega)} {+} \epsilon^2 \norm{\nabla  (y^{}_1 {-} y^{}_2)}^2_{L^{}_2(\Omega)}  + \big( F'(y_1^{})   (y^{}_1 {-} y^{}_2),  y^{}_1 {-} y^{}_2 \big).
\end{aligned}
\end{equation*}
Combining the last two relations, and using H\"older and Young's inequalities it yields,
\begin{equation*}
    \begin{aligned}
    & \frac{1}{2} \frac{\diff }{\diff t} \norm{y^{}_1 {-} y^{}_2}^{2}_{L^{}_2(\Omega)} {+} \epsilon^{2} \norm{\nabla (y^{}_1 {-} y^{}_2)}^{2}_{L^{}_2(\Omega)} {+} \epsilon^{-2} \norm{y^{}_1 {-} y^{}_2}^{4}_{L^{}_4(\Omega)} \\
    & \leq \frac{1}{2} \norm{  u^{}_1 {-} u^{}_2}^2_{L^{}_2(\Omega)}{+} \big( {\lambda}(t) (1 {-} \epsilon^2) {+} \frac{3}{2} \big)\norm{  y^{}_1 {-} y^{}_2}^2_{L^{}_2(\Omega)} {+} 3 \epsilon^{-2} \norm{  y^{}_1}^{}_{L^{}_{\infty}(\Omega)}\norm{  y^{}_1 {-} y^{}_2
    }^3_{L^{}_3(\Omega)}.
\end{aligned}
\end{equation*}
Adding $\epsilon^2 \|y_1{-}y_2\|^2_{L^2(\Omega)}$ on both sides and integrating over $t\in(0,\uptau)$, we have
\begin{equation} \label{new}
\begin{aligned}
&  \norm{(y^{}_1 {-} y^{}_2)(\uptau)}^{2}_{L^{}_2(\Omega)} {+} 2\epsilon^{2} \timeintt{\norm{ (y^{}_1 {-} y^{}_2)}^{2}_{H^1(\Omega)}}   {+} 2 \epsilon^{-2} \timeintt{  \norm{y^{}_1 {-} y^{}_2}^{4}_{L^{}_4(\Omega)}} \\
&  \leq \norm{(y^{}_1 {-} y^{}_2)(0)}^{2}_{L^{}_2(\Omega)}  {+} \timeintt{ \norm{  u^{}_1 {-} u^{}_2}^2_{L^{}_2(\Omega)} } \\
& \quad {+} \timeintt{ \left( 2{\lambda}(t) (1 {-} \epsilon^2) {+} 3 {+} 2 \epsilon^2 \right)\norm{  y^{}_1 {-} y^{}_2}^2_{L^{}_2(\Omega)} }{+} 6 \epsilon^{-2} \norm{  y^{}_1}^{}_{_{\infty}} \timeintt{  \norm{  y^{}_1 {-} y^{}_2}^3_{L^{}_3(\Omega)} }.
\end{aligned}
\end{equation}
Note that  the interpolation inequality \eqref{inter}, the embedding $H^1(\Omega) \subset L^{}_4(\Omega)$ imply
$\norm{y^{}_1 {-} y^{}_2}^3_{L^{}_3(\Omega)} \leq \frac{C} {\epsilon^2}  \norm{y^{}_1 {-} y^{}_2}^{}_{L^{}_2(\Omega)} \Big(  \epsilon^2 \norm{y^{}_1 {-} y^{}_2}^2_{H^{1}(\Omega)} \Big).$
Then, for
 \begin{align*}
    & w^{}_1(\uptau) {=}  \norm{(y^{}_1 {-} y^{}_2)(\uptau)}^{2}_{L^{}_2(\Omega)}, \ w^{}_2(t) =  2 \epsilon^2 \norm{ (y^{}_1 {-} y^{}_2)}^{2}_{H^1(\Omega)}  {+} 2 \epsilon^{-2} \norm{y^{}_1 {-} y^{}_2}^{4}_{L^{}_4(\Omega)}, \\
    & w^{}_3(t) {=}  \norm{  y^{}_1 {-} y^{}_2}^3_{L^{}_3(\Omega)}, \ A {=}  \norm{u^{}_1 {-} u^{}_2}^{2}_{L^{}_2(I;L^{}_{2}(\Omega))}, \ B{=}6 C \epsilon^{-4} \norm{  y^{}_1}^{}_{_{\infty}},\\
    & \beta {=} 1/2, \  \alpha(t) {=} 2 \lambda(t)(1 - \epsilon^2){+}3{+} 2 \epsilon^2 ,
 \end{align*}
 Lemma \ref{continuous_GL} implies the result for $d=3$.
 For $d=2$, using \eqref{inter} and the GNL inequality \eqref{GNL1}, we deduce,
 $\norm{y^{}_1 {-} y^{}_2}^3_{L^{}_3(\Omega)} \leq {\tilde{c}} \ \norm{y^{}_1 {-} y^{}_2}^{2}_{L_2(\Omega)}\norm{ y^{}_1 {-} y^{}_2}_{H^1(\Omega)} $. Therefore, substituting the above inequality into \eqref{new}, and using Young's inequality,
 \begin{equation*}
\begin{aligned}
&  \norm{(y^{}_1 {-} y^{}_2)(\uptau)}^{2}_{L^{}_2(\Omega)} + 2\epsilon^{2}  \timeintt{\norm{ y^{}_1 {-} y^{}_2}^{2}_{H^1(\Omega)}} + 2 \epsilon^{-2} \timeintt{  \norm{y^{}_1 {-} y^{}_2}^{4}_{L^{}_4(\Omega)}} \\
&  \leq \norm{(y^{}_1 {-} y^{}_2)(0)}^{2}_{L^{}_2(\Omega)}  + \timeintt{ \norm{  u^{}_1 {-}u^{}_2}^2_{L^{}_2(\Omega)} } \\
&  {+} \timeintt{ \left( 2{\lambda}(t) (1 {-} \epsilon^2) {+} 3{+} 2 \epsilon^2 \right)\norm{  y^{}_1 {-} y^{}_2}^2_{L^{}_2(\Omega)} } \\
&  {+} \frac{9 {\tilde c}^2}{\epsilon^4} \norm{  y^{}_1}^{2}_{_{\infty}} \timeintt{  \norm{  y^{}_1 {-} y^{}_2}^2_{L^{}_2(\Omega)}  \|  y_1 {-}y_2 \|^2_{H^1(\Omega)} }   {+} \timeintt{ \|y_1{-}y_2 \|^2_{L_2(\Omega)} }.
\end{aligned}
\end{equation*}
 The result now follows upon choosing
  $\beta = 1,\  B= 9 {\tilde c}^2 \epsilon^{-6} \norm{  y^{}_1}^{2}_{_{\infty}}$ and    $\alpha(t) = 2 \lambda(t)(1 - \epsilon^2) +4 + 2 \epsilon^2 $ in Lemma \ref{continuous_GL}.
 \end{proof}
\subsection{Differentiability} Next, we determine the first and second order derivatives of the control to state mapping G of Definition \ref{controltostate}.
\begin{theorem}
Let $u,\varv \in L_2(I;L_2(\Omega))$. The mapping $G:L_2(I;L_2(\Omega)) \rightarrow H^{2,1}_{}(\Omega^{}_{T})  \cap C(I;H^{1}(\Omega))$, such that $y^{}_u := G(u)$, is of class $C^{\infty}_{}$. We denote by $z^{}_{\varv}= G'(u)\varv$ and $z^{}_{\varv\varv} = G{''}(u){\varv}^2_{}$,  the unique solutions to the following equations
\begin{equation}\label{1st-deriv}
\begin{aligned}
    & z^{}_{\varv,t} {-} \Delta z^{}_{\varv} {+} \epsilon^{-2} \left( 3 y^{2}_{u} {-}1 \right) z^{}_{\varv}  {=} \varv \quad \text{in} \ \Omega^{}_{T}, \\
    & \frac{ \partial z_{\varv}}{\partial n} {=} 0 \quad \text{on} \ \Sigma^{}_T, \ z^{}_{\varv}(0)  {=} 0 \quad  \text{in} \ \Omega,
\end{aligned}
\end{equation}
\begin{equation}\label{2nd-deriv}
\begin{aligned}
    & z^{}_{\varv\varv,t} {-} \Delta z^{}_{\varv\varv} {+} \epsilon^{-2} \left( 3 y^{2}_{u} {-}1 \right) z^{}_{\varv\varv} {=} {-} 6 \epsilon^{-2}  y^{}_{u} z^{2}_{\varv} \quad \text{in} \ \Omega^{}_{T}, \\
    &  \frac{\partial z_{\varv \varv} }{\partial n} {=} 0 \quad \text{on} \ \Sigma^{}_T, \ z^{}_{\varv \varv }(0)  {=} 0 \quad \text{in} \ \Omega.
    \end{aligned}
\end{equation}
\end{theorem}
\begin{proof}
 The proof follows using similar arguments as in \cite{Casas-Chrysafinos_2012,Casas-Chrysafinos_2016}.
\end{proof}

The above differentiability properties of G imply that the reduced cost functional $J: L_2(I;L_2(\Omega)) \rightarrow \mathbb{R}$ is of class $C^{\infty}_{}$, as well.

\begin{lemma}\label{1-2_deriv_J}
For any $u,v \in L_2(I;L_2(\Omega))$, it holds that
\begin{equation}\label{1-deriv_J}
    J'(u)\varv {=} \int_{I} \spaceint{( \varphi^{}_u {+} \mu u )
    \varv} \diff{t},
\end{equation}
\begin{equation}\label{2-deriv_J}
\begin{aligned}
    J''(u)\varv^{2}_{} {=} & \int_{I} \spaceint{\abs{z^{}_{\varv}}^2} \diff{t} {+} \gamma \spaceint{\abs{z^{}_{\varv}(T)}^2} {+} \mu \int_{I} \spaceint{\abs{\varv}^2}  \\
&  {-} 6 \epsilon^{-2} \int_{I} \spaceint{y^{}_{u} \ z^{2}_{\varv} \ \varphi^{}_{u}} \diff{t},
\end{aligned}
\end{equation}
where $z^{}_{\varv}$ is the solution of \eqref{1st-deriv} and $\varphi^{}_{u} \in H^{2,1}_{}(\Omega^{}_{T}) \cap  C(I;H^{1}(\Omega))$, with $\frac{\partial \varphi_{u}}{\partial n} = 0 \ \mbox{on $\Sigma_T$}$,  is the unique solution of the adjoint sate equation, i.e., $\forall w \in H^1(\Omega)$, $a.e. t \in I$
\begin{equation}\label{adjoint}
   \begin{cases}
   {-} \left(\varphi^{}_{u,t},w \right) {+} \left( \nabla \varphi^{}_{u}, \nabla w \right) {+} \epsilon^{-2} \left( \left( 3y^{2}_{u} {-} 1 \right) \varphi^{}_{u}, w \right ) {=} \left( y^{}_{u} {-} y^{}_{d},w \right)\\
   (\varphi^{}_{u}(T),w) {=}  ( \gamma(y^{}_{u}(T) - y^{}_{\Omega}),w).
   \end{cases}
\end{equation}
\end{lemma}

\begin{remark} \label{terminal}
The assumption that $y^{}_{\Omega} \in H^{1}(\Omega)$ together with the regularity properties of the state solution $y^{}_u$ implies that $\varphi^{}_u(T) \in H^{1}(\Omega)$. Thus, we  deduce that
\begin{align*}
    \norm{\varphi^{}_u(T)}^{}_{H^{1}(\Omega)} =  \norm{\gamma(y^{}_{u}(T) {-} y^{}_{\Omega})}^{}_{H^{1}(\Omega)} \leq \gamma \Big( \sup^{}_{t \in I} \norm{y^{}_{u}(t)}^{}_{H^{1}(\Omega)} {+} \norm{y^{}_{\Omega}}^{}_{H^{1}(\Omega)} \Big).
\end{align*}
From Lemma \ref{stability}, we obtain $\norm{\varphi^{}_u(T)}^{}_{H^{1}(\Omega)} \leq \tilde C \epsilon^{-1}$ when $ \epsilon \| \nabla y_0 \|^{}_{L^{}_2(\Omega)} {+}\frac{1}{2} \| (y_0^2-1)^2 \|^{1/2}_{L^{}_1(\Omega)} \leq C $. Similarly, we have that $\norm{\varphi^{}_u(T)}^{}_{H^{1}(\Omega)} \leq \tilde C$ when $ \| \nabla y_0 \|^{}_{L^{}_2(\Omega)} {+} \frac{1}{2\epsilon} \| (y_0^2{-}1)^2 \|^{1/2}_{L^{}_1(\Omega)} \leq D.$
\end{remark}

\begin{lemma}\label{adjoint_stability}
Let $\varphi^{}_u$ be the solution to \eqref{adjoint},  $y^{}_{d} \in L^{}_{2}(I;L^{}_{2}(\Omega))$ and $\varphi^{}_{u}(T) \in H^{1}(\Omega)$. Then, there exists a constant $C>0$, depending on $\Omega$ such that
\begin{equation}\label{adj_stab_1}
\begin{aligned}
    & \sup^{}_{t \in I} \norm{\varphi^{}_{u}(t)}^{}_{L^{}_2(\Omega)} {+} \  \epsilon \norm{ \varphi^{}_{u}}^{}_{L^{}_{2}(I;H^{1}(\Omega))} {+}  \| \varphi_u y_u \|_{L_2(I;L_2(\Omega))} \\
    & \ \leq CC^{1/2}_{\varphi}\left( \norm{ \varphi^{}_{u}(T)}^{}_{L^{}_{2}(\Omega)} {+} \norm{y^{}_{u} {-} y^{}_{d}}^{}_{L^{}_{2}(I;L^{}_{2}(\Omega))} \right) : =  D^{}_{\text{st},1} ,
\end{aligned}
\end{equation}
\begin{equation}\label{adj_stab_2}
\begin{aligned}
   & \norm{  \varphi^{}_{u,t}}^{}_{L^{}_{2}(I;L^{}_{2}(\Omega))} {+} \sup^{}_{t \in I} \norm{\nabla \varphi^{}_{u}(t)}^{}_{L^{}_2(\Omega)}      \leq   C \Big(  \norm{y^{}_{u} {-} y^{}_{d}}^{}_{L^{}_{2}(I;L^{}_{2}(\Omega))} \\
    &  +  \norm{ \nabla \varphi^{}_{u}(T)}^{}_{L^{}_2(\Omega)} {+}  \epsilon^{-2} \left(  \norm{ y^{}_{u}}_{_{\infty}} {+} T^{1/2} \right) D^{}_{\text{st},1} \Big)  {:=} D^{}_{\text{st},2},
\end{aligned}
\end{equation}
\begin{equation}\label{adj_stab_3}
\begin{aligned}
     & \norm{ \varphi^{}_{u}}^{}_{L^{}_{2}(I;H^{2}_{}(\Omega))}  \leq C \Big  ( \norm{y^{}_{u} {-} y^{}_{d}}^{}_{L^{}_{2}(I;L^{}_{2}(\Omega))}  {+} D^{}_{\text{st},2}    {+}    \epsilon^{-2}\left(1 {+}  \norm{ y^{}_{u}}_{_{\infty}}\right)  D^{}_{\text{st},1}  \Big ) ,
\end{aligned}
\end{equation}
where we denote by $C^{}_{\varphi}:= \exp \left( \int_{I}(2 \lambda(t)(1{-}\epsilon^2) +3 + 2 \epsilon^2) \diff{t} \right)$.
\end{lemma}
\proof
To derive the first stability estimate, we test \eqref{adjoint} with $w = \varphi^{}_u$.
\begin{align*}
    - \frac{1}{2} \frac{\diff}{\diff{t}} \norm{ \varphi^{}_u}^2_{L^{}_2(\Omega)} {+} \norm{ \nabla  \varphi^{}_u}^2_{L^{}_2(\Omega)} {+} \epsilon^{-2} ( F'(y^{}_u) \varphi^{}_u, \varphi^{}_u) {=} (y^{}_u {-} y^{}_d, \varphi^{}_u).
\end{align*}
Recalling \eqref{spectral_classical} for $\varv = \varphi^{}_u$ about state solution $y^{}_u$, adding $\epsilon^2 \| \varphi \|^2_{L_2(\Omega)}$ on both sides applying Cauchy-Schwarz and Young's inequalities on the right-hand side and integrating with respect to $t\in(\uptau,T)$ we obtain
\begin{align*}
     &\frac{1}{2}  \norm{ \varphi^{}_u(\uptau)}^2_{L^{}_2(\Omega)} {+}  \epsilon^2 \int^{T}_{\uptau} \norm{  \varphi^{}_u}^2_{H^1(\Omega)} \diff{t}  {+} 3   \int^{T}_{\uptau} \norm{ \varphi^{}_u y^{}_u}^2_{L^{}_2(\Omega)} \diff{t}\\
     & \leq  \frac{1}{2}  \norm{ \varphi^{}_u(T)}^2_{L^{}_2(\Omega)} {+} \int^{T}_{\uptau} \Big( \lambda(t)(1{-}\epsilon^2 ) {+} \frac{3}{2} + \epsilon^2 \Big)  \norm{ \varphi^{}_u}^2_{L^{}_2(\Omega)} \diff{t} {+} \frac{1}{2} \int^{T}_{\uptau}  \norm{y^{}_u {-} y^{}_d}^2_{L^{}_2(\Omega)} \diff{t}.
\end{align*}
The (linear) Gronwall inequality yields the result. Setting $w{=}{-}\varphi^{}_{u,t}$ into \eqref{adjoint} we have
\begin{align} \label{newphi}
      \norm{ \varphi^{}_{u,t}}^2_{L^{}_2(\Omega)} {-} \frac{1}{2} \frac{\diff}{\diff{t}} \norm{ \nabla \varphi^{}_u}^2_{L^{}_2(\Omega)} {=} (y^{}_u {-} y^{}_d, \varphi^{}_{u,t}) {+} \epsilon^{-2} ( F'(y^{}_u) \varphi^{}_u, \varphi^{}_{u,t}).
\end{align}
Note that using Young's inequalities, we may bound the last two terms as follows:
\begin{align*}
& \abs{ (y^{}_u {-} y^{}_d, \varphi^{}_{u,t}) }  \leq (1/4) \|\varphi_{u,t}\|^2_{L^{}_2(\Omega)}  {+} \|y_u{-}y_d\|^2_{L_2(\Omega)}, \\
& \epsilon^{-2}  \abs{ ( F'(y^{}_u) \varphi^{}_u, \varphi^{}_{u,t}) }   = \epsilon^{-2} \abs{ \left ( (3y^2_u {-}1) \varphi_u , \varphi_{u.t} \right ) } \\
 & \  \leq (1/4) \|\varphi_{u,t}\|^2_{L^{}_2(\Omega)} {+}  18 \epsilon^{-4} \|y_u\|^2_{L_{\infty}(\Omega)} \| \varphi_u y_u \|^2_{L_2(\Omega)} {+} 2 \epsilon^{-4} \| \phi_u \|^2_{L_2(\Omega)}.
\end{align*}
Substituting the last two inequalities into \eqref{newphi} and integrating from $\uptau$ to $T$ it yields,
\begin{align*}
    & \int^{T}_{\uptau} \norm{ \varphi^{}_{u,t}}^2_{L^{}_2(\Omega)} \diff{t} {+} \norm{ \nabla \varphi^{}_u(\uptau)}^2_{L^{}_2(\Omega)}   \leq  2 \int^{T}_{\uptau} \norm{ y^{}_u {-} y^{}_d}^2_{L^{}_2(\Omega)} \diff{t}
    {+}  \norm{ \nabla \varphi^{}_u(T)}^2_{L^{}_2(\Omega)} \\
    & {+}  36 \epsilon^{-4} \norm{y^{}_u}^2_{_{\infty}} \int^{T}_{\uptau} \norm{  \varphi^{}_u y_u}^2_{L^{}_2(\Omega)} \diff{t} {+} 4 \epsilon^{-4} \int^T_{\uptau} \|\varphi_{u}\|^2_{L^2(\Omega)} \diff{t}.
\end{align*}
Using the stability bound \eqref{adj_stab_1} we obtain \eqref{adj_stab_2}.
The third estimate follows using similar techniques by setting $w= - \Delta \varphi^{}_u$ into \eqref{adjoint} and using the previous bounds to estimate $\|  \varphi^{}_{u,t} \|_{L_2(I;L_2(\Omega))}, \norm{\nabla\varphi^{}_{u} }_{L_2(I;L_2(\Omega))} \ \text{and} \ \norm{y^{}_{u} \varphi^{}_{u}}_{L_2(I;L_2(\Omega))} $.
\endproof
Let $u^{}_1, u^{}_2 \in L_2(I;L_2(\Omega))$ be the control functions. Then, we denote by $y^{}_i = y^{}_{u^{}_{i}}$ and $\varphi^{}_{i} = \varphi^{}_{u^{}_{i}}$ the associated state and adjoint state solutions for $i=1,2$, respectively.
\begin{lemma}\label{lipcon-5}
Assume that \eqref{cond_1} holds. Then, for $d=2$, there exists a constant $C_T$ depending only on the domain $\Omega_T$ such that,
\begin{equation}\label{LC_3}
\begin{aligned}
 &    \sup^{}_{t \in I} \norm{\varphi^{}_1 {-} \varphi^{}_2}^{}_{L^{}_2(\Omega)} +   \epsilon \norm{ \varphi^{}_1 {-} \varphi^{}_2 }^{}_{L^{}_2(I;H^{1}(\Omega))}  \\
& \  \leq  C_T E^{1/2}_{\varphi} L^{}_1\left ( 1 {+} C_{\infty} {\tilde c} D_{\text{st},1} \epsilon^{-7/2} \right )   \norm{ u^{}_1 {-} u^{}_2 }^{}_{L_2(I;L_2(\Omega))}.
\end{aligned}
\end{equation}
For $d=3$, there exists a constant $C_T$ depending only on the domain $\Omega_T$ such that,
\begin{equation}\label{LC_2}
\begin{aligned}
&     \sup^{}_{t \in I} \norm{\varphi^{}_1 {-} \varphi^{}_2}^{}_{L^{}_2(\Omega)} +  \epsilon \norm{ \varphi^{}_1 {-} \varphi^{}_2 }^{}_{L^{}_2(I;H^{1}(\Omega))}  \\
 & \ \leq  C_T E^{1/2}_{\varphi} L^{}_1 \left ( 1 {+} C_{\infty} {\tilde c} D_{\text{st},1} \epsilon^{-15/4} \right )   \norm{ u^{}_1 {-} u^{}_2 }^{}_{L_2(I;L_2(\Omega))}.
\end{aligned}
\end{equation}
Here, we denote by $ C^{}_{\infty} {:=} \left( C \big( \norm{y^{}_1}^{2}_{_{\infty}} {+} \norm{y^{}_2}^{2}_{_{\infty}} \big)\right)^{1/2}_{}$ with $C$ depending on $\abs{ \Omega }$, and by
$ E_\varphi {: =} \int_{I} \left ( 2 \lambda (t) (1{-}\epsilon^2) {+} 4 + 2 \epsilon^2 \right ) \diff{t}.$
\end{lemma}
\begin{proof}
Subtracting the equations satisfied by $\varphi^{}_1$ and $\varphi^{}_2$, it yields for $w \in H^{1}(\Omega)$ 
\begin{equation*} \label{subtract_3}
\begin{cases}
   - \langle \varphi^{}_{1,t} {-} \varphi^{}_{2,t} , w \rangle {+} \left( \nabla ( \varphi^{}_1 {-} \varphi^{}_2 ) , \nabla w  \right) {+} \epsilon^{-2} \langle  F'(y^{}_1) \varphi^{}_1 {-}   F'(y^{}_2) \varphi^{}_2 , w \rangle  {=} (y^{}_1 {-} y^{}_2,w) \\
    \left( \varphi^{}_1 {-} \varphi^{}_1 \right)(T) {=} \gamma \left(y^{}_1 {-} y^{}_2 \right)(T).
\end{cases}
\end{equation*}
Inserting the identity,
$ F'(y^{}_1) \varphi^{}_1 {-}   F'(y^{}_2) \varphi^{}_2 {=} \left( 3y^{2}_1 {-} 1 \right) \left( \varphi^{}_1 {-} \varphi^{}_2 \right) {+} 3 \left( y^{2}_1 {-} y^{2}_2 \right) \varphi^{}_2$
 and testing with $w=\varphi^{}_1 {-} \varphi^{}_2$, we deduce that
 \begin{equation*}
    \begin{aligned}
        -& \frac{1}{2}  \frac{\diff }{\diff t} \norm{\varphi^{}_1 {-} \varphi^{}_2}^{2}_{L^{}_2(\Omega)} {+} \norm{\nabla (\varphi^{}_1 {-} \varphi^{}_2)}^{2}_{L^{}_2(\Omega)} {+} \epsilon^{-2} \left( F'(y^{}_{1})(\varphi^{}_1 {-} \varphi^{}_2), \varphi^{}_1 {-} \varphi^{}_2 \right)\\
        & {=} \left( y^{}_1 {-} y^{}_2, \varphi^{}_1 {-} \varphi^{}_2 \right) {-} 3 \epsilon^{-2} \left( ( y^{2}_1 {-} y^{2}_2 ) \varphi^{}_2,\varphi^{}_1 {-} \varphi^{}_2  \right) {:=} \mathcal{K}^{}_1 {+} \mathcal{K}^{}_2.
    \end{aligned}
 \end{equation*}
Cauchy-Schwarz and Young's inequalities yield,
 $\mathcal{K}^{}_1 {\leq} \frac{1}{2} \norm{\varphi^{}_1 {-} \varphi^{}_2}^{2}_{L^{}_2(\Omega)} {+} \frac{1}{2} \norm{y^{}_1 {-} y^{}_2}^{2}_{L^{}_2(\Omega)}$.
Applying  H\"{o}lder inequality, the embedding $H^1(\Omega) \subset L_4(\Omega)$ and Young's inequality on the second term, we obtain
\begin{equation*}
\begin{aligned}
     \mathcal{K}^{}_2
     & \leq 3 \epsilon^{-2} \norm{y^{}_1 {-} y^{}_2}^{}_{L^{}_2(\Omega)} \left( \norm{y^{}_1}^{}_{L^{}_{\infty}(\Omega)} + \norm{y^{}_2}^{}_{L^{}_{\infty}(\Omega)} \right) \norm{\varphi^{}_2}^{}_{L^{}_{4}(\Omega)}\norm{\varphi^{}_1 {-} \varphi^{}_2}^{}_{L^{}_4(\Omega)}\\
    &\leq  C^{2}_{\infty} \epsilon^{-6}\norm{y^{}_1 {-} y^{}_2}^{2}_{L^{}_2(\Omega)}  \norm{\varphi^{}_2}^{2}_{L^{}_{4}(\Omega)} {+} (1/2) \epsilon^2  \norm{ \varphi^{}_1 {-} \varphi^{}_2}^{2}_{H^{1}(\Omega)}.
\end{aligned}
\end{equation*}
Substituting the bounds on $\mathcal{K}_i$, the spectral estimate \eqref{spectral_classical} for $\varv = \varphi^{}_1 - \varphi^{}_2$, $y=y^{}_1$, adding $\epsilon^2 \| \varphi_{1}-\varphi_{2} \|^2_{L_2(\Omega)}$ and integrating from $\uptau$ to $T$, \eqref{subtract_3} yields,
\begin{equation*}
\begin{aligned}
   & \norm{(\varphi^{}_1 {-} \varphi^{}_2)(\uptau)}^{2}_{L^{}_2(\Omega)} {+}\epsilon^2  \int^{T}_{\uptau} \norm{ \varphi^{}_1 {-} \varphi^{}_2 }^{2}_{H^1(\Omega)} \diff{t} \\
   & \leq \norm{\gamma(y^{}_1 {-} y^{}_2)(T)}^{2}_{L^{}_2(\Omega)} {+} \int_{\uptau}^T {\norm{y^{}_1 {-} y^{}_2}^{2}_{L^{}_2(\Omega)}} \diff{t} \\
 & \  {+} \int_{\uptau}^T \left(2 \lambda(t)(1{-} \epsilon^2) {+} 4 {+} 2 \epsilon^2 \right )  \norm{\varphi^{}_1 {-} \varphi^{}_2}^{2}_{L^{}_2(\Omega)} \diff{t} \\
& \ {+} C^{2}_{\infty}\epsilon^{-6} 
    \norm{y^{}_1 {-} y^{}_2}^{2}_{L_{\infty}(I;L_2(\Omega))}  \int_{\uptau}^T
    \norm{  \varphi^{}_2}^{2}_{L^{}_{4}(\Omega)} \diff{t}.
\end{aligned}
\end{equation*}
For  $d=3$, \eqref{GNL3} and H\"older's inequalities with $s_1=4$ and $s_2 = 4/3$, and \eqref{adj_stab_1} imply,
\begin{align*}
& \int_{\uptau}^T     \norm{  \varphi^{}_2}^{2}_{L^{}_{4}(\Omega)} \diff{t} \leq {\tilde c}^2 \int_{\uptau}^T \| \varphi^{}_2 \|^{1/2}_{L^{}_2(\Omega)}  \| \varphi^{}_2 \|^{3/2}_{H^1(\Omega)}  \diff{t} \\
 &   \leq {\tilde c}^2 \| \varphi^{}_2 \|^{1/2}_{L_2(I;L_2(\Omega))}   \|\varphi^{}_2 \|^{3/2}_{L^{}_2(I;H^1(\Omega))}  \leq T^{1/4} {\tilde c}^2 D^{2}_{st,1} \epsilon^{-3/2}.
 \end{align*}
 Similarly, for $d=2$, we have
\begin{align*}
 & \int_{\uptau}^T     \norm{  \varphi^{}_2}^{2}_{L^{}_{4}(\Omega)} \diff{t}  \leq {\tilde c}^2 \| \varphi^{}_2 \|^{}_{L_2(I;L_2(\Omega))} \| \varphi^{}_2 \|^{}_{L^{}_2(I;H^1(\Omega))}  \leq  T^{1/2} {\tilde c}^2 D^{2}_{st,1}\epsilon^{-1}.
 \end{align*}
Then, for $d=3$, the (linear) Gronwall inequality implies that
\begin{align*}
    & \sup^{}_{t \in[0,T]} \norm{(\varphi^{}_1 {-} \varphi^{}_2)(t)}^{2}_{L^{}_2(\Omega)} {+}   \epsilon^2 \norm{ \varphi^{}_1 {-} \varphi^{}_2 }^{2}_{L^{}_2(I;H^1(\Omega))}  \\
    & \leq E^{}_{\varphi} \Big(  \norm{\gamma(y^{}_1 {-} y^{}_2)(T)}^{2}_{L^{}_2(\Omega)} {+}  \Big( T {+} T^{1/4} C^2_{\infty} D^2_{\text{st},1} \epsilon^{-15/2} \Big) \sup^{}_{t\in I} \norm{y^{}_1 {-} y^{}_2}^{2}_{L^{}_2(\Omega)}  \Big).
\end{align*}
The estimate now follows by using the estimate Lipschitz continuity \eqref{LC_1}.
Working in an identical way, we deduce the estimate for $d=2$
\end{proof}


\subsection{Necessary and sufficient conditions} Below, we state the  optimality conditions. We refer the reader to \cite[ Theorems 3.4 and 3.3]{Casas-Chrysafinos_2012} for the related  proofs.

\begin{theorem}\label{1st-necessary}
Every locally optimal control $\bar{u}$ for problem \eqref{P}, satisfies, together with its associated state $\bar y \in H^{2,1}(\Omega_T)$ and adjoint state $\bar{\varphi} \in H^{2,1}(\Omega_T)$
\begin{equation}\label{forward}
\begin{cases}
    \langle {\bar y}_t, \varv \rangle {+} \left( \nabla {\bar y}, \nabla \varv \right) {+} \epsilon^{-2} \langle {\bar y}^3 {-} {\bar y},\varv \rangle  {=} \left( {\bar u}, \varv \right) \quad \forall \varv \in H^{1}(\Omega)\\
    {\bar y}(0) {=}  y_0^{},
\end{cases}
\end{equation}
\begin{equation}\label{backwards}
   \begin{cases}
   -\langle {\bar{\varphi}}^{}_{t},w \rangle {+} \left( \nabla {\bar {\varphi}}, \nabla w \right) {+} \epsilon^{-2} \big \langle \left( 3{\bar y}^{2} {-}1 \right) {\bar {\varphi}}^{}, w \big \rangle {=} \left( {\bar y} {-} y^{}_{d},w \right) \ \forall w \in H^{1}(\Omega)\\
   {\bar \varphi}(T) {=} \gamma( {\bar y}(T) {-} y^{}_{\Omega}),
   \end{cases}
\end{equation}
and the variational inequality (optimality condition)
\begin{equation}\label{1_variational_ineq}
\int_{I} \spaceint{\left( \bar{\varphi} {+} \mu \bar{u} \right) \left( u {-} \bar{u} \right) } \diff{t} \geq 0 \quad \forall u \in U^{}_{ad},
\end{equation}
where $\bar u \in L^{}_2(I;W^{1,p}(\Omega)) \cap C(I;H^1_{}(\Omega)) \cap H^1_{}(\Omega^{}_T)$, for any $1 \leq p < \infty.$\\
Furthermore, let $\epsilon \|\nabla y_0 \|^{}_{L^{}_2(\Omega)} {+} \frac{1}{2} \|(y^2_0{-}1)^2 \|^{1/2}_{L^{}_1(\Omega)} \leq C$. Then ,
\begin{align*} \|\bar y \|^{}_{H^{2,1}(\Omega^{}_T)} \leq \tilde{C} \epsilon^{-2}, \ \text{and} \ \
\| \bar \varphi \|^{}_{H^{2,1}(\Omega^{}_T)} \leq \tilde{D}^{}_{} \|\bar y\|_{_{\infty}} \epsilon^{-2}.
\end{align*}
If in addition, $\|\nabla y_0 \|^{}_{L^{}_2(\Omega)}{+} \frac{1}{2\epsilon} \|(y^2_0{-}1)^2 \|^{1/2}_{L^{}_1(\Omega)} \leq D$ then,
\begin{align*}
\|\bar y \|^{}_{H^{2,1}(\Omega^{}_T)} \leq {\tilde{C}^{}_{}} \epsilon^{-1}, \ \mbox{and} \  \| \bar \varphi \|^{}_{H^{2,1}(\Omega^{}_T)} \leq \tilde{D}^{}_{} \| \bar y\|_{_{\infty}} \epsilon^{-2}.
\end{align*}
Here, the constants $\tilde{C}^{}_{}$, $\tilde{D}^{}_{}$ are independent of $\epsilon$ and depend only on data.
\end{theorem}
\begin{proof} Note that every local optimal solution satisfies $J'(\bar u) (u- \bar u) \geq 0,$ $\forall u \in U_{ad}$. The optimality system \eqref{forward}, \eqref{backwards} and \eqref{1_variational_ineq} follows from \eqref{1-deriv_J}. Inequality \eqref{1_variational_ineq} implies the standard projection  formula
\begin{equation} \label{projectionformula}
    \bar{u}(t,x) = \mathrm{Proj}^{}_{[u^{}_{a}, u^{}_{b}]} \Big( - \frac{1}{\mu} \bar{\varphi}(t,x) \Big) \ \text{for a.e.} \ (t,x) \in \Omega^{}_{T},
\end{equation}
from which we deduce $\bar{u} \in H^{1}_{}({\Omega^{}_T}) \cap  C(I;H^{1}_{}(\Omega)) \cap L^{}_2(I;W^{1,p}_{}(\Omega)) $,  for all $1 \leq p < \infty$.
Therefore, if $\epsilon \| \nabla y_0 \|^{}_{L^{}_2(\Omega)} + \frac{1}{2}\| (y^2_0-1)^2 \|^{1/2}_{L^1(\Omega)} \leq C$, (independent of $\epsilon$)  we observe from Lemma \ref{stability} and Remark \ref{first} that $\| \bar y \|^{}_{H^{2,1}(\Omega^{}_T)} {\leq} \tilde{C} \epsilon^{-2}$.  In addition, we have that $\|\bar \phi (T)\|_{L_2(\Omega)} = \gamma \|\bar y(T)-y_\Omega \|_{L_2(\Omega)}$ is bounded independent of $\epsilon$ which implies that the constant $D_{st,1}$ of Lemma \ref{adjoint_stability} is bounded independent of $\epsilon$. Hence, using Remark \ref{terminal} and the estimates of Lemma \ref{adjoint_stability} we obtain $\|\bar \varphi\|^{}_{H^{2,1}(\Omega^{}_T)} \leq  \tilde{D} \| \bar y\|_{_{\infty}} \epsilon^{-2}$. Under the assumption $\| \nabla y_0 \|^{}_{L^{}_2(\Omega)} + \frac{1}{2\epsilon} \|(y^2_0-1)^2 \|^{1/2}_{L^{}_1(\Omega)} \leq D$, a similar boot-strap argument,
 imply that $\| \bar y \|^{}_{H^{2,1}(\Omega^{}_T)} \leq \tilde{C}\epsilon^{-1}$  and $\|\bar \phi\|^{}_{H^{2,1}(\Omega^{}_T)} \leq  \tilde{D}\| \bar y\|_{_{\infty}} \epsilon^{-2}$.
 \end{proof}

In the usual manner, we deduce from \eqref{1_variational_ineq} that for a.e. $(t,x)\in \Omega^{}_{T}$,
\begin{alignat*}{2}
    & \begin{aligned}
    & \begin{cases}
    \bar{u}(t,x) = u^{}_{a} \ \Rightarrow  \bar{\varphi}(t,x) + \mu \bar{u}(t,x) \geq 0,\\
    \bar{u}(t,x) = u^{}_{b} \ \Rightarrow \bar{\varphi}(t,x) + \mu \bar{u}(t,x) \leq 0,\\
    \bar{u}(t,x) \in ( u^{}_{a}, u^{}_{b} )\ \Rightarrow  \bar{\varphi}(t,x) + \mu \bar{u}(t,x) = 0,\\
  \end{cases}\\
  \end{aligned}
    & \hskip -0.9em
  &\begin{aligned}
  & \begin{cases}
  \bar{\varphi}(t,x) + \mu \bar{u}(t,x) > 0  \Rightarrow  \bar{u}(t,x) = u^{}_{a},\\
  \bar{\varphi}(t,x) + \mu \bar{u}(t,x) < 0   \Rightarrow  \bar{u}(t,x) = u^{}_{b}.\\
  \end{cases} \\
  \end{aligned}
\end{alignat*}

We introduce the cone of critical directions that is necessary  to state the second order conditions.
\begin{equation}
    \mathcal{C}^{}_{\bar{u}} = \big \lbrace \varv \in L^{}_{2}(I;L^{}_{2}(\Omega)) : \varv \ \text{satisfies} \ \eqref{cone-conditions} \big \rbrace,
\end{equation}
\begin{equation}\label{cone-conditions}
    \begin{cases}
     \varv(t,x) = 0 \quad \  \text{if} \ \  \bar{\varphi}(t,x) + \mu \bar{u}(t,x) \neq 0\\
     \varv(t,x) \geq 0 \quad \  \text{if} \ \  \bar{u}(t,x) = u^{}_{a} \\
     \varv(t,x) \leq 0 \quad \  \text{if} \ \  \bar{u}(t,x) = u^{}_{b}.
    \end{cases}
\end{equation}
Let us notice that
\begin{equation}
\begin{aligned}
    & J'(\bar{u}) \varv = \int_{I} \spaceint{\left( \bar{\varphi}(t,x) + \mu \bar{u}(t,x) \right) \varv } \diff{t}, \\
    & \left( \bar{\varphi}(t,x) + \mu \bar{u}(t,x) \right) \varv(t,x) = 0 \ \text{for a.e.} \ (t,x) \in \Omega^{}_{T} \ \text{and} \ \forall \varv \in \mathcal{C}^{}_{\bar{u}}.
\end{aligned}
\end{equation}
\begin{theorem}\label{2nd-necessary}
Let $\bar{u}$ be a local solution of problem \eqref{P}. Then, it holds that $J''(\bar{u}) \varv^2 \geq 0$,  $\forall \varv \in \mathcal{C}^{}_{\bar{u}}$ . Conversely, if $\bar{u} \in U^{}_{ad}$ satisfies
\begin{align}
    J'(\bar{u})(u - \bar{u}) & \geq 0 \ \  \forall u \in U^{}_{ad},\\
    J''(\bar{u}) \varv^2  & > 0 \ \ \forall \varv \in C^{}_{\bar{u}}\setminus \lbrace 0 \rbrace, \label{2nd-sufficient}
\end{align}
then there exist $\alpha > 0$ and $\delta>0$  such that
\begin{equation}
     J(\bar{u}) + \frac{\delta}{2} \norm{u - \bar{u}}^{2}_{L_2(I;L_2(\Omega))} \leq J(u) \ \   \forall u \in U^{}_{ad} \cap B^{}_{\alpha}(\bar{u}),
\end{equation}
 where $ B^{}_{\alpha}(\bar{u})$ is the open ball of $L^{}_{2}(I;L^{}_2(\Omega))$ centered at $\bar{u}$ with radius $\alpha$.
\end{theorem}
\proof The proof follows identically to \cite{Casas-Chrysafinos_2012,Casas-Chrysafinos_2016} based on arguments of \cite{Casas-Mateos-Raymond_2007} (see also references within). We note that $\|y^{}_u\|^{}_{L_2(I;L_2(\Omega))}$ is bounded independent of $\epsilon$ for any $u \in {U}_{ad}$ (see Lemma \ref{stability}). In the remaining, we point out that the constant $\delta >0$ will not appear at any exponential.
\endproof
\begin{remark}
The second order sufficient condition  \eqref{2nd-sufficient} is equivalent to
\begin{equation}\label{equiv_representation}
     J''(\bar{u}) \varv^2 \geq \delta \norm{ \varv}^{2}_{L_2(I;L_2(\Omega))} \ \ \forall \varv \in \mathcal{C}^{}_{\bar{u}}.
\end{equation}
\end{remark}

\section{Approximation of the control problem}
Let $\lbrace \mathcal{T}^{}_{h} \rbrace^{}_{h>0}$ be a family of triangulations of $\Bar{\Omega}$. We consider each $\mathcal{T}^{}_h$ to be a conforming and quasi-uniform subdivision  such that $\cup_{\tau \in \mathcal{T}^{}_h} \tau = \Bar{\Omega}$. With each element $\tau \in \mathcal{T}^{}_h$ we associate two parameters $h^{}_{\tau}$ and $\rho^{}_{\tau}$, where $h^{}_{\tau}$ is the diameter of $\tau$ while $\rho^{}_{\tau}$ is the diameter of the largest ball contained in $\tau$. Then, we define the meshsize parameter as $h:= \max^{}_{\tau \in \mathcal{T}^{}_h} h^{}_{\tau}$. To each $\mathcal{T}^{}_h$ we associate the finite element space:
$$Y^{}_h := \big \lbrace y^{}_h \in C(\Bar{\Omega}) ; \ y^{}_h \vert^{}_{\tau} \in \mathbb{P}^{}_{1}(\tau) , \ \forall \tau \in \mathcal{T}^{}_h \big \rbrace \subset H^{1}(\Omega),$$
with $\mathbb{P}^{}_{1}$ denoting the $d$-variate space of linear polynomials.
We assume the following classical inverse estimates:
\begin{equation} \label{inv}
\norm{\varv^{}_h}^{}_{L^{}_3(\Omega)} \leq C^{}_{\text{inv}} h_{}^{-(d/6)} \norm{\varv^{}_h}^{}_{L^{}_2(\Omega)}, \ \text{and} \ \norm{\varv^{}_h}^{}_{H^1(\Omega)} \leq C^{}_{\text{inv}} h_{}^{-1} \norm{\varv^{}_h}^{}_{L^{}_2(\Omega)}.
\end{equation}
Furthermore,  we set
$$ U^{}_{h} = \big \lbrace u^{}_{h} \in L_2(I;L_2(\Omega)) ; \ u^{}_{h} \vert^{}_{\tau} \equiv u^{}_{\tau} \in \mathbb{R} \big \rbrace. $$
Let $0=t^{}_0 < t^{}_1 < \ldots < t^{}_N =T$. We partition the time interval $I$ into subintervals $J^{}_n := ( t^{}_{n-1}, t^{}_n ]$ with $k^{}_n := t^{}_n - t^{}_{n-1}$, $n=1,\ldots,N$ each time step. We assume that
\begin{align} \label{quasi}
 \exists C^{}_0 >0  \ \text{s.t.} \  k = \max^{}_{1\leq n \leq N} k^{}_{n} < C^{}_{0} k^{}_{n} \ \  \forall 1 \leq n \leq N \ \text{and} \  \forall k > 0.
\end{align}
Setting $\sigma = (k,h)$, we consider the following fully discrete spaces:
\begin{equation*}
\begin{aligned}
  Y^{}_{\sigma} & := \big \lbrace y^{}_{\sigma} \in L^{}_2(I;H^{1}(\Omega)); \ y^{}_{\sigma} \vert^{}_{J^{}_n} \in Y^{}_{h} , \ 1 \leq n \leq N \rbrace, \\
  U^{}_{\sigma} & := \big \lbrace u^{}_{\sigma} \in L_2(I;L_2(\Omega)) ; \ u^{}_{\sigma} \vert^{}_{J^{}_n} \in U^{}_{h} , \ 1 \leq n \leq N \rbrace.
\end{aligned}
\end{equation*}
The functions of $Y^{}_{\sigma}$ and $U^{}_{\sigma}$ are piecewise constant in time. We seek discrete controls in $U^{}_{\sigma}$ that are written as:
$ u^{}_{\sigma} = \sum^{N}_{n=1}\sum^{}_{\tau \in \mathcal{T}^{}_{h}} u^{}_{n,\tau} \mathcal{\chi}^{}_{n} \mathcal{\chi}^{}_{\tau}, \quad \text{with} \ u^{}_{n,\tau} \in \mathbb{R},
$ where
$\mathcal{\chi}^{}_{n}$, $\mathcal{\chi}^{}_{\tau}$ are the characteristic functions over $(t^{}_{n-1},t^{}_{n})$ and $\tau$, respectively. We consider the convex subset of $U^{}_{\sigma}$,
\begin{equation*}
    U^{}_{\sigma,ad} = U^{}_{\sigma} \cap U^{}_{ad} = \big \lbrace u^{}_{\sigma} \in U^{}_{\sigma}: u^{}_{n,\tau} \in [u^{}_{a}, u^{}_{b}] \big \rbrace.
\end{equation*}
Every element of $ Y^{}_{\sigma}$ can be written as,
$
    y^{}_{\sigma} = \sum^{N}_{n=1} y^{}_{n,h} \mathcal{\chi}^{}_{n}, \quad \text{with} \ y^{}_{n,h} \in Y^{}_h.
$
We fix $y^{}_{\sigma}(t^{}_n)=y^{}_{n,h}$ in order $y^{}_{\sigma}$ to be continuous from the left. Thus, we have $y^{}_{\sigma}(T) = y^{}_{\sigma}(t^{}_{N}) = y^{}_{N,h}$.

To introduce the discrete control problem, we need to define the fully discrete scheme of the state equation \eqref{state}. For any $u \in L_2(I;L_2(\Omega))$, the backward Euler-finite element method (discontinuous in time Galerkin dG(0)) reads:
For each $n=1,\ldots,N$ and for all $w^{}_h \in Y^{}_h$,
\begin{equation}\label{state_sigma}
     \Big( \frac{y^{}_{n,h} {-} y^{}_{n-1,h}}{k^{}_n}, w^{}_h \Big) {+} \left( \nabla y^{}_{n,h}, \nabla w^{}_h\right) {+} \epsilon^{-2} \left( F(y^{}_{n,h}), w^{}_h\right) {=} ( u^{}_{n}, w^{}_h),
\end{equation}
with $ y^{}_{0,h} {:=} y^{}_{0h}$ and
\begin{equation}
\begin{aligned}
    & ( u^{}_{n}, w^{}_h) := \frac{1}{k^{}_n} \int^{t^{}_{n}}_{t^{}_{n-1}} ( u(t), w^{}_h) \diff{t}, \\
    & y^{}_{0h} \in Y^{}_h \ \ \text{s.t.} \ \  \norm{y^{}_{0} - y^{}_{0h}}^{}_{L^{}_2(\Omega)} \leq C h \ \ \text{and} \ \ \norm{y^{}_{0h}}^{}_{H^{1}_{}(\Omega)} \leq C  \ \ \forall h>0.
\end{aligned}
\end{equation}
Then, we define the discrete control problem as follows,
\begin{equation}\label{P_sigma}
\begin{cases}
  & \min J^{}_{\sigma}(u^{}_{\sigma}) \\
  & u^{}_{\sigma} \in U^{}_{\sigma,ad},
\end{cases}
\end{equation}
where
\begin{equation}\label{cost-fun_sigma}
\begin{aligned}
&  J^{}_{\sigma}(u^{}_{\sigma}) {=}  \frac{1}{2} \int_{I} \spaceint{ \abs{y^{}_{\sigma}(u^{}_{\sigma}) {-} y^{}_d}^{2}_{} } \diff{t} {+} \frac{\gamma}{2} \spaceint{ \abs{y^{}_{\sigma}(T) {-} y_{\Omega,h}^{}}^2_{}} {+} \frac{\mu}{2} \int_{I} \spaceint{ \abs{u^{}_{\sigma}}^2_{}} \diff{t}, \\
& y^{}_{\Omega,h} \in Y^{}_h \ \ \text{s.t.} \ \  \norm{y^{}_{\Omega} {-} y^{}_{\Omega,h}}^{}_{L^{}_2(\Omega)} \leq C h \ \ \text{and} \ \ \norm{y^{}_{\Omega,h}}^{}_{H^{1}_{}(\Omega)} \leq C  \ \ \forall h>0.
\end{aligned}
\end{equation}
The study of the control problem consists of four steps. We begin with the analysis and error estimation of the discrete state equation. The choice of the dG(0) method is due to the low regularity imposed by the optimal control setting.
Other approaches for discretization of the Allen-Cahn can be found in \cite{AkrivisLi_2022,AkrivisLiLi_2019} (BDF and extrapolated Runge-Kutta methods via an auxialiary variable formulation), \cite{FengLi_2015} (symmetric interior penalty discontinuous Galerkin methods),  \cite{LiShenRui_2019,ShenXu_2018}  (scalar auxiliary variable / Crank-Nickolson scheme) \cite{ShenYang_2010,Yang_2009} (second order semi-implicit scheme). Results regarding discrete maximum principles can be found in \cite{YangDuZhang_2018,DuJuLiQiao_2019,DuJuLiQiao_2021} (see also references within).

\subsection{Analysis of the discrete state equation}
Let $y = y^{}_u= G(u)$ and $y^{}_{\sigma} = y^{}_{\sigma}(u) \in Y^{}_{\sigma}$ be a solution to \eqref{state_sigma}. We begin by presenting stability estimates.
\begin{lemma}\label{dg-stability-y_sigma}
Let $y^{}_{\sigma} $ be a solution to \eqref{state_sigma}  corresponding to the control function $u \in L_2(I;L_2(\Omega))$ and $y^{}_{0h} {:=} P^{}_{h} y^{}_0$. Then, there exists  a constant $C>0$  independent of $\sigma=(k,h)$, $\epsilon$ and $\norm{y_u}^{}_{_{\infty}},$  such that
\begin{equation} \label{basic_dstatestability}
\begin{aligned}
&  \|y_\sigma \|^{}_{L_2(I;L_2(\Omega))} {+} \|y_\sigma \|^2_{L_{4}(I;L_4(\Omega))}  \leq C( \abs {\Omega_T }^{\frac{1}{2}} {+} \|y_0\|^{}_{L^{}_2(\Omega)} {+} \|u\|^{}_{L_2(I;L_2(\Omega))} ) {:=} C^{\text{dG}}_{\text{st},1},
\end{aligned}
\end{equation}
\begin{equation} \label{dstatestability}
 \| y_\sigma \|_{L_{\infty}(I;L_2(\Omega))} {+} \| y_\sigma \|_{L_2(I;H^1(\Omega))}  {+} \Big(\sum^{N}_{n=1}  \norm{  y^{}_{n,h} {-} y^{}_{n{-}1,h} }^2_{L^{}_2(\Omega)}  \Big)^{1/2} {\leq} C^{\text{dG}}_{\text{st},1}\epsilon^{-1}.
\end{equation}
If in addition, $k {\leq} \frac{3 C_0 \epsilon^2}{2}$, with $C_0$ defined by \eqref{quasi}, then the following estimate holds,
\begin{align} \label{dstatestability_1}
  & \| y_\sigma \|^{}_{L^{}_{\infty}(I;H^{1}(\Omega))} {+} \epsilon^{-1} \| y_\sigma \|^{2}_{L^{}_{\infty}(I;L^{}_{4}(\Omega))} \\
& \leq C \left ( \| \nabla y_{0h} \|^{}_{L_2(\Omega)} {+} \epsilon^{-1} \|y_{0h}\|^2_{L^{}_4(\Omega)} {+}  \abs{\Omega}^{1/2}\epsilon^{-1} {+} \|u\|^{}_{L^{}_2(I;L^{}_2(\Omega))} \right ) {:=} C^{dG}_{st,2} . \nonumber
\end{align}
\end{lemma}
\begin{proof}
The first two stability estimates can be derived by setting $w_h {=} y_\sigma$ into \eqref{state_sigma} (see also \cite[Section 3]{Chrysafinos_2019}). For the proof of the third, we proceed as follows:
We choose $w^{}_h = (y^{}_{n,h} {-} y^{}_{n{-}1,h})/k^{}_n$ in  \eqref{state_sigma}, to get
\begin{align*}
& \Big  \|\frac{y^{}_{n,h} {-} y^{}_{n{-}1,h}}{k^{}_n}  \Big  \|^2_{L^{}_2(\Omega)} {+} \frac{1}{k_n} ( \nabla y_{n,h}, \nabla (y_{n,h}-y_{n-1,h}) ) {+} \frac{1}{\epsilon^2 k_n} (y^3_{n,h},y_{n,h}-y_{n-1,h} ) \\
& {=} \frac{1}{k_n} (u_n,y_{n,h}-y_{n-1,h}) {+} \frac{1}{\epsilon^2 k_n} (y_{n,h},y_{n,h}-y_{n-1,h} ).
\end{align*}
Young's inequality  yields  $\int_\Omega \abs{ y_{n,h}}^3 \abs{ y_{n-1,h}} \diff{x} \leq \frac{3}{4} \|y_{n,h}\|^4_{L_4(\Omega)} {+} \frac{1}{4} \|y_{n-1,h}\|^4_{L_4(\Omega)},$ for $p=\frac{3}{4}$ and $q=4$.
Hence, using H\"{o}lder and Young's inequalities and after some standard algebra, we deduce that
\begin{align*}
    & \frac{3}{4}  \Big \|\frac{y^{}_{n,h} {-} y^{}_{n{-}1,h}}{k^{}_n}  \Big  \|^2_{L^{}_2(\Omega)} {+} \frac{1}{2k^{}_n} \norm{\nabla y^{}_{n,h}}^2_{L^{}_2(\Omega)} {+} \frac{1}{2k^{}_n}\norm{\nabla( y^{}_{n,h} {-} y^{}_{n{-}1,h} )}^2_{L^{}_2(\Omega)} \\
    & +  \frac{1}{4\epsilon^2k^{}_n } \norm{  y^{}_{n,h}}^4_{L^{}_4(\Omega)} {+} \frac{1}{2\epsilon^2k^{}_n } \norm{  y^{}_{n{-}1,h}}^2_{L^{}_2(\Omega)} \leq \norm{u^{}_n}^2_{L^{}_2(\Omega)} {+} \frac{1}{2k^{}_n} \norm{\nabla y^{}_{n{-}1,h}}^2_{L^{}_2(\Omega)} \\
    & + \frac{1}{4\epsilon^2k^{}_n } \norm{  y^{}_{n{-}1,h}}^4_{L^{}_4(\Omega)} {+} \frac{1}{2\epsilon^2k^{}_n } \norm{  y^{}_{n,h}}^2_{L^{}_2(\Omega)} {+} \frac{1}{2\epsilon^2k^{}_n }  \norm{  y^{}_{n,h} {-} y^{}_{n{-}1,h} }^2_{L^{}_2(\Omega)}.
\end{align*}
Multiplying by $k_n$, summing from $n=1$ up to $m$, and using \eqref{quasi}, where $1 \leq m \leq N$, it yields that
\begin{equation}\label{just_a_relation} \hskip-10pt
\begin{aligned}
    & \frac{3C_0}{2k} \sum_{n=1}^m \| y^{}_{n,h} {-} y^{}_{n{-}1,h}  \|^2_{L^{}_2(\Omega)} {+}  \norm{ y^{}_{m,h}}^2_{H^1(\Omega)} {+}  \frac{1}{2\epsilon^2} \norm{  y^{}_{m,h}}^4_{L^{}_4(\Omega)} \\
&  \leq  \norm{\nabla y^{}_{0h}}^2_{L^{}_2(\Omega)} {+} \frac{1}{2\epsilon^2} \norm{  y^{}_{0h}}^4_{L^{}_4(\Omega)}  {+} \left ( \frac{1}{\epsilon^2} {+}1 \right )\norm{ y^{}_{m,h}}^2_{L^{}_2(\Omega)} {+} 2 \norm{u}^{2}_{L^{}_2(0,t^{}_{m};L^{}_2(\Omega))}  \\
& \  {+} \frac{1}{\epsilon^2} \sum^{m}_{n=1}  \norm{  y^{}_{n,h} {-} y^{}_{n{-}1,h} }^2_{L^{}_2(\Omega)},
\end{aligned}
\end{equation}
where we have added $\|y_{m,h}\|^2_{L_2(\Omega)}$ on both sides after summation. Since $\epsilon {\leq} 1$ Young's inequality gives $ (\frac{1}{\epsilon^2}{+}1) \|y_{m,h} \|^2_{L_2(\Omega)} {\leq}  \frac{2}{\epsilon^2} \|y_{m,h}\|^2_{L_2(\Omega)} {\leq} \frac{1}{4 \epsilon^2} \|y_{m,h}\|^4_{L_4(\Omega)} {+} \frac{2 \abs{\Omega}}{\epsilon^2}.$ Finally, choosing $k$ in a way to hide the last term of the right-hand side of \eqref{just_a_relation} on the left, i.e. for $ \frac{1}{\epsilon^2} \leq \frac{3C_0}{2 k},$ we deduce the desired estimate.
\end{proof}
\begin{definition}
We define the projection operator $P^{}_h: L^{}_2(\Omega) \rightarrow Y^{}_h$ through
 $\left( P^{}_h y, w^{}_h \right) {=} \left(y, w^{}_h \right) \quad \forall w^{}_h \in Y^{}_h$.
Also, we  define $P^{}_{\sigma}: C(I;L^{}_2(\Omega)) \rightarrow Y^{}_{\sigma}$ by $\left(P^{}_{\sigma}y\right)^{}_{n,h} {=} P^{}_{h}y(t^{}_n)$, for each $1 \leq n \leq N$.
\end{definition}
\begin{lemma}\label{P_sigma-approximations}
There exists a constant $C>0$ independent of $\sigma$ such that for every $y \in H^{2,1}(\Omega^{}_{T}) \cap C(I;H^{1}(\Omega))$ it holds that
\begin{align}
     & \norm{y {-} P^{}_{\sigma}y}^{}_{L_2(I;L_2(\Omega))}  {\leq} C \Big( k \norm{y^{}_{t}}^{}_{L_2(I;L_2(\Omega))}  {+} h^2  \norm{y }^{}_{L^{}_2(I;H^{2}_{}(\Omega))}\Big), \\
     & \norm{y {-} P^{}_{\sigma}y}^{}_{L^{}_2(I;H^{1}(\Omega))}  {\leq} C \Big( \sqrt{k} \norm{y^{}_{t}}^{}_{L_2(I;L_2(\Omega))}  {+} ( \sqrt{k} {+} h)  \norm{y }^{}_{L^{}_2(I;H^{2}_{}(\Omega))}\Big).
\end{align}
\end{lemma}
Another technical tool is the definition of a global space-time projection onto $Y^{}_{\sigma}$ as a discrete solution to the following auxiliary linear parabolic problem. Let $y \in H^{2,1}(\Omega^{}_{T}) \cap C(I;H^{1}(\Omega))$ be the solution to \eqref{state} and $y^{}_{0h} {:=} P^{}_h y^{}_0.$ We define $\hat{y}^{}_{\sigma} \in Y^{}_{\sigma}$ that satisfies: For each $n=1,\ldots,N$ and $\forall w^{}_h \in Y^{}_h$
\begin{equation}\label{global-projection}
\begin{aligned}
     \Big( \frac{\hat{y}^{}_{n,h} {-} \hat{y}^{}_{n-1,h}}{k^{}_n}, w^{}_h \Big) {+} \left( \nabla \hat{y}^{}_{n,h}, \nabla w^{}_h\right)  & {=} ( \hat{f}^{}_{n}, w^{}_h) \mbox{ \quad and \quad } \hat{y}^{}_{0,h}  = y^{}_{0h}.
    \end{aligned}
\end{equation}
The right hand side is defined as
\begin{equation} \label{global-projection-2}
\begin{aligned}
     ( \hat{f}^{}_{n}, w^{}_h) & {=} \frac{1}{k^{}_n} \int^{t^{}_{n}}_{t^{}_{n-1}}  ( y_t(t) - \Delta y(t), \nabla  w^{}_h) \diff{t}\\
     & {=} \frac{1}{k^{}_n} \int^{t^{}_{n}}_{t^{}_{n-1}} ( \nabla y(t), \nabla  w^{}_h) \diff{t} {+} \left( \frac{y(t^{}_{n}) - y(t^{}_{n-1})}{k^{}_n}, w^{}_h \right),
\end{aligned}
\end{equation}
where we have used integration by parts in space and the fact that $w_h \in Y_h$ is independent of $t$. It is clear (see for instance \cite[Lemma 4.6]{Casas-Chrysafinos_2015}) that \eqref{global-projection} has a unique solution $\hat{y}^{}_{\sigma} \in Y^{}_{\sigma}$.
We split the error as follows:
$\hat{e}=y {-} \hat{y}^{}_{\sigma} {=} (y - P^{}_{\sigma}y) {+} ( P^{}_{\sigma}y {-} \hat{y}^{}_{\sigma}).$
Substituting \eqref{global-projection-2} into \eqref{global-projection}, we note that $\hat{e}$ satisfies the following orthogonality condition, for each $w_h \in Y_h$, $\forall n=1,...,N$,
\begin{equation}\label{global-projection-3}
\begin{aligned}
     \Big( \hat{e}(t_n) {-} \hat{e}(t_{n-1}) , w^{}_h \Big) {+} \int_{t_{n-1}}^{t_n} \left( \nabla \hat{e}, \nabla w^{}_h\right) \diff{t} & {=} 0.
    \end{aligned}
\end{equation}
The following Lemma collects various stability and error estimates.
\begin{lemma}\label{global-space-time-approximations}
Suppose that $\hat{y}^{}_{\sigma} \in Y^{}_{\sigma}$
 is the solution  to \eqref{global-projection}. Then, there exists $C>0$ independent of $\sigma$ and $\epsilon$, such that,
 \begin{equation}
 \begin{cases}
    & \norm{\hat{e}^{}}^{}_{L_{\infty}(I;L_2(\Omega))} {+} \norm{\hat{e}^{}}^{}_{L^{}_{2}(I;H^{1}(\Omega))}\\
    & \qquad \leq C \left( \sqrt{k} \norm{y^{}_{t}}^{}_{L_2(I;L_2(\Omega))}  {+} h  \norm{y }^{}_{L^{}_2(I;H^{2}_{}(\Omega))} {+} h \norm{y^{}_0}_{H^{1}(\Omega)} \right), \\
   &  \norm{\hat{e}}^{}_{L^{}_{2}(I;L^{}_2(\Omega))}  \leq C \left( k \norm{y^{}_{t}}^{}_{L_2(I;L_2(\Omega))}  {+} h^2  \norm{y }^{}_{L^{}_2(I;H^{2}_{}(\Omega))} \right), \\
   &\label{projection_stability}
  \norm{\hat{y}^{}_{\sigma}}^{}_{L^{}_{\infty}(I;H^1(\Omega))} \leq C \| y\|^{}_{H^{2,1}(\Omega^{}_T)}.
\end{cases}
\end{equation}
\end{lemma}
\begin{proof}
Recall that we split the error as $\hat{e}=y - \hat{y}^{}_{\sigma} = (y - P^{}_{\sigma}y) + ( P^{}_{\sigma}y - \hat{y}^{}_{\sigma}).$ In view of Lemma \ref{P_sigma-approximations}, the first term is bounded immediately. As for the estimation of the second one, that belongs on $Y^{}_{\sigma}$, we refer the readers to \cite{Chrysafinos-Walkington_2006,Chrysafinos-Walkington_2010}, \cite{Casas-Chrysafinos_2019} and \cite{Casas-Chrysafinos_2015}.
\end{proof}
The following result states the main error estimate for the control to state mapping. We emphasise that, unlike previous works for the uncontrolled Allen-Cahn equation, we do not exceed the $H^{2,1}(\Omega^{}_T)$ regularity. Our technique  employs the spectral estimate at the ``continuous level", and hence it avoids the construction of discrete approximations, which typically lead to higher regularity requirements.
\begin{theorem}\label{Theorem_1}
Let $u \in L_2(I;L_2(\Omega)),y \in  H^{2,1}(\Omega^{}_{T}) \cap C(I;H^{1}(\Omega))$  and $y^{}_{\sigma} \in Y^{}_{\sigma}$ satisfy \eqref{state} and \eqref{state_sigma} respectively. Suppose that \eqref{spectral_classical} holds with $\| \lambda \|_{L^{}_{\infty}(I)} {\leq} C.$ \\
1. For $d=2$, if there exists $C>0$ such that
\begin{align}\label{sigma-epsilon-condition_clean_2}
    (\sqrt{k} {+} h) \norm{y}^{1/2}_{H^{2,1}(\Omega^{}_T)} {\max} \Big \lbrace C^{}_{\infty}\epsilon^{-1},  \norm{y}^{}_{H^{2,1}(\Omega^{}_T)} \Big \rbrace^{{1/2}}_{} \norm{y}^{1/2}_{_{\infty}} {\leq} \epsilon^{2} C  E^{-1/2},
\end{align}
 then, the following estimates hold:
\begin{equation}\label{es_1_1} \hskip-8pt
\begin{aligned}
    & \norm{y {-}y^{}_{\sigma}}^{}_{L_{\infty}(I;L_2(\Omega))} + \epsilon^{-1}
\|y-y_\sigma\|^2_{L^{}_4(I;L^{}_4(\Omega)}  \\
    & \leq   \max \Big\lbrace  C^{}_{I}\epsilon^{-2}( \sqrt{k} {+} h),C^{}_{II}\epsilon^{-1} ( \sqrt{k} {+} h)  \norm{y}^{}_{H^{2,1}_{}(\Omega^{}_T)}, C\Big \rbrace ( \sqrt{k} {+} h) \norm{y}^{}_{H^{2,1}_{}(\Omega^{}_T)},
\end{aligned}
\end{equation}
\begin{equation}\label{es_1_2} \hskip-8pt
\begin{aligned}
    & \norm{y {-}y^{}_{\sigma}}^{}_{L^{}_{2}(I;H^{1}(\Omega))} \\
    & {\leq}  \max \Big\lbrace  C^{}_{I}\epsilon^{{-}3}  ( \sqrt{k} {+} h),C^{}_{II}\epsilon^{{-}2} ( \sqrt{k} {+} h)  \norm{y}^{}_{H^{2,1}_{}(\Omega^{}_T)} ,C   \Big \rbrace ( \sqrt{k} {+} h) \norm{y}^{}_{H^{2,1}_{}(\Omega^{}_T)}.
\end{aligned}
\end{equation}
2. For $d=3$, if there exists $C>0$ such that
\begin{align} \label{sigma-epsilon-condition_clean_3}
   (\sqrt{k} {+} h) \norm{y}^{4/3}_{H^{2,1}(\Omega^{}_T)}  \norm{y}^{2/3}_{_{\infty}}{\leq} \epsilon^{3} C  E^{-1},
\end{align}
then, the following estimates hold:
\begin{equation}\label{es_2_1}
\begin{aligned}
 &\norm{y{-}y^{}_{\sigma}}^{}_{L_{\infty}(I;L_2(\Omega))} + \epsilon^{-1} \|y-y_\sigma \|^2_{L_4(I;L_4(\Omega))} \\
& \leq \max \Big \lbrace
  C^{}_{II} \epsilon^{-1} ( \sqrt{k} {+}h)^{1/2}_{} \norm{y}^{}_{H^{2,1}_{}(\Omega^{}_T)}, C \Big \rbrace ( \sqrt{k} {+} h) \|y \|_{H^{2,1}(\Omega_T)},
\end{aligned}
\end{equation}
\begin{equation}\label{es_2_2}
\begin{aligned}
 &\norm{y {-}y^{}_{\sigma}}^{}_{L^{}_{2}(I;H^{1}(\Omega))} \\
 & \leq \max \Big \lbrace
  C^{}_{II}\epsilon^{-2}  (\sqrt{k} {+}h)^{1/2}_{} \norm{y}^{}_{H^{2,1}_{}(\Omega^{}_T)}, C \Big \rbrace ( \sqrt{k} {+} h) \|y \|_{H^{2,1}(\Omega_T)}.
\end{aligned}
\end{equation}
Here,  we denote by  $ C^{}_{I}{: =} 2E^{1/2}  C^{}_{\infty}$, $C^{}_{II}{: =} CE^{1/2}$and $E{ : =} \exp(2T \alpha)$ where
\begin{align*}
    C^{}_{\infty} & {: =}  C\Big(1 {+} (1{+ }\epsilon^2) \norm{y}^{2}_{_{\infty}} \Big)^{1/2}, \ \alpha    {: =}   \sup^{}_{t \in I} (2\lambda(t) (1{-}\epsilon^2) {+}  4 + 2 \epsilon^2),
\end{align*}
and $C{>}0$ an algebraic constant (that might be different in each occurrence) but independent of $\sigma, \epsilon$, and $\|y\|_{_{\infty}}$.
\end{theorem}
 None of the above  constants depend on $\norm{y}^{}_{_{\infty}}$ exponentially. Indeed, $C^{}_{I}$ depends on the $\norm{y}^{}_{_{\infty}}$ polynomially while $C>0$ is an algebraic constant  independent of $\epsilon.$ We mainly focus on the case where  $\norm{y}^{}_{_{\infty}}$ is bounded independent  $1/\epsilon$ (see Remark \ref{rates} for a detailed discussion).
However,  our results hold, without any exponential dependence upon $1 / \epsilon$ even in more general cases.
\begin{proof}
We begin by splitting the total error:
\begin{equation}\label{error-decomposition}
    e {=} y {-} y^{}_{\sigma} {=} (y {-} \hat{y}^{}_{\sigma}){+} (\hat{y}^{}_{\sigma} {-} y^{}_{\sigma}) {=} \hat{e} {+} e^{}_{\sigma}.
\end{equation}
The aim of our analysis is to bound the term $ e^{}_{\sigma}$ in terms of $\hat{e}$, whose bounds are known from Lemma \ref{global-space-time-approximations}.
From relations \eqref{weak-state} and \eqref{state_sigma}, for all $w^{}_h \in Y^{}_h$, $\forall n=1,...,N$ it holds that
\begin{equation*}
    \left( e(t^{}_n){-}e(t^{}_{n-1}) , w^{}_h \right) {+} \int^{t^{}_n}_{t^{}_{n-1}} \left( \nabla e(t), \nabla w^{}_h \right) \diff{t} {+} \epsilon^{-2}  \int^{t^{}_n}_{t^{}_{n-1}} \left(F(y) {-} F(y^{}_{n,h}) ,  w^{}_h \right) \diff{t} {=} 0.
\end{equation*}
Using \eqref{error-decomposition}, the orthogonality condition \eqref{global-projection-2} and choosing $w^{}_h {=} e^{}_{n,h}$ we obtain,
\begin{equation*}
    \left( e^{}_{n,h}{-}e^{}_{n-1,h} , e^{}_{n,h} \right) {+} \int^{t^{}_n}_{t^{}_{n-1}} \norm{ \nabla e^{}_{n,h}}^{2}_{L^{}_2(\Omega)} \diff{t} {+} \epsilon^{-2}  \int^{t^{}_n}_{t^{}_{n-1}} \left( F(y){-} F(y^{}_{n,h}),  e^{}_{n,h} \right) \diff{t} {=} 0.
\end{equation*}
A standard algebraic manipulation implies that
\begin{align*}
  F(y){-}F(y^{}_{n,h})  {=} (3y^2{-}1)(y {-} y^{}_{n,h}) {-} 3y(y{-} y^{}_{n,h})^2 {+} (y{-}y^{}_{n,h})^3.
\end{align*}
Using the decomposition \eqref{error-decomposition} and elementary identities, we deduce that
\begin{equation} \label{important}
    \begin{aligned}
    &\frac{1}{2} \norm{e^{}_{n,h}}^{2}_{L^{}_2(\Omega)} {-} \frac{1}{2} \norm{e^{}_{n{-}1,h}}^{2}_{L^{}_2(\Omega)} {+} \frac{1}{2} \norm{e^{}_{n,h} {-} e^{}_{n{-}1,h} }^{2}_{L^{}_2(\Omega)} {+} \epsilon^{-2} \int^{t^{}_n}_{t^{}_{n-1}}  \norm{  e^{}_{n,h}}^{4}_{L^{}_4(\Omega)} \diff{t}\\
    & {+}\int^{t^{}_n}_{t^{}_{n-1}} \left( \norm{ \nabla e^{}_{n,h}}^{2}_{L^{}_2(\Omega)}  {+} \epsilon^{-2}  \left( F'(y) e^{}_{n,h}, e^{}_{n,h} \right) \right) \diff{t}
    {+} 3 \epsilon^{-2} \int^{t^{}_n}_{t^{}_{n-1}} (\hat{e}^2, e^{2}_{n,h}) \diff{t}\\
    {=} & {-} \epsilon^{-2} \int^{t^{}_n}_{t^{}_{n-1}} \left( F'(y) \hat{e}, e^{}_{n,h} \right)  \diff{t} {+}  3 \epsilon^{-2} \int^{t^{}_n}_{t^{}_{n-1}} \left( y \hat{e}^2, e^{}_{n,h} \right)  \diff{t} {+} 6 \epsilon^{-2} \int^{t^{}_n}_{t^{}_{n-1}} \left( y \hat{e}, e^{2}_{n,h} \right)  \diff{t} \\
    &   {-} \epsilon^{-2} \int^{t^{}_n}_{t^{}_{n-1}} \left( \hat{e}^3, e^{}_{n,h} \right)  \diff{t} {-} 3 \epsilon^{-2} \int^{t^{}_n}_{t^{}_{n-1}} \left( \hat{e}, e^{3}_{n,h} \right)  \diff{t} {+} 3 \epsilon^{-2} \int^{t^{}_n}_{t^{}_{n-1}} \left( y , e^{3}_{n,h} \right)  \diff{t}
    {=:}   \sum_{j=1}^{6} \mathcal{I}^{}_{j}.
    \end{aligned}
\end{equation}
First, we recover additional coercivity at the right-hand side based on \eqref{spectral_classical}. Indeed,
\begin{equation*}
\begin{aligned}
& \int^{t^{}_n}_{t^{}_{n-1}} \left( \norm{ \nabla e^{}_{n,h}}^{2}_{L^{}_2(\Omega)}  {+} \epsilon^{-2}  \left( F'(y) e^{}_{n,h}, e^{}_{n,h} \right) \right) \diff{t} \\
& \geq   \intJn{ \left((\epsilon^2 {-}1 )  \lambda(t) {-}1\right) \|e_{n,h}\|^2_{L_2(\Omega)} } {+} \epsilon^2 \intJn{ \| \nabla e_{n,h}\|^2_{L_2(\Omega)} }  {+} 3 \intJn{\|y e_{n,h}\|^2_{L_2(\Omega)}}.
\end{aligned}
\end{equation*}
We estimate the right-hand side of \eqref{important}. H\"older and Young's inequalities yield,
\begin{align*}
    \mathcal{I}^{}_1 & {\leq}   \intJn{\|y e_{n,h}\|^2_{L_2(\Omega)}} {+} \frac{1}{4\epsilon^{4}} \intJn{\left( 9 \norm{y}^2_{L^{}_{\infty}(\Omega)} {+} 1 \right) \norm{\hat{e}}^{2}_{L^{}_2(\Omega)} } {+}  \intJn{ \norm{e^{}_{n,h}}^{2}_{L^{}_2(\Omega)} },\\
    \mathcal{I}^{}_2 & {\leq} \frac{1}{2\epsilon^{2}} \intJn{ \| {\hat e} e_{n,h} \|^2_{L_2(\Omega)} } {+} \frac{9}{2 \epsilon^2} \intJn{  \| y \|^2_{L_{\infty}(\Omega)} \norm{ {\hat e} }^{2}_{L^{}_2(\Omega)} },\\
    \mathcal{I}^{}_3 & {\leq} \frac{36}{\epsilon^{2}} \intJn{ \norm{y}^2_{L^{}_{\infty}(\Omega)} \norm{\hat{e}}^{2}_{L^{}_2(\Omega)} } {+} \frac{1}{4\epsilon^{2}} \intJn{ \norm{e^{}_{n,h}}^{4}_{L^{}_4(\Omega)} }, \\
   \mathcal{I}^{}_4 & \leq \frac{3}{4\epsilon^{2}} \intJn{  \norm{\hat{e}}^{4}_{L^{}_4(\Omega)} } + \frac{1}{4\epsilon^{2}} \intJn{ \norm{e^{}_{n,h}}^{4}_{L^{}_4(\Omega)} },\\
   \mathcal{I}^{}_5 & \leq \frac{3^7}{4\epsilon^{2}} \intJn{  \norm{\hat{e}}^{4}_{L^{}_4(\Omega)} } + \frac{1}{4\epsilon^{2}} \intJn{ \norm{e^{}_{n,h}}^{4}_{L^{}_4(\Omega)} }.
\end{align*}
The bound  of $\mathcal{I}^{}_6$ for $d=3$ differs from one for $d=2$. We  begin with $d=3$:
\begin{align*}
& \mathcal{I}^{}_6  \leq \frac{3}{\epsilon^{2}} \intJn{ \norm{y}^{}_{L^{}_{\infty}(\Omega)} \norm{e^{}_{n,h}}^{3}_{L^{}_3(\Omega)} }\\
& \leq \frac{12}{\epsilon^{2}}\intJn{ \norm{y}^{}_{L^{}_{\infty}(\Omega)} \norm{ e^{}_{n,h} -  e^{}_{n{-}1,h}}^{3}_{L^{}_3(\Omega)} } + \frac{12}{\epsilon^{2}}\intJn{ \norm{y}^{}_{L^{}_{\infty}(\Omega)} \norm{e^{}_{n{-}1,h}}^{3}_{L^{}_3(\Omega)} }.
\end{align*}
We want to ensure that the first term of ${\mathcal I}_{6}$ can be absorbed by the term  $\frac{1}{2} \norm{e^{}_{n,h} - e^{}_{n{-}1,h} }^{2}_{L^{}_2(\Omega)}$ of the left-hand side of \eqref{important}. Using \eqref{GNL2} and \eqref{inv} it yields that
$\norm{e^{}_{n,h} {-} e^{}_{n{-}1,h}}^{3}_{L^{}_3(\Omega)} \leq h^{-1/2}C^{}_{\text{inv}} \tilde{c}^3 \norm{e^{}_{n,h} {-} e^{}_{n{-}1,h}}^{2}_{L^{}_2(\Omega)} \norm{ e^{}_{n,h} {-}  e^{}_{n{-}1,h}}^{}_{H^1(\Omega)}.$  We can bound the quantity  $\norm{ e^{}_{n,h} -  e^{}_{n{-}1,h}}^{}_{H^1(\Omega)}$ through \eqref{dstatestability_1} and \eqref{projection_stability}.
Therefore, it is enough to assume that for all $n=1,\ldots,N$ it holds
\begin{align} \label{not_another_cond_3}
   & 12 \tilde{c}^3 C^{}_{\text{inv}} \norm{y}^{}_{L^{}_{\infty}(t^{}_{n-1},t^{}_n;L^{}_{\infty}(\Omega))} \frac{k^{}_n}{\epsilon^{2} \sqrt{h}}  \left( C^{\text{dG}}_{\text{st,2}} + \|y\|_{H^{2,1}(\Omega_T)}  \right) \leq 1/4.
\end{align}
 Recalling \eqref{spectral_classical} for $\varv=e^{}_{n,h}$, inserting the bounds of ${\mathcal I}_i$ into \eqref{important} and adding $\epsilon^2 \int_{t_{n-1}}^{t_n} \|e_{n,h}\|^2_{L_2(\Omega)} \diff{t}$ to both sides we deduce,
\begin{align*}
    & \norm{e^{}_{n,h}}^{2}_{L^{}_2(\Omega)} {-}  \norm{e^{}_{n{-}1,h}}^{2}_{L^{}_2(\Omega)} {+}
     \frac{1}{2\epsilon^{2}} \int^{t^{}_n}_{t^{}_{n-1}}
    \norm{  e^{}_{n,h}}^{4}_{L^{}_4(\Omega)} \diff{t} {+} 2 \epsilon^{2} \int^{t^{}_n}_{t^{}_{n-1}}  \norm{ e^{}_{n,h}}^{2}_{H^1(\Omega)}  \diff{t}   \\
    & \ {+} \frac{5}{ \epsilon^2} \intJn{ \|{\hat e} e_{n,h} \|^2_{L_2(\Omega)}} {+} 4 \intJn{\| ye_{n,h}\|^2_{L_2(\Omega)} }
     + \frac{1}{2} \|e_{n,h} {-} e_{n-1,h} \|^2_{L_2(\Omega)}  \\
    & {\leq}   \intJn{   \left( \Big(\frac{9}{2\epsilon^4} {+}  \frac{81}{\epsilon^2} \Big) \norm{y}^2_{L^{}_{\infty}(\Omega)} {+} \frac{1}{2\epsilon^4}  \right) \norm{\hat{e}}^{2}_{L^{}_2(\Omega) } } {+} \intJn{   \Big(  \frac{3}{2\epsilon^2}  {+} \frac{3^7}{2\epsilon^2}  \Big) \norm{\hat{e}}^{4}_{L^{}_4(\Omega) } }\\
    &  \ {+} \intJn{\left(2 \lambda(t)(1{-} \epsilon^2) {+} 4 + 2 \epsilon^2 \right) \norm{e^{}_{n,h}}^{2}_{L^{}_2(\Omega) }} {+} \frac{24}{\epsilon^{2}} \intJn{ \norm{y}^{}_{L^{}_{\infty}(\Omega)} \norm{e^{}_{n{-}1,h}}^{3}_{L^{}_3(\Omega) }}.
\end{align*}
Summing from $n{=}1$ up to $m$ where $1\leq m \leq N$, and dropping positive terms, we obtain
\begin{equation}
    \begin{aligned}
    & \norm{e^{}_{m,h}}^{2}_{L^{}_2(\Omega)}
     {+} \frac{1}{2\epsilon^{2}} \sum^{m}_{n=1} k^{}_n
    \norm{  e^{}_{n,h}}^{4}_{L^{}_4(\Omega)} {+} 2 \epsilon^{2}  \sum^{m}_{n=1} k^{}_n \norm{  e^{}_{n,h}}^{2}_{H^1(\Omega)} \\
    & {\leq}  A {+} \sum^{m}_{n=1} k^{}_n \alpha \norm{e^{}_{n,h}}^{2}_{L^{}_2(\Omega) }{+} \frac{24}{\epsilon^{2}}  \sum^{m-1}_{n=1} \intJn{ \norm{y}^{}_{L^{}_{\infty}(\Omega)} \norm{e^{}_{n,h}}^{3}_{L^{}_3(\Omega) }},
\end{aligned}
\end{equation}
where we denote by $\alpha    {=}   \sup^{}_{t \in I} (2\lambda(t) (1{-}\epsilon^2) {+}  4 +2 \epsilon^2 )$ and by
\begin{equation*}\label{A}
\begin{aligned}
    A & {=:}  \left(
    \Big( (9/2) \epsilon^{-4} {+} 81 \epsilon^{-2} \Big) \norm{y}^{2}_{_{\infty}} {+} (1/2) \epsilon^{-4}  \right) \norm{\hat{e}}^{2}_{L^{}_{2}(I;L^{}_{2}(\Omega))} \\
&{+} \left( (3/2)\epsilon^{-2} {+} (3^7/2)\epsilon^{-2} \right) \norm{\hat{e}}^{4}_{L^{}_{4}(I;L^{}_{4}(\Omega))}.
\end{aligned}
\end{equation*}
Assume that $\sup^{}_{n=1,\ldots,N} k^{}_n \alpha \leq 1/2$. Note that for every $m=1,\ldots,N,$  \eqref{GNL2} and Young's inequalities (with $p=4$ and $q=4/3$) and standard algebra imply that
\begin{align*}
    & \frac{24}{\epsilon^{2}}  \sum^{m-1}_{n=1} \intJn{ \norm{y}^{}_{L^{}_{\infty}(\Omega)} \norm{e^{}_{n,h}}^{3}_{L^{}_3(\Omega) }} \\
& {\leq} \frac{24 \tilde{c}^3}{\epsilon^{2}}  \sum^{m-1}_{n=1} \intJn{ \norm{y}^{}_{L^{}_{\infty}(\Omega)} \norm{e^{}_{n,h}}^{3/2}_{L^{}_2(\Omega) } \| e_{n,h} \|^{3/2}_{H^1(\Omega)}}  \\
    & {=} \frac{24 \tilde c^3}{\epsilon^{7/2}}  \sum^{m-1}_{n=1} \intJn{ \norm{y}^{}_{L^{}_{\infty}(\Omega)} \norm{e^{}_{n,h}}_{L^{}_2(\Omega) } \left (  \norm{e^{}_{n,h}}^{1/2}_{L^{}_2(\Omega) } \epsilon^{3/2} \| e_{n,h} \|^{3/2}_{H^1(\Omega)} \right )} \\
    & {\leq} \frac{24 \tilde{c}^3}{\epsilon^{7/2}} \norm{y}^{}_{_{\infty}} \sup^{}_{n=1,\ldots,m-1} \norm{e^{}_{n,h}}^{}_{L^{}_{2}(\Omega)}  \sum^{m-1}_{n=1} k^{}_n \left(  \norm{e^{}_{n,h}}^{2}_{L^{}_2(\Omega) } {+}  \epsilon^2 \norm{e^{}_{n,h}}^{2}_{H^1(\Omega)} \right).
\end{align*}
 Then, applying  Lemma \ref{discrete_GL}  we deduce that
 \begin{equation}\label{GL-discrete-result}
    \sup^{}_{n=1,\ldots,N} \norm{e^{}_{n,h}}^{2}_{L^{}_{2}(\Omega)} {+} 2 \epsilon^{2}   \sum^{N}_{n=1} k^{}_n \norm{e^{}_{n,h}}^{2}_{H^1(\Omega)} {+} \frac{2}{\epsilon^2} \sum^{N}_{n=1} k^{}_n \norm{ e^{}_{n,h}}^{4}_{L^{}_{4}(\Omega)} {\leq} 4A E.
 \end{equation}
 The above estimate holds,  upon setting  $\beta = 1/2$ and $B = 24 \epsilon^{-7/2}\tilde{c} \norm{y}^{}_{_{\infty}}$ in Lemma \ref{discrete_GL} as long as
 \begin{align}\label{sigma-epsilon-condition_1}
     A  {\leq}  \epsilon^7 \left( 192\tilde{c} \norm{y}^{}_{_{\infty}} (T{+}1) E^{3/2} \right)^{-2} := \epsilon^7 C_T E^{-3} \|y\|^{-2}_{_{\infty}}.
 \end{align}
To quantify the relation between $k,h$ and $\epsilon$ from \eqref{sigma-epsilon-condition_1}, observe that  Lemma \ref{global-space-time-approximations} and the embedding $H^{2,1}(\Omega_{T}) \subset C(I;H^1(\Omega))$ imply $\|y-\hat y_\sigma \|^{}_{L^{}_{\infty}(I;H^1(\Omega))} \leq C\|y\|^{}_{H^{2,1}(\Omega_T)}$. Hence, denoting by $C^2_{\infty}:= C(1 +(1+\epsilon^2) \|y\|^2_{_{\infty}} )$ where $C$ is an algebraic constant, \eqref{GNL3} and  Lemma \ref{global-space-time-approximations} imply that
 \begin{equation*} \label{Acondition}
 \begin{aligned}
      A & {\sim} C^{2}_{\infty}\epsilon^{-4} \norm{ \hat {e} }^{2}_{L^{}_{2}(I;L^{}_{2}(\Omega))}  {+} C \tilde{c}^4 \|y\|^{}_{H^{2,1}(\Omega_T)} \epsilon^{-2} \norm{ \hat{e} }^{}_{L^{}_{\infty}(I;L^{}_{2}(\Omega ))} \norm{ \hat{e} }^{2}_{L^{}_{2}(I;H^{1}(\Omega ))}   \\
     & {\sim}   C^{2}_{\infty}\epsilon^{-4} ( k+h^2)^2 \norm{y}^{2}_{H^{2,1}(\Omega^{}_T)} {+} C\tilde{c}^4\epsilon^{-2}(\sqrt{k}+h)^3 \norm{y}^{4}_{H^{2,1}(\Omega^{}_T)} .
 \end{aligned}
 \end{equation*}
It is clear that if
\begin{equation} \label{not_condition}
\sqrt{k} + h \leq \epsilon^2 \|y\|^2_{H^{2,1}(\Omega_T)}C^{-2}_{\infty}
\end{equation}
then the second term of \eqref{Acondition} dominates the above quantity, and hence \eqref{Acondition} and \eqref{sigma-epsilon-condition_1} result to the following restriction:
$$  (\sqrt{k} + h)^3 \norm{y}^{4}_{H^{2,1}(\Omega^{}_T)} \tilde{c}^{4} C \leq \epsilon^{9} C^{}_T E^{-3} \norm{y}^{-2}_{_{\infty}},$$
where $C^{}_T{>}0$ is independent of $k,h,\epsilon$, and $\|y\|^{}_{_{\infty}}$, $\|y\|^{}_{H^{2,1}(\Omega^{}_{T})}$.  The estimates \eqref{es_2_1} and \eqref{es_2_2} follow by triangle inequality and the estimate of Lemma \ref{global-space-time-approximations}. Note that if \eqref{sigma-epsilon-condition_clean_3}  is satisfied then both conditions \eqref{not_another_cond_3} and \eqref{not_condition} are also satisfied.  \\
For $d=2$,  we proceed in a similar way. Using H\"older, \eqref{GNL2}, and Young's inequalities we deduce,
\begin{align*}
\mathcal{I}^{}_6 & \leq \frac{3\tilde{c}}{\epsilon^{2}} \intJn{ \norm{y}^{}_{L^{}_{\infty}(\Omega)} \norm{e^{}_{n,h}}^{2}_{L^{}_2(\Omega)}  \norm{e^{}_{n,h}}_{H^1(\Omega)} }\\
& \leq \frac{6 \tilde{c}}{\epsilon^2} \intJn{ \norm{y}_{L^{}_{\infty}(\Omega)} \left( \norm{e^{}_{n,h}{-}e_{n-1,h}}^{2}_{L^{}_2(\Omega)} {+} \norm{e^{}_{n-1,h}}^{2}_{L^{}_2(\Omega)}\right)  \| e_{n,h} \|_{H^1(\Omega)} } \\
& \leq   \frac{6 \tilde{c}}{\epsilon^2} \intJn{ \norm{y}_{L^{}_{\infty}(\Omega)}  \norm{e^{}_{n,h}{-}e_{n-1,h}}^{2}_{L^{}_2(\Omega)}  \|e_{n,h} \|_{H^1(\Omega)} } \\
& \quad  {+} \frac{36 \tilde{c}^2}{\epsilon^6} \intJn{ \|y\|^2_{L_{\infty}(\Omega)} \|e_{n-1,h} \|^4_{L_2(\Omega)} } {+} \frac{\epsilon^2}{4} \intJn{  \|e_{n,h} \|^2_{H^1(\Omega)} }.
\end{align*}
Then, through \eqref{dstatestability_1} and \eqref{projection_stability}, we conclude to an analogous  assumption to \eqref{not_another_cond_3},
\begin{align} \label{not_another_cond_2}
   & 6 \tilde{c}  \norm{y}_{L^{}_{\infty}(t^{}_{n-1},t^{}_n;L^{}_{\infty}(\Omega))}  k^{}_n \epsilon^{-2}   \left( C^{\text{dG}}_{\text{st,2}} {+} \|y\|_{H^{2,1}(\Omega_T)}  \right) \leq 1/4.
\end{align}
Working exactly as in $d=3$-case, we have that
\begin{equation}
    \begin{aligned}
    & \norm{e^{}_{m,h}}^{2}_{L^{}_2(\Omega)}
     {+} \frac{1}{2\epsilon^{2}} \sum^{m}_{n=1} k^{}_n
    \norm{  e^{}_{n,h}}^{4}_{L^{}_4(\Omega)} {+} \frac{3}{2} \epsilon^{2}  \sum^{m}_{n=1} k^{}_n \norm{ e^{}_{n,h}}^{2}_{H^1(\Omega)} \\
    & \leq  A {+} \sum^{m}_{n=1} k^{}_n \alpha \norm{e^{}_{n,h}}^{2}_{L^{}_2(\Omega) } {+} \frac{36 \tilde{c}^2}{\epsilon^{6}}  \sum^{m-1}_{n=1} \intJn{ \norm{y}^{2}_{L^{}_{\infty}(\Omega)} \norm{e^{}_{n,h}}^{4}_{L^{}_2(\Omega) }},
\end{aligned}
\end{equation}
 where $\sup^{}_{n=1,\ldots,N} k^{}_n \alpha \leq 1/2.$ Note that for every $m=1,\ldots,N$
\begin{align*}
    & \frac{36 \tilde{c}^2}{\epsilon^{6}}  \sum^{m-1}_{n=1} \intJn{ \norm{y}^{2}_{L^{}_{\infty}(\Omega)} \norm{e^{}_{n,h}}^{4}_{L^{}_2(\Omega) }} \\
&\  \leq \frac{36 \tilde{c}^2}{\epsilon^{6}} \norm{y}^{2}_{_{\infty}} \sup^{}_{n=1,\ldots,m-1} \norm{e^{}_{n,h}}^{2}_{L^{}_{2}(\Omega)}  \sum^{m-1}_{n=1} k^{}_n  \norm{e^{}_{n,h}}^{2}_{L^{}_2(\Omega) }.
\end{align*}
From Lemma \ref{discrete_GL} for $\beta = 1$ and $B= 36 \epsilon^{-6}\tilde{c}^2 \norm{y}^{2}_{_{\infty}}$, there holds an identical result to \eqref{GL-discrete-result}  as long the assumption
\begin{align}\label{sigma-epsilon-condition_2}
     A {\leq}  \epsilon^6 \left( 16 {\cdot} 36 \tilde{c}^4 \norm{y}^{2}_{_{\infty}} (T{+}1) E^{2}_{} \right)^{-1} {:=} \epsilon^6 C_T E^{-2} \|y\|^{-2}_{_{\infty}} ,
 \end{align}
is satisfied.  To quantify the dependence of $k$,$h$ upon $\epsilon$ through \eqref{sigma-epsilon-condition_2} observe that \eqref{GNL1} and Lemma \ref{global-space-time-approximations}, implies
\begin{equation*}\label{A_2}
 \begin{aligned}
      A & \sim \frac{C^{2}_{\infty}}{\epsilon^{4}}   \norm{\hat{e}}^{2}_{L^{}_{2}(I;L^{}_{2}(\Omega))} + \frac{C \tilde{c}^4}{\epsilon^2}  \norm{ \hat{e} }^{2}_{L^{}_{\infty}(I;L^{}_{2}(\Omega ))} \norm{ \hat{e} }^{2}_{L^{}_{2}(I;H^{1}(\Omega ))}  \\
     & \sim  \frac{C^{2}_{\infty}}{\epsilon^{4}}  ( k{+}h^2)^2 \norm{y}^{2}_{H^{2,1}(\Omega^{}_T)} {+}  \frac{C\tilde{c}^4}{\epsilon^2} ( k{+}h^2)^2 \norm{y}^{4}_{H^{2,1}_{}(\Omega^{}_T)},
 \end{aligned}
 \end{equation*}
 that along with \eqref{sigma-epsilon-condition_1} imply  $$(k {+} h^2)^2 \norm{y}^{2}_{H^{2,1}(\Omega^{}_T)} \max \Big \lbrace \frac{C^{2}_{\infty}}{\epsilon^2},  C\tilde{c}^4 \norm{y}^{2}_{H^{2,1}(\Omega^{}_T)} \Big \rbrace \leq \epsilon^{8} C^{}_{T} E^{-2} \norm{y}^{-2}_{_{\infty}}.$$
 Observe that if the above estimate is satisfied then \eqref{not_another_cond_2} also holds.
 The estimate follows by triangle inequality and the estimate of Lemma \ref{global-space-time-approximations}.
\end{proof}

\begin{remark}\label{rates}
Recall that $C^{}_{I} {\sim}  \norm{y}^{}_{_{\infty}}$ and  $\norm{y}^{}_{H^{2,1}_{}(\Omega^{}_T)}{\sim} \epsilon^{-r},$ $r \in \lbrace 1,2\rbrace.$ Assume that there exists  $C{>}0$ independent of $\sigma,\epsilon$ such that: $\norm{y}^{}_{_{\infty}} {\leq} C$.
Then,  even for $r=2$, the conditions \eqref{not_another_cond_2}, \eqref{not_another_cond_3} are less restrictive than \eqref{sigma-epsilon-condition_clean_2}, \eqref{sigma-epsilon-condition_clean_3}. Now, we shall indicate the dominant term in the bounds  of  Theorem \ref{Theorem_1}. Indeed,  for $d=2$, \eqref{sigma-epsilon-condition_clean_2} becomes
\begin{align}\label{sigma-epsilon-condition_clean_2_simple}
(\sqrt{k} {+} h) \norm{y}^{}_{H^{2,1}(\Omega^{}_T)}  \leq C \epsilon^{2} \ \Longrightarrow \  \sqrt{k} {+} h \leq C \epsilon^{2+r},
\end{align}
where $C$ is an algebraic constant depending only on the domain, and $\|\lambda\|_{L_{\infty}(I)}$ and  hence replacing \eqref{sigma-epsilon-condition_clean_2_simple} into the estimates \eqref{es_1_1} and \eqref{es_1_2} we obtain,
\begin{align*}
 \norm{y {-}y^{}_{\sigma}}^{}_{L_{\infty}(I;L_2(\Omega))} {+} \epsilon^{-1} \|y{-}y_\sigma\|^2_{L_{4}(I;L_4(\Omega))}  & \leq C (\sqrt{k} {+} h)\epsilon^{-r},\\
 \norm{y {-}y^{}_{\sigma}}^{}_{L^{}_{2}(I;H^{1}(\Omega))} & \leq C (\sqrt{k} {+} h)\epsilon^{-r}.
\end{align*}
For $d=3$,  we deduce that \eqref{sigma-epsilon-condition_clean_3} becomes
\begin{align}\label{sigma-epsilon-condition_clean_3_simple}
  (\sqrt{k} {+} h) \norm{y}^{4/3}_{H^{2,1}(\Omega^{}_T)}  \leq C  \epsilon^{3} \ \Longrightarrow \
  \sqrt{k} {+} h \leq C \epsilon^{3+(4r/3)},
\end{align}
and  the estimates \eqref{es_2_1} and \eqref{es_2_2} can be written as
\begin{align*}
& \norm{y {-}y^{}_{\sigma}}^{}_{L_{\infty}(I;L_2(\Omega))}{+} \epsilon^{-1} \|y{-}y_\sigma\|^2_{L_{4}(I;L_4(\Omega))}  \leq  C \max    \lbrace \epsilon^{(1/2) {-} (r/3)}
 ,1  \rbrace ( \sqrt{k} {+} h) \epsilon^{-r} \\
& \leq \begin{cases} C ( \sqrt{k} {+} h) \epsilon^{-1} \quad \quad \  \text{when} \ r = 1, \\ C (\sqrt{k} {+} h) \epsilon^{-13/6} \quad \text{when} \ r=2,
\end{cases} \leq C (\sqrt{k}{+}h) \epsilon^{-(7r-1)/6}, \\
& \norm{y {-} y^{}_{\sigma}}^{}_{L^{}_{2}(I;H^{1}(\Omega))}   \leq C ( \sqrt{k} {+} h ) \epsilon^{- (8r {+} 3) / 6 }.
\end{align*}
\end{remark}
\begin{remark}
If additional regularity is present, i.e, if $\|y\|^{}_{W^{2,1}_4(\Omega^{}_T)} \approx \epsilon^{-r}$, then observe that using the maximal parabolic regularity results of \cite{LeykekhmanVexler_2017} we may bound $\|\hat e\|^4_{L_{4}(I;L_4(\Omega))} \sim (\kappa {+} h^2)^4 \|y\|^4_{W^{2,1}_4(\Omega^{}_T)} {\sim} (\kappa {+} h^2)^4 \epsilon^{-4r}$. Therefore, we easily deduce that $A \sim  \frac{C^{2}_{\infty}}{\epsilon^{4}}  ( k{+}h^2)^2 \norm{y}^{2}_{H^{2,1}(\Omega^{}_T)}$ in both $d{=}2,3.$
As a consequence the conditionality reads as $\sqrt{k}{+}h \sim \epsilon^{2{+}(r/2)}$ for $d{=}2$ and $\sqrt{k}{+}h {\sim} \epsilon^{(11/4){+}(r/2)}$ for $d{=}3$.
\end{remark}
Theorem \ref{Theorem_1} and Remark \ref{rates} play a  crucial role in the following result.
\begin{corollary}\label{consequence_Theorem_1}
Let $u,\varv \in L_2(I;L_2(\Omega))$, $y^{}_u \in  H^{2,1}(\Omega^{}_{T}) \cap C(I;H^{1}(\Omega))$ be the solution of \eqref{state} while $y^{}_{\sigma}(\varv) \in Y^{}_{\sigma}$ the solution of \eqref{state_sigma} corresponding to the control $\varv$.  Suppose that the assumptions of  Theorems \ref{lipcon-1} and \ref{Theorem_1}  hold. In addition, let $\|y^{}_{\varv}\|^{}_{_{\infty}} \leq C$, and $\|y^{}_{\varv}\|^{}_{H^{2,1}(\Omega^{}_{T})} \leq C \epsilon^{-r}$ with $r \in \lbrace 1,2 \rbrace$, where $C$ denotes a constant that depends only on data and it is independent of $\epsilon$ and that \eqref{sigma-epsilon-condition_clean_2_simple} or \eqref{sigma-epsilon-condition_clean_3_simple} hold for $d=2$ or $3$, respectively. Then, for $d {=}2,3$ there holds
    \begin{align}\label{es_4_2}
    & \norm{ y^{}_u {-} y^{}_{\sigma}(\varv)   }^{}_{L_{\infty}(I;L_2(\Omega))}   \leq L^{}_1 \norm{u{-}\varv}^{}_{L^{}_{2}(I;L^{}_2(\Omega))}
     + C\epsilon^{-s_1} (\sqrt{k} {+}h),
\\ \label{es_4_3}
    & \norm{ y^{}_u {-} y^{}_{\sigma}(\varv)   }^{}_{L^{}_{2}(I;H^{1}(\Omega))}  \leq L^{}_1\epsilon^{-1} \norm{u{-}\varv}^{}_{L^{}_{2}(I;L^{}_2(\Omega))}
     + C\epsilon^{-s_2} (\sqrt{k}{+}h) ,
    \end{align}
where $s_1=s_2=r$ when $d=2$ and $s_1 = (7r-1)/6$, $s_2=(8r+3) / 6$ when $d=3$.
Let $u^{}_{\sigma} \in U^{}_{\sigma}$. If $u^{}_{\sigma} \rightharpoonup u$ weakly in $ L_2(I;L_2(\Omega))$ for every $\sigma$, then it holds
\begin{equation}\label{convergence_1}
\begin{cases}
  \norm{ y^{}_u {-} y^{}_{\sigma}(u^{}_{\sigma}) }^{}_{L^{}_{2}(I;H^{1}(\Omega))} \rightarrow 0, \
  \norm{ y^{}_u {-} y^{}_{\sigma}(u^{}_{\sigma}) }^{}_{L^{}_{p}(I;L^{}_2(\Omega))} \rightarrow 0  \\
  \norm{ y^{}_u(T) {-} y^{}_{\sigma}(u^{}_{\sigma})(T) }^{}_{L^{}_{2}(\Omega)} \rightarrow 0.
\end{cases}
\end{equation}
for all $1 \leq p < \infty.$
\end{corollary}
\begin{proof} Inequalities \eqref{es_4_2}, \eqref{es_4_3}, follow from \eqref{LC_1} and Remark \ref{rates} using triangle inequality.
For \eqref{convergence_1} we split, $y^{}_u {-} y^{}_{\sigma}(u^{}_{\sigma}) =  (y^{}_u {-} y^{}_{u^{}_{\sigma}}) {+} (y^{}_{u^{}_{\sigma}} {-} y^{}_{\sigma}(u^{}_{\sigma}))$. According to Lemma \ref{stability} and the boundedness of $\lbrace u^{}_{\sigma} \rbrace^{}_{\sigma}$ in $L^{}_{2}(I;L^{}_2(\Omega))$  we have that $\norm{y^{}_{u^{}_{\sigma}}}^{}_{H^{2,1}(\Omega^{}_{T})} \leq C^{}_{\text{st}} \epsilon^{-r}$, $r \in \lbrace 1,2 \rbrace.$ Then, any weakly convergent subsequence of $\lbrace y^{}_{u^{}_{\sigma}} \rbrace^{}_{\sigma}$  in $H^{2,1}(\Omega^{}_{T})$, converges to $y^{}_{u}$. Note that the compact embeddings $H^{2,1}(\Omega^{}_{T}) \subset L^{}_{2}(I;H^{1}(\Omega))$, $H^{2,1}(\Omega^{}_{T}) \subset L^{}_{p}(I;L^{}_2(\Omega))$ for $1 \leq p < \infty$ and $H^{2,1}(\Omega^{}_{T}) \hookrightarrow L^{}_2(\partial \Omega^{}_{T})$ imply that
\begin{align*}
    \norm{ y^{}_u {-} y^{}_{u^{}_{\sigma}} }^{}_{L^{}_{2}(I;H^{1}(\Omega))} {+}  \norm{ y^{}_u {-} y^{}_{u^{}_{\sigma}} }^{}_{L^{}_{
    p}(I;L^{}_2(\Omega))} {+}  \norm{ y^{}_u(T) {-} y^{}_{u^{}_{\sigma}}(T) }^{}_{L^{}_{2}(\Omega)} \rightarrow 0.
\end{align*}
For every fixed $\epsilon$ as $\sigma \rightarrow 0$, with $\kappa, h$ under the assumptions of Theorem \eqref{Theorem_1} we get,
\begin{align*}
  \norm{ y^{}_{u^{}_{\sigma}} {-} y^{}_{\sigma}(u^{}_{\sigma}) }^{}_{L_{\infty}(I;L_2(\Omega))} & {+}  \norm{ y^{}_{u^{}_{\sigma}}(T) {-} y^{}_{\sigma}(u^{}_{\sigma})(T) }^{}_{L^{}_{2}(\Omega)} \\
& \ {+} \norm{ y^{}_{u^{}_{\sigma}} {-} y^{}_{\sigma}(u^{}_{\sigma}) }^{}_{L^{}_{2}(I;H^{1}(\Omega))} \rightarrow 0,
\end{align*}
which completes the proof.
\end{proof}

The next Theorem studies the differentiability of the relation between control and discrete state. The proof follows well known techniques and it is omitted (see for instance \cite[Theorem 4.10]{Casas-Chrysafinos_2012}).
\begin{theorem}
Let $u,\varv \in L^{}_{2}(I;L^{}_2(\Omega))$. The mapping $G^{}_{\sigma}: L^{}_{2}(I;L^{}_2(\Omega)) \rightarrow Y^{}_{\sigma}$, such that $y^{}_{\sigma}= y^{}_{\sigma} (u) = G^{}_{\sigma}(u)$, is of class $C^{\infty}_{}$. We denote by $z^{}_{\sigma}(\varv) = G'^{}_{\sigma}(u) \varv$ the unique solution to the problem: For $n = 1,\ldots,N$ and for all $w^{}_h \in Y^{}_h$,
\begin{equation}\label{1st-discrete_deriv}
    \begin{aligned}
     & \Big( \frac{z^{}_{n,h} {-} z^{}_{n-1,h}}{k^{}_n}, w^{}_h \Big) {+} \left( \nabla z^{}_{n,h} , \nabla w^{}_h \right) {+} \epsilon^{-2} \left( (3 y^{2}_{n,h} {-} 1)z^{}_{n,h} ,w^{}_h\right)\\
     & \ = \frac{1}{k^{}_n} \int^{t^{}_n}_{t^{}_{n-1}} ( \varv(t), w^{}_h) \diff{t}, \quad \textit{with} \quad z^{}_{0,h}=0.
    \end{aligned}
\end{equation}
\end{theorem}
\begin{remark}
We note that existence and uniqueness of \eqref{1st-discrete_deriv} can be proved by standard techniques for any fixed $\epsilon.$ The main difficulty is to prove bounds with constants that do not depend exponentially upon $1/\epsilon$. Let us mention the absence of the cubic term and the fact that \eqref{spectral_classical}  can not being used for $y_{n,h}$ hence the recovery of stability bounds that are independent of the exponential of $1/\epsilon$ is not straightforward. A key part of the remaining of our work is to circumvent this difficulty. Our approach is demonstrated in Section 4.2 for the discrete adjoint-state equation, but can be applied to \eqref{1st-discrete_deriv} in a straightforward manner.
\end{remark}

\subsection{Analysis of the discrete adjoint state equation}
The diffrentiability properties of $G^{}_{\sigma}: L^{}_{2}(I;L^{}_2(\Omega)) \rightarrow Y^{}_{\sigma}$ imply that the reduced cost functional $J^{}_{\sigma}: L^{}_{2}(I;L^{}_2(\Omega)) \rightarrow \mathbb{R}$ is of class $C^{\infty}_{}$, as well. Applying the chain rule, we get
\begin{equation}\label{1-deriv_J_sigma}
    \begin{aligned}
     J'^{}_{\sigma}(u)\varv {=} &  \int_{I} \spaceint{(y^{}_{\sigma} {-} y^{}_{d}) z^{}_{\sigma} } \diff{t} {+} \gamma \spaceint{(y^{}_{\sigma}(T) {-} y^{}_{\Omega,h}) z^{}_{\sigma}(T)} {+} \mu \int_{I} \spaceint{u \varv} \diff{t},
    \end{aligned}
\end{equation}
To eliminate $z^{}_{\sigma}$ from \eqref{1-deriv_J_sigma} we consider  the corresponding discrete scheme of the adjoint state equation \eqref{adjoint}, reading:
For each $n=N,\ldots,1$ and for all $w^{}_h \in Y^{}_h$,
\begin{equation}\label{adjoint_state_sigma}
\begin{aligned}
    & \big( \varphi^{}_{n,h} {-} \varphi^{}_{n+1,h}, w^{}_h \big)  {+} \int_{t_{n-1}}^{t_n} \left ( \left ( \nabla \varphi^{}_{n,h}  ,  \nabla w^{}_h \right ) {+} \epsilon^{-2} \left( ( 3y^{2}_{n,h}{-}1)\varphi^{}_{n,h}, w^{}_h \right) \right ) \diff{t} \\
    & \ = \intJn{ \left( y^{}_{n,h}{-} y^{}_d(t), w^{}_{h} \right)}, \ \ \quad \text{with} \quad  \varphi^{}_{N+1,h} {=} \gamma \left(y^{}_{N,h}{-} y^{}_{\Omega,h} \right).
\end{aligned}
\end{equation}
The above (backwards in time) fully-discrete equation is understood as follows: We begin by computing $\varphi^{}_{N,h}$ using  $\varphi^{}_{N+1,h}$ and then we descent from $n{=}N$ until $n{=}1$. Note that we set $\varphi^{}_{n,h}{=} \varphi^{}_{\sigma}(t^{}_{n-1})$,  $1{\leq} n{\leq} N$.  Following identically \cite[Section 4.2]{Casas-Chrysafinos_2012}, the expression \eqref{1-deriv_J_sigma} can be written as
\begin{equation}\label{1-deriv_J_sigma_alternate}
    J'^{}_{\sigma}(u)\varv {=} \int_{I} \spaceint{\left(\varphi^{}_{\sigma} {+} \mu u \right)\varv} \diff{t}, \quad \forall \varv \in L^{}_{2}(I;L^{}_2(\Omega)).
\end{equation}
The following projection operator $R_{\sigma}$ is the analogue of projection operator $P_\sigma$ suitably modified to handle the backwards in time problem.
Let $R^{}_{\sigma}: C(I;L^{}_2(\Omega)) \rightarrow Y^{}_{\sigma}$  defined through $\left( R^{}_{\sigma} w \right)^{}_{n,h}{=} \left( R^{}_{\sigma} w \right)(t^{}_{n-1}){=} P^{}_h w (t^{}_{n-1})$, $1 \leq n \leq N$. If $w \in H^{2,1}_{}(\Omega^{}_{T}) \cap C(I;H^{1}(\Omega))$ there exists $C>0$ such that:
 \begin{align}\label{approx_R_1}
     \norm{w {-} R^{}_{\sigma}w}^{}_{L^{}_{2}(I;L^{}_2(\Omega))} & {\leq} C \left( k \norm{w^{}_t}^{}_{L^{}_{2}(I;L^{}_2(\Omega))} {+} h^2 \norm{w}^{}_{L^{}_{2}(I;H^{2}_{}(\Omega))} \right),\\\label{approx_R_2}
     \norm{w {-} R^{}_{\sigma}w}^{}_{L^{}_{2}(I;H^{1}(\Omega))} & \leq C \left( \sqrt{k} \norm{w^{}_t}^{}_{L^{}_{2}(I;L^{}_2(\Omega))} {+} h \norm{w}^{}_{L^{}_{2}(I;H^{2}_{}(\Omega))} \right),
 \\\label{approx_R_3}
     \norm{w {-} R^{}_{\sigma}w}^{}_{L_{\infty}(I;L_2(\Omega))} &\leq C \left( \sqrt{k} \norm{w^{}_t}^{}_{L^{}_{2}(I;L^{}_2(\Omega))} {+} h \norm{w}^{}_{L^{}_{\infty}(I;H^{1}_{}(\Omega))} \right).
 \end{align}
The following Lemma provides the basic stability estimates for \eqref{adjoint_state_sigma}. We note that our approach avoids assumptions regarding the construction of a discrete approximation of the spectral estimate.
\begin{lemma} \label{dstability_adjoint}
Let $u \in L_2(I;L_2(\Omega))$, $y^{}_u \in  H^{2,1}(\Omega^{}_{T}) \cap C(I;H^{1}(\Omega))$ be the solution of \eqref{state} while $y^{}_{\sigma}(u){: =} y^{}_\sigma \in Y^{}_{\sigma}$ the solution of \eqref{state_sigma} corresponding to the control $u.$ Suppose that the assumptions of  Theorem \ref{lipcon-1} hold.  Let  $C^{}_{\infty} {\sim} \|y_u\|_{_{\infty}},$  and $\|y_u\|^{}_{H^{2,1}(\Omega^{}_{T})} \leq C \epsilon^{-r}$ with $r \in \lbrace 1,2 \rbrace$, where $C$ denotes a constant that depends only on data and it is independent of $\epsilon$ and  $C_\zeta {:=} \exp \left ( \int_{I} 2 \lambda (t) (1{-}\epsilon^2) {+} 3 \diff{t} \right ),$ and that \eqref{sigma-epsilon-condition_clean_2_simple} or \eqref{sigma-epsilon-condition_clean_3_simple} hold for $d=2$ or $3$, respectively. If in addition,
\begin{equation} \label{condition2_1} \sqrt{k} + h  \leq \frac{C \epsilon^{q_d}}{C_{\infty} C_\zeta}
\end{equation}
where $q_2=2+r$ when $d=2$ and $q_3=3+(4r/3)$ when $d=3$, 
then there exists $D^{\text{dG}}_{\text{st},1} >0$ (independent of  $\sigma=(k,h)$ and $\epsilon$)  such that:
\begin{equation*}
\begin{aligned}
& \|\varphi_\sigma \|_{L_2(I;H^1(\Omega))} + \epsilon^{-1} \|y^{}_{\sigma}\varphi_\sigma \|_{L_2(I;L_2(\Omega))} \leq D^{\text{dG}}_{\text{st},1} \epsilon^{-1}.
\end{aligned}
\end{equation*}
Here, $ D^{\text{dG}}_{\text{st},1} : = C \left (  \gamma \| y_{N,h} {-} y_{\Omega,h} \|_{L_2(\Omega)} + \|y_\sigma {-} y_d\|_{L_2(I;L_2(\Omega))} \right )$ with $C$ denoting an algebraic constant independent of $\epsilon.$
\end{lemma}
\begin{proof}
Standard arguments imply existence and uniqueness of solution $\varphi_\sigma$ of \eqref{adjoint_state_sigma}.
As usual, we need to develop stability bounds with constants independent of $1/\epsilon$. For this purpose, we employ a duality argument. Given right-hand side  $\varphi_\sigma$, we define $\zeta^{}_{\sigma} \in Y_\sigma$ such that for $n{=}1,\cdots,N$, and for $w_h \in Y_\sigma$,
\begin{equation} \label{zeta}
\begin{aligned}
    & \big( \zeta^{}_{n,h} {-} \zeta^{}_{n-1,h}, w^{}_h \big )  {+} \int_{t_{n-1}}^{t_n} \left ( \left( \nabla \zeta^{}_{n,h}  ,  \nabla w^{}_h \right) {+} \epsilon^{-2} \left( ( 3y^{2}{-}1)\zeta^{}_{n,h}, w^{}_h \right) \right ) \diff{t} \\
    & \ {=} \intJn{ \left( \varphi_{n,h}, w^{}_{h} \right)}, \quad \text{with}
    \quad  \zeta^{}_{0,h} {=} 0.
\end{aligned}
\end{equation}
Setting $w_h {=} \zeta^{}_{n,h}$ and using \eqref{spectral_classical}, we easily deduce that
\begin{equation}\label{zeta1}
\begin{aligned}
& \|\zeta_\sigma\|_{L_{\infty}(I;L_2(\Omega))}{+}  \| \zeta^{}_{\sigma} \|_{L_2(I;H^1(\Omega))} {+} \|y \zeta^{}_{\sigma} \|_{L_2(I;L_2(\Omega))} \leq C_\zeta \| \varphi^{}_{\sigma} \|_{L_2(I;L_2(\Omega))},
\end{aligned}
\end{equation}
where $C_\zeta{: =}  C \exp \left ( \int_{I} 2 \lambda (t) (1{-}\epsilon^2) {+} 3{+}  2 \epsilon^2  \diff{t} \right )$ (with  $C$ an algebraic constant depending on the domain).
Setting now $w_h {=} (\zeta_{n,h} {-} \zeta_{n-1,h})/k_n$ into \eqref{zeta}, and using H\"older and Young's inequalities, we obtain:
\begin{align*}
& \frac{1}{k_n} \Big \|  \zeta_{n,h} {-} \zeta_{n{-}1,h} \Big \|^2_{L_2(\Omega)} {+}  \left( \| \nabla z_{n,h} \|^2_{L_2(\Omega)} {-}  \| \nabla z_{n{-}1,h} \|^2_{L_2(\Omega)} {+} \| \nabla ( \zeta_{n,h} {-} \zeta_{n{-}1,h} ) \|^2_{L_2(\Omega)} \right) \\
& \leq \left ( \frac{3 \norm{y}^2_{_{\infty}} {+} 1}{\epsilon^2} \ \| \zeta_{n,h} \|_{L_2(\Omega)} {+} \| \varphi_{n,h} \|_{L_2(\Omega)} \right) \Big \|  \zeta_{n,h}{-} \zeta_{n-1,h} \Big \|_{L_2(\Omega)}  \\
& \leq \frac{1}{2 k_n} \Big \| \zeta_{n,h}{-} \zeta_{n-1,h}  \Big \|^2_{L_2(\Omega)} {+} \frac{(3 \norm{y}^2_{_{\infty}} {+}1)^2}{\epsilon^4} k_n\|\zeta_{n,h} \|^2_{L_2(\Omega)} {+} k_n \| \varphi_{n,h} \|^2_{L_2(\Omega)}.
\end{align*}
Summing the above inequalities and using \eqref{zeta1} we derive the estimate,
\begin{equation} \label{zeta2}
 \| \zeta_\sigma \|_{L_{\infty}(I;H^1(\Omega))}  \leq C_\zeta \Big(\frac{C^2_{\infty}{+}1}{\epsilon^2}{+}1 \Big) \| \varphi_\sigma \|_{L_2(I;L_2(\Omega))},
\end{equation}
upon setting $C^2_{\infty}:= 3\norm{y}^2_{_{\infty}}$ for brevity. Now we proceed with the duality argument.  Setting $w_h {=} \zeta_{n,h}$ into \eqref{adjoint_state_sigma} we obtain that
\begin{align*}
& \Big( \varphi^{}_{n,h} {-} \varphi^{}_{n+1,h}, \zeta_{n,h} \Big)  {+} \intJn{  \left ( \left( \nabla \varphi^{}_{n,h}  ,  \nabla \zeta_{n,h} \right) {+} \epsilon^{-2} \left( ( 3y^{2}_{n,h}{-}1)\varphi^{}_{n,h}, \zeta_{n,h} \right ) \right)} \\
    & \ =  \intJn{ \left( y^{}_{n,h}{-}y^{}_d(t), \zeta_{n,h} \right)}.
\end{align*}
Similarly setting $w_h {=} \varphi_{n,h}$ into \eqref{zeta} respectively
\begin{align*}
& \Big( \zeta^{}_{n,h} {-} \zeta^{}_{n-1,h}, \varphi_{n,h} \Big) {+} \intJn{ \left ( \left( \nabla \varphi^{}_{n,h}  ,  \nabla \zeta_{n,h} \right) {+} \epsilon^{-2} \left( ( 3y^{2}{-}1)\zeta^{}_{n,h}, \varphi_{n,h} \right ) \right)}  \\
& \ {=}  \intJn{ \|\varphi_{n,h} \|^2_{L_2(\Omega)} }.
\end{align*}
Subtracting the last two equalities and summing them from $1$  up to $N$, we deduce,
\begin{equation}\label{zeta3}
\begin{aligned}
 & \| \varphi_\sigma \|^2_{L_2(I;L_2(\Omega))} {=}  ( \zeta_{N,h}, \varphi_{N+1,h} )
{+} \int_{I} \left ( y_\sigma {-} y_d, \zeta_\sigma \right ) \diff{t} {+} \frac{3}{\epsilon^2} \int_{I}  \left((y^2 {-} y^2_\sigma) \varphi_\sigma ,\zeta_\sigma \right) \diff{t} ,
\end{aligned}
\end{equation}
where we have used the definition $\zeta_{0,h}=0$ and the calculation
$$\sum_{n=1}^N \Big \lbrace \left( \zeta^{}_{n,h}{-} \zeta^{}_{n-1,h}, \varphi^{}_{n,h} \right) {-}  \left( \varphi^{}_{n,h} {-} \varphi^{}_{n+1,h}, \zeta^{}_{n,h} \right)  \Big \rbrace {=}  ( \zeta_{N,h}, \varphi_{N+1,h} ).$$
Therefore, we may bound the terms of the right-hand side of \eqref{zeta3}, as follows: Using the estimates of \eqref{zeta1}, and Young's inequality,
\begin{align} \label{basicfirst}
 & \int_{I} \left ( y_\sigma {-} y_d, \zeta_\sigma \right ) \diff{t} {+} ( \zeta_{N,h}, \varphi_{N+1,h} ) \\
 & \ \leq \|y_\sigma {-} y_d \|^{}_{L_2(I;L_2(\Omega))} \| \zeta_\sigma \|^{}_{L_2(I;L_2(\Omega))} {+} \| \zeta_{N,h} \|^{}_{L^{}_2(\Omega)} \| \varphi_{N+1,h} \|^{}_{L^{}_2(\Omega)} \nonumber \\
 & \ \leq  \frac{1}{4} \|\varphi_\sigma \|^2_{L_2(I;L_2(\Omega))} {+} C^2_\zeta \left (  \|y^{}_{\sigma}{-}y_d \|^2_{L_2(I;L_2(\Omega))} {+} \| \varphi_{N+1,h} \|^2_{L^{}_2(\Omega)} \right ). \nonumber
 \end{align}
 For the third term of \eqref{zeta3}, the identity $y^2{-}y^2_\sigma = 2y(y{-}y_\sigma) {-} (y{-}y_\sigma)^2$ yields,
 \begin{equation}\label{basic}
\begin{aligned}
\frac{3}{\epsilon^2} \int_{I} \left( (  y^2 {-} y^2_\sigma)) \varphi_\sigma, \zeta_\sigma  \right)\diff{t}  &= \frac{6}{\epsilon^2} \int_{I} \left( y (y{-}y_\sigma )   \varphi_\sigma, \zeta_\sigma \right) \diff{t} \\
& \ {-} \frac{3}{\epsilon^2}
\int_{I} \left (  - (y{-}y_\sigma)^2 y_\sigma \varphi_\sigma, \zeta_\sigma \right ) \diff{t}.
\end{aligned}
\end{equation}
Once again we need to distinguish the cases $d{=}2$ and $d{=}3$.
For $d{=}3$ we work as follows: Using H\"older and Young's inequalities (with an appropriate $\delta_1>0$ to be chosen later), estimate \eqref{zeta1} and Theorem \ref{Theorem_1} - Remark \ref{rates} we may bound the first term of \eqref{basic},
\begin{align*}
& \frac{6}{\epsilon^2} \int_{I} \left( (y{-}y_\sigma )
y \varphi_\sigma, \zeta_\sigma \right) \diff{t} \leq  \frac{6}{\epsilon^2} \int_{I} \norm{y {-} y^{}_{\sigma}}^{}_{L^{}_6(\Omega)} \|y\|^{}_{L^{}_{\infty}(\Omega)} \norm{ \varphi_\sigma}^{}_{L^{}_3(\Omega)} \norm{\zeta_\sigma}^{}_{L^{}_2(\Omega)} \diff{t} \\
& \leq  \frac{3\tilde{c}^2 C^2_{\infty} }{2 \epsilon^5 \delta_1} \norm{ \zeta^{}_{\sigma}}^{2}_{L_{\infty}(I;L_2(\Omega))} \norm{ y {-} y^{}_{\sigma}}^{2}_{L^{}_{2}(I;H^{1}(\Omega))} {+} \epsilon \delta_1 \norm{ \varphi_\sigma}^{2}_{L^{}_2(I;L^{}_3(\Omega))}\\
& \leq \frac{3 C^2_{\infty} C^2_\zeta}{2 \delta_1 \epsilon^5} \frac{(k {+}h^2)}{\epsilon^{(8r/3){+}1}} \|\varphi_\sigma \|^2_{L_2(I;L_2(\Omega))}  {+}  \frac{1}{8} \norm{ \varphi_\sigma}^{2}_{L_2(I;L_2(\Omega))} {+} 2 \tilde{c}^2 \epsilon^2 \delta^2_1\| \varphi_\sigma \|^2_{L^{}_2(I;H^1(\Omega))} .
\end{align*}
Choosing $k,h$ such that
\begin{align}
    \frac{3 C^2_{\infty} C^2_\zeta (k {+}h^2) }{2 \delta_1 \epsilon^{6+(8r/3)}} \leq \frac{1}{8},
\end{align}
 we obtain,
\begin{equation} \label{phi1}
  \frac{6}{\epsilon^2} \int_{I} \left( (y{-}y_\sigma )  y \varphi_\sigma, \zeta_\sigma \right) \diff{t} {\leq} \frac{1}{4} \|\varphi_\sigma \|^2_{L_2(I;L_2(\Omega))} {+} 2 \tilde{c}^2 \epsilon^2 \delta^2_1  \|  \varphi_\sigma \|^2_{L^{}_2(I;H^1(\Omega))}.
\end{equation}
For the second term of \eqref{basic}, we proceed as follows: For a suitable chosen $\delta_2 >0$ (to be determined later) we use H\"older and Young's inequalities, the Gagliardo-Nirenberg inequality \eqref{GNL2}, estimates \eqref{zeta1}, \eqref{zeta2}  and the estimates of Theorem \ref{Theorem_1} - Remark \ref{rates} for $y-y_\sigma$ to get,
\begin{align*}
 & \frac{3}{\epsilon^2} \int_{I} \left( (y{-}y_\sigma )^2
 \varphi_\sigma, \zeta_\sigma \right) \diff{t} \leq  \frac{3}{\epsilon^2} \int_{I} \norm{y{-} y^{}_{\sigma}}^{2}_{L^{}_4(\Omega)} \norm{\varphi_\sigma}^{}_{L^{}_6(\Omega)} \norm{\zeta_\sigma}^{}_{L^{}_3(\Omega)} \diff{t} \\
& \leq  \frac{9 \tilde{c}^2}{\delta_2 \epsilon^6} \norm{ \zeta^{}_{\sigma}}_{L^{}_{\infty}(I;H^1(\Omega))} \norm{ \zeta^{}_{\sigma}}_{L_{\infty}(I;L_2(\Omega))} \norm{ y{-}y^{}_{\sigma}}^{4}_{L^{}_{4}(I;L_4(\Omega))} \\
& \quad {+} \delta_2 \epsilon^2 \norm{ \varphi_\sigma}^{2}_{L^{}_2(I;H^1(\Omega))}\\
& \leq \frac{9 \tilde{c}^2 }{\delta_2 \epsilon^6}  C^2_\zeta \Big(\frac{C^2_{\infty} {+}1}{\epsilon^2} {+}1 \Big )   \epsilon^2 \frac{(k{+}h^2)}{\epsilon^{(7r{-}1)/3}}  \|\varphi_\sigma \|^2_{L_2(I;L_2(\Omega))}  {+}  \delta_2 \epsilon^2 \norm{\varphi_\sigma}^{2}_{L^{}_2(I;H^1(\Omega))}.
\end{align*}
Choosing $\kappa$, $h$ such that
\begin{equation} \label{restrictionzeta_2a}
\left ( \frac{9}{\delta_2} \right ) \frac{\tilde{c}^2}{\epsilon^4}  C^2_\zeta \Big(\frac{C^2_{\infty} {+}1}{\epsilon^2} {+}1 \Big)  \frac{(k +h^2)}{\epsilon^{(7r-1)/3}}  \leq \frac{1}{4},
\end{equation}
there holds that
\begin{equation} \label{phi2}
\frac{3}{\epsilon^2} \int_{I} \abs{ \left ( (y{-}y_\sigma) y_\sigma \varphi_\sigma, \zeta_\sigma \right ) } \diff{t} {\leq}\frac{1}{4} \| \varphi_\sigma \|^2_{L_2(I;L_2(\Omega))} {+} \delta_2 \epsilon^2 \|\varphi_\sigma \|^2_{L^{}_2(I;H^1(\Omega))}.
\end{equation}
Substituting \eqref{phi1}, \eqref{phi2} into \eqref{basic} we get,
\begin{equation} \label{phi3}
\begin{aligned}
& \frac{3}{\epsilon^2} \int_{I} \left( (  y^2 {-} y^2_\sigma)) \varphi_\sigma, \zeta_\sigma  \right)  \diff{t}
 {\leq} \frac{1}{2} \| \varphi_\sigma \|^2_{L_2(I;L_2(\Omega))} {+} ( 2\tilde{c}^2 \delta^2_1 {+} \delta_2 ) \epsilon^2 \| \varphi_\sigma \|^2_{L^{}_2(I;H^1(\Omega))}.
\end{aligned}
\end{equation}
Thus, returning back to \eqref{zeta3}, to substitute \eqref{basicfirst} and \eqref{phi3}, it yields,
\begin{equation} \label{zeta4}
 \begin{aligned}
& \frac{1}{4} \| \varphi_\sigma \|^2_{L_2(I;L_2(\Omega))} \leq   (2 \tilde{c}^2 \delta^2_1 + \delta_2 ) \epsilon^2
\|  \varphi_\sigma \|^2_{L^{}_2(I;H^1(\Omega))} \\
& \ + C^2_{\zeta} \left (  \|y_\sigma - y_d \|^2_{L_2(I;L_2(\Omega))} + \| \varphi_{N+1,h} \|^2_{L^{}_2(\Omega)} \right ).
 \end{aligned}
 \end{equation}
Next we employ a boot-strap argument. First, we return to \eqref{adjoint_state_sigma} and set $w_h = \varphi^{}_{n,h}$ to deduce, after adding to both sides $\int_{t_{n-1}}^{t_n} \|\varphi_{n,h} \|^2_{L_2(\Omega)} \diff{t}$ summing the resulting equalities from $1$ to $N$, and standard algebra,
\begin{align*}
& (1/2) \| \varphi_{0,h} \|^2_{L^{}_2(\Omega)} + \| \varphi_\sigma \|^2_{L^{}_2(I;H^1(\Omega))}{+} 3\epsilon^{-2}  \| y_\sigma \varphi_\sigma \|^2_{L_2(I;L_2(\Omega))} \\
& \leq  ( (3/2) {+} \epsilon^{-2}) \| \varphi_\sigma \|^2_{L_2(I;L_2(\Omega))} {+} (1/2) \|y_\sigma {-} y_d \|^2_{L_2(I;L_2(\Omega))} {+} (1/2) \| \varphi_{N+1,h} \|^2_{L^{}_2(\Omega)}.
\end{align*}
We apply \eqref{zeta4} to substitute the first term of the right-hand side,
\begin{align*}
& \| \varphi_{0,h} \|^2_{L^{}_2(\Omega)} {+} 2 \| \varphi_\sigma \|^2_{L_2(I;H^1(\Omega))}{+} 6 \epsilon^{-2} \| y_\sigma \varphi_\sigma \|^2_{L_2(I;L_2(\Omega))} \\
&  \leq   ( 3{+} 2\epsilon^{-2} ) 4 ( 2 \tilde{c}^2 \delta^2_1 {+} \delta_2 ) \epsilon^2   \|\varphi_\sigma \|^2_{L^{}_2(I;H^1(\Omega))} \\
&  \ \  {+}  \left ( ( 3 {+} 2\epsilon^{-2} )  4 C^2_{\zeta} {+} 1 \right ) \left (  \|y_\sigma {-} y_d \|^2_{L_2(I;L_2(\Omega))} {+} \| \varphi_{N+1,h} \|^2_{L^{}_2(\Omega)} \right ),
\end{align*}
Choosing $\delta_1, \delta_2>0$ such that $\left ( 3 {+} \frac{2}{\epsilon^2} \right ) 4 ( 2\tilde{c}^2 \delta^2_1 {+} \delta_2 ) \epsilon^2  {\leq} 1$,  we finally deduce the desired estimate. For the estimate
For $d{=}2$, we work analogously.
\end{proof}
\begin{remark} \label{final1}
We note that under the condition \eqref{condition2_1}, the constant $D^{dG}_{st,1}$ is bounded independent of $\epsilon$. Indeed, adding and subtracting $y(T)$, and using the estimate of Theorem \ref{Theorem_1} and Remark \ref{rates} and triangle inequality, 
\begin{align*}
D^{dG}_{st,1} & \leq C \left ( \gamma \|y_{N,h}-y(T) \|_{L_2(\Omega)} + \gamma \|y(T) - y_{\Omega,h}\|_{L_2(\Omega)} 
+ \|y_\sigma - y_d\|_{L_2(I;L_2(\Omega))} \right ) \\
& \leq C \left ( \epsilon^{-(7r-1)/6} (\sqrt{k}+h) + \left ( \gamma \|y(T) - y_{\Omega,h}\|_{L_2(\Omega)} 
+ \|y_\sigma - y_d\|_{L_2(I;L_2(\Omega))} \right ) \right ).
\end{align*}
Substituting now \eqref{condition2_1} we obtain that the first term can be bounded independent of $1/\epsilon$, from Lemma \ref{stability} we have that the second term is bounded independent of $\epsilon$ when $ \epsilon \|\nabla y_0\|_{L_2(\Omega)} + \|y_0\|^2_{L_4(\Omega)}$ is bounded independent of $\epsilon$.  For the third term the estimate \eqref{basic_dstatestability} implies the desired bound.  
\end{remark}
 \begin{theorem}\label{Theorem_2}
Let $u \in L^{}_{2}(I;L^{}_2(\Omega))$ and $y,\varphi \in H^{2,1}_{}(\Omega^{}_T) \cap C(I;H^{1}(\Omega))$ be the associated state solution of \eqref{state} and the  adjoint state solution of \eqref{adjoint}, respectively. Let, $y^{}_{\sigma}$ be the associated discrete state solution of \eqref{state_sigma} while $\varphi^{}_{\sigma}$ the associated discrete adjoint state solution of \eqref{adjoint_state_sigma}.  Then, under the assumptions of Theorem \ref{Theorem_1} and Lemma \ref{dstability_adjoint} there exists $\Hat{C}_d{>}0$ such that for $r \in \lbrace 1,2\rbrace$ there holds that:
  \begin{equation}\label{es_8_1}
    \norm{\varphi {-} \varphi^{}_{\sigma}}^{}_{L_{\infty}(I;L_2(\Omega))}  {+} \epsilon \norm{\varphi {-}\varphi^{}_{\sigma}}^{}_{L^{}_{2}(I;H^{1}(\Omega))} \leq \hat C_{d}  \epsilon^{- {\rho}_d} (\sqrt{k} + h) 
\end{equation}
where $\rho_2 {=} 4{+}r$ when $d=2$ and $\rho_3 {=}4 {+} (7r{-}1)/6$ when $d=3$.
Here, we denote by  $E {:=} \exp(T \alpha^{})$  where $\alpha^{} {:=}\sup^{}_{t \in I} (2\lambda(t) (1{-}\epsilon^2) {+}  5 {+} \epsilon^2)$ and by ${\hat C}_{d}$, 
\hspace{-6mm}
 $\Hat{C}_{3} {:=} C(E^{1/2} \left( (C^{dG}_{st,2} {+} C_{st,2} ) D^{\text{dG}}_{\text{st},1}\right),$ 
 $\Hat{C}_{2} {:=} C(E^{1/2} \left( \epsilon (C^{dG}_{st,2} {+} C_{st,2} ) D^{\text{dG}}_{\text{st},1}\right).$ 

 \end{theorem}
 \begin{proof} We split the total error as follows:
 \begin{equation}\label{error_decomposition_2}
     e {:=} \varphi {-} \varphi^{}_{\sigma}{=} (\varphi {-} R^{}_{\sigma} \varphi) {+} (R^{}_{\sigma} \varphi {-} \varphi^{}_{\sigma}){=} \eta {+} \ \xi^{}_{\sigma}.
 \end{equation}
 For each $n{=}0,\ldots,N{-}1$ we have,
 \begin{align*}
     & \eta(t^{}_n) {=} \varphi(t^{}_n) {-}  (R^{}_{\sigma} \varphi)(t^{}_n) {=} \varphi(t^{}_n) {-} \left( R^{}_{\sigma} \varphi \right)^{}_{n+1,h}{=} \varphi(t^{}_n) {-} P^{}_h \varphi(t^{}_n)\\
     & \xi^{}_{\sigma} (t_n){=} (R^{}_{\sigma} \varphi)(t^{}_n) {-} \varphi^{}_{\sigma}(t^{}_n) {=} \left( R^{}_{\sigma} \varphi \right)^{}_{n+1,h}{-} \varphi^{}_{n{+}1,h} \equiv \xi^{}_{n{+}1,h}.
 \end{align*}
 For $n=N$, note that  $(R^{}_{\sigma} \varphi)^{}_{N+1,h} {=} P^{}_h \varphi(T)$ and $\varphi^{}_{N+1,h} {=} \gamma (y^{}_{N,h}{-} y^{}_{\Omega,h} )$.
 From  \eqref{adjoint} and \eqref{adjoint_state_sigma}, we deduce  for each $n=N,\ldots,1$ and for all $w^{}_h \in Y^{}_h$ that
 \begin{equation} \label{errorphi}
     \begin{aligned}
         &\left( e(t^{}_{n{-}1}){-}e(t^{}_{n}), w^{}_h \right) {+} \intJn{\left(\nabla e(t), \nabla w^{}_h\right)} {+} \epsilon^{-2} \intJn{\left( (3y^2{-}1)\varphi ,w^{}_h \right)}\\
         & {-} \epsilon^{-2} \intJn{\left( (3y^2_{n,h}{-}1)\varphi^{}_{n,h} ,w^{}_h \right)}{=} \intJn{\left( y(t) {-} y^{}_{n,h}, w^{}_h \right)}.
     \end{aligned}
 \end{equation}
Setting $w^{}_h = \xi^{}_{n,h}$, using \eqref{error_decomposition_2} and $(\eta(t^{}_n),\xi^{}_{n,h})= 0$, and adding and subtracting appropriate terms, \eqref{errorphi} yields,
 \begin{equation*}
    \begin{aligned}
    &\left( \xi^{}_{n,h}{-} \xi^{}_{n+1,h}, \xi^{}_{n,h} \right) {+} \intJn{ \norm{\nabla \xi^{}_{n,h}}^2_{L^{}_2(\Omega)}} {+} \epsilon^{-2} \intJn{\left( (3y^2{-}1) ( \varphi {-} \varphi^{}_{n,h} ),\xi^{}_{n,h} \right)} \\
& {+} 3\epsilon^{-2} \intJn{\left( (y^2
    {-} y^2_{\sigma})  \varphi^{}_{n,h} ,\xi^{}_{n,h} \right)} {=}  \intJn{\left( y{-}y^{}_{\sigma}, \xi^{}_{n,h}\right)} {-} \intJn{\left( \nabla \eta, \nabla \xi^{}_{n,h}\right)}.
    \end{aligned}
 \end{equation*}
Relation \eqref{error_decomposition_2}, and standard algebraic manipulations imply that
 \begin{equation*}
     \begin{aligned}
     & \quad   \frac{1}{2} \norm{\xi^{}_{n,h}}^2_{L^{}_2(\Omega)} {-}   \frac{1}{2} \norm{\xi^{}_{n+1,h}}^2_{L^{}_2(\Omega)} {+}  \frac{1}{2} \norm{\xi^{}_{n,h}{-}\xi^{}_{n+1,h} }^2_{L^{}_2(\Omega)}  \\
     & \quad   {+} \intJn{\left( \norm{\nabla \xi^{}_{n,h}}^2_{L^{}_2(\Omega)} {+} \epsilon^{-2} \left((3y^2{-}1) \xi^{}_{n,h} ,\xi^{}_{n,h} \right)\right) } \\
     & {=}    \intJn{\left( y-y^{}_{\sigma}, \xi^{}_{n,h}\right)} {-} \intJn{\left( \nabla \eta, \nabla \xi^{}_{n,h}\right)} {-} \epsilon^{-2} \intJn{ \left( (3y^2{-}1)\eta , \xi^{}_{n,h} \right)} \\
     & \quad {-} 3\epsilon^{-2} \intJn{ \left(  ( y {-} y^{}_{\sigma}) (y^{}_{\sigma}{+} y) \varphi^{}_{n,h}, \xi^{}_{n,h} \right) } {:=} \sum^{4}_{j=1} \mathcal{T}^{}_j.
    \end{aligned}
 \end{equation*}
Applying H\"older and Young's inequalities, we easily get
\begin{align*}
    \mathcal{T}^{}_1 & {\leq} \frac{1}{2} \intJn{ \norm{ y - y^{}_{\sigma} }^{2}_{L^{}_2(\Omega)} }  {+} \frac{1}{2} \intJn{ \norm{ \xi^{}_{n,h} }^{2}_{L^{}_2(\Omega)} }, \\
    \mathcal{T}^{}_2 & {\leq} \frac{1}{\epsilon^2} \intJn{ \norm{\nabla \eta }^{2}_{L^{}_2(\Omega)} } {+} \frac{\epsilon^2}{4} \intJn{ \norm{ \nabla \xi^{}_{n,h} }^{2}_{L^{}_2(\Omega)} },\\
     \mathcal{T}^{}_3 & {\leq}   \intJn{\|y \xi_{n,h}\|^2_{L_2(\Omega)}} {+} \frac{1}{4\epsilon^{4}} \intJn{\left( 9 \norm{y}^2_{L^{}_{\infty}(\Omega)} {+} 1 \right) \norm{\eta}^{2}_{L^{}_2(\Omega)} } {+}  \intJn{ \norm{\xi^{}_{n,h}}^{2}_{L^{}_2(\Omega)} }.
\end{align*}
For the term $\mathcal{T}^{}_4$, we distinguish two cases. For $d=3$, we note that using H\"older and Young's inequalities, and the stability estimates of $y,y_\sigma$ and $\varphi_\sigma$, we deduce,
\begin{align*}
    & \mathcal{T}_4   \leq \frac{3}{ \epsilon^2} \intJn{ \norm{ y {-} y^{}_{\sigma}}^{}_{L^{}_2(\Omega)} \| \varphi_{n,h} \|_{L^{}_6(\Omega)} \left ( \norm{ y_\sigma }^{}_{L^{}_6(\Omega)} {+} \|y\|_{L^{}_{6}(\Omega)} \right ) \| \xi_{n,h} \|^{}_{L^{}_6(\Omega)}   } \\
    &  \leq \frac{9C(C_{st,2} {+} C^{dG}_{st,2})^2 }{\epsilon^6} \intJn{ \|y {-} y^{}_{\sigma}\|^2_{L^{}_2(\Omega)} \|\varphi_{n,h} \|^2_{L^{}_ 6(\Omega)}}  + \frac{\epsilon^2}{4}  \intJn{ \| \xi_{n,h} \|^2_{H^1(\Omega)}},\\
    & \leq \frac{9C(C_{st,2} {+} C^{dG}_{st,2})^2}{\epsilon^6}  \| y{-} y_\sigma \|^2_{L_{\infty}(I;L_2(\Omega))} \intJn{ \| \varphi_{n,h} \|^2_{L^{}_6(\Omega)}}  {+} \frac{\epsilon^2}{4}  \intJn{ \| \xi_{n,h} \|^2_{H^1(\Omega)}},
  \end{align*}
where $C$ is a constant depending only on the domain. Using the spectral estimate \eqref{spectral_classical} for $\varv = \xi_{n,h}$, collecting the above bounds, adding to both sides $(\epsilon^2 /2)  \int_{t_{n-1}}^{t_n} \| \xi_{n,h} \|^2_{L^2(\Omega)} \diff{t}$ summing from $n=m$ up to $N$ where $1\leq m \leq N$, using a standard (linear) Gronwall Lemma, for $\sup^{}_{n=1, \ldots,   N} \alpha^{} k^{}_n < 1$, where $\alpha := \sup^{}_{t \in I} ( 2 \lambda (t)(1-\epsilon^2) + 5 +\epsilon^2 )$, we deduce that
\begin{equation*}
    \begin{aligned}
    & \norm{\xi^{}_{m,h}}^2_{L^{}_2(\Omega)} {+} \sum_{n=m}^N \norm{\xi^{}_{n,h} {-} \xi^{}_{n+1,h} }^2_{L^{}_2(\Omega)}   {+} \epsilon^2 \sum_{n=m}^N k^{}_n \norm{ \xi^{}_{n,h}}^2_{H^1(\Omega)}\\
    & \leq E \Bigg \{  \norm{ \xi^{}_{N+1,h}}^2_{L_2(\Omega)} {+} \int_{t_m}^T \left( \frac{  9\norm{ y }^{2}_{L^{}_{\infty}(\Omega)} {+}1  }{2\epsilon^{4}}
    \norm{ \eta}^{2}_{L^{}_{2}(\Omega)} {+} \frac{2}{\epsilon^2}
    \norm{ \nabla \eta}^{2}_{L^{}_{2}(\Omega)}  \right) \diff{t} \\
    & \ {+}\int_{t_m}^T  \norm{ y{-} y^{}_{\sigma}}^{2}_{L^{}_{2}(\Omega)} \diff{t} {+} \frac{C(C_{st,2} {+} C^{dG}_{st,2})^2}{\epsilon^6}  \| y{-} y_\sigma \|^2_{L_{\infty}(I;L_2(\Omega))} \int_{t_m}^T \| \varphi_{\sigma} \|^2_{L_6(\Omega)} \diff{t}    \Bigg \},
    \end{aligned}
\end{equation*}
where $E^{}_{} {:=} \exp( T \alpha^{}).$ The estimate follows using the embedding $H^1_{}(\Omega) \subset L^{}_6(\Omega)$ and Lemma \ref{dstability_adjoint} to bound $\int_{t_m}^T \| \varphi_\sigma \|^2_{L_6(\Omega)} \diff{t} \leq (D^{dG}_{st,1} / \epsilon )^2$ and Remark \ref{rates} to bound  $ \| y{-}y_\sigma \|^2_{L_{\infty}(I;L_2(\Omega))} \leq \frac{C}{\epsilon^{2r}} (\kappa {+} h^2)$ for $d=2$ and $ \| y{-}y_\sigma \|^2_{L_{\infty}(I;L_2(\Omega))} \leq \frac{C}{\epsilon^{(7r{-}1)/3}} (\kappa {+} h^2)$ for $d=3.$  Note that the term $\frac{2}{\epsilon^2} \int_{t_m}^T \| \nabla \eta \|^2_{L_2(\Omega)} \diff{t} \leq \frac{C}{\epsilon^2} (k+h^2) \|\phi \|^2_{H^{2,1}(\Omega_T)}$ dominates all $\eta$ terms. Although,   the last term of the right-hand side dominates all terms. For the two dimensional case, in order to use that bound (independent of $\epsilon$) of $\|y_\sigma \|_{L_{\infty}(I;L_4(\Omega))}$, we split ${\mathcal T}_4$ as follows:
\begin{align*}
    & \mathcal{T}_4   \leq \frac{3}{ \epsilon^2} \intJn{ \norm{ y {-} y^{}_{\sigma}}^{}_{L^{}_2(\Omega)} \| \varphi_{n,h} \|_{L^{}_8(\Omega)} \left ( \norm{ y_\sigma }^{}_{L^{}_4(\Omega)} {+} \|y\|_{L^{}_{4}(\Omega)} \right ) \| \xi_{n,h} \|^{}_{L^{}_8(\Omega)}   } \\
    & \leq \frac{9C(C_{st,2} {+} C^{dG}_{st,2})^2}{\epsilon^6}  \| y{-} y_\sigma \|^2_{L_{\infty}(I;L_2(\Omega))} \intJn{ \| \varphi_{n,h} \|^2_{H^1(\Omega)}}  {+} \frac{\epsilon^2}{4}  \intJn{ \| \xi_{n,h} \|^2_{H^1(\Omega)}},
  \end{align*}
where we have used the embedding $H^1(\Omega) \subset L_8(\Omega)$, and $C$ is a constant depending only on the domain. The remaining of the proof follows identical to the three dimensional case. 
 \end{proof}

The following result is analogous to Corollary \ref{consequence_Theorem_1} and an immediate consequence of the  Theorem \ref{Theorem_2}, Lemma \ref{lipcon-5} and triangle inequality.

\begin{corollary}\label{consequence_Theorem_2}
Let $u,\varv \in L_2(I;L_2(\Omega))$ and $\varphi^{}_u \in H^{2,1}_{}(\Omega^{}_T)\cap C(I;H^{1}(\Omega))$ be the solution of \eqref{adjoint} while $\varphi^{}_{\sigma}(\varv) \in Y^{}_{\sigma}$ the solution of \eqref{adjoint_state_sigma} corresponding to the control $\varv.$ Assume that  Lemmas \ref{lipcon-5} and \ref{dstability_adjoint} and Theorem \ref{Theorem_2} hold and let  $r \in \lbrace 1,2 \rbrace$.
Then, for $d=2,3$ there holds
 \begin{equation}\label{es_11}
\begin{aligned}
 &    \norm{ \varphi^{}_{u} {-} \varphi^{}_{\sigma}(\varv)}^{}_{L_{\infty}(I;L_2(\Omega))} + \epsilon \norm{ \varphi^{}_{u} {-} \varphi^{}_{\sigma}(\varv)}^{}_{L^{}_{2}(I;H^{1}(\Omega))} \\
 &   \leq   L^{}_d \epsilon^{-s_{3}} \norm{u {-} \varv}^{}_{L^{}_{2}(I;L^{}_2(\Omega))} + \Hat{C}_d \epsilon^{-\rho_{d}}(\sqrt{k}{+}h),
\end{aligned}
\end{equation}
where $s_3=7/2$ when $d=2$ and $s_3= 15/4$ when $d=3$, and $\rho_d$, $\hat C_d$ are defined as Theorem in \ref{Theorem_2}.
 Here, according to the notation of Lemma \ref{lipcon-5} we denote by
 \begin{equation}\label{L23} \hskip-8pt
     L^{}_2:= L^{}_1 \left( \epsilon^{7/2} C^{}_{T} E^{1/2}_{\varphi} {+} \tilde{c} C^{1/2}_{\infty} D^{}_{\text{st},1} \right), L^{}_3:= L^{}_1 \left( \epsilon^{15/4} C^{}_{T} E^{1/2}_{\varphi} {+} \tilde{c} C^{1/2}_{\infty} D^{}_{\text{st},1} \right).
 \end{equation}
 \end{corollary}


\begin{remark} \label{final}
To quantify the dependence upon $1/\epsilon$, in the worst case scenario where $\epsilon \| \nabla y_0\|_{L_2(\Omega)} + \|y_0\|^2_{L_4(\Omega)} \leq C$ (which corresponds the case $r=2$) observe first that $L_2$ and $L_3$ of Corollary \ref{consequence_Theorem_2} are bounded independent of $\epsilon$  since the constant $D^{dG}_{st,1}$ of Lemma 17 is bounded independent of $\epsilon$ (see Remark \ref{final1}) and $L_1$ is also bounded independent of $\epsilon$. Note also that for $d=2$, $\Hat{C}_2$ is bounded independent of $\epsilon$. Indeed, from (4.9) of Lemma 11, we deduce that $\|y_\sigma\|_{L_{\infty}(I;L_4(\Omega))}$ is bounded independent of $\epsilon$.  From the definition of $C_{st,2}$ (Lemma \ref{stability}, see also Remark \ref{first}), and of  $C^{dG}_{st,2}$ (Lemma 11), imply that $C_{st,2}+ C^{dG}_{st,2} \approx \epsilon^{-1}$ and hence the desired estimate for $\Hat{C}_2$ follows. Therefore, the estimate \eqref{es_11}, when $d=2$ and $r=2$,  is dominated by 
$$ {\mathcal C} \left ( \epsilon^{-7/2} \|u- \varv \|_{L_2(I;L_2(\Omega))} +  \epsilon^{-6} (\sqrt{k} + h) \right ).$$
where ${\mathcal C}$ denotes a constant that depends only on data. 
Similarly for $d=3$, $r=2$ we have $s_3 = 15/4$ and $\rho_3 = 4 + 13/6$, while ${\Hat C}_3$ is bounded by $1/\epsilon$, hence the estimate takes the form 
$$ {\mathcal C} \left ( \epsilon^{-15/4} \|u- \varv \|_{L_2(I;L_2(\Omega))} +  \epsilon^{-(5+(13/6))} (\sqrt{k} + h) \right ).$$ 
\end{remark}

\subsection{Convergence of the discrete control problem}
In this section we study the convergence of the solutions of the discrete control problem \eqref{P_sigma} towards the solutions of the continuous problem \eqref{P}. Every  discrete problem \eqref{P_sigma} has at least one solution because the minimisation function is continuous and coercive on a nonempty closed subset of a finite dimensional space.

\begin{theorem}\label{Theorem_4}
 For every $\sigma = (k,h)$, let $\bar{u}^{}_{\sigma}$ be a global solution of problem \eqref{P_sigma}. Then, the sequence $\lbrace \bar{u}^{}_{\sigma} \rbrace^{}_{\sigma}$ is bounded in $L^{}_{2}(I;L^{}_2(\Omega))$ and there exist subsequences denoted in the same way, converging to a point $\bar{u}$ weakly in $L^{}_{2}(I;L^{}_2(\Omega))$. Any of these limit points is a solution of problem \eqref{P_sigma}. Moreover, we have
 \begin{equation}\label{convergence_2}
     \lim^{}_{\sigma \rightarrow 0} \norm{ \bar{u} - \bar{u}^{}_{\sigma}}^{}_{L^{}_{2}(I;L^{}_2(\Omega))} = 0 \quad \textit{and} \quad \lim^{}_{\sigma \rightarrow 0} J^{}_{\sigma}( \bar{u}^{}_{\sigma} ) = J(\bar{u} ).
 \end{equation}
\end{theorem}

\begin{theorem}\label{Theorem_5}
 Let $\bar{u}$ be a strict local minimum of \eqref{P}. Then, there exists a sequence of local minima of problems \eqref{P_sigma} such that \eqref{convergence_2} holds.
\end{theorem}
The proofs of the  Theorems are presented in \cite[Theorems 4.15 and 4.17]{Casas-Chrysafinos_2012}. In our case, the above convergence results are proved for each fixed interface length, $\epsilon$.

\subsection{Error estimates}
In this section we denote by $\bar{u}$ a local solution of \eqref{P} and by $\bar{u}^{}_{\sigma}$ a local solution of \eqref{P_sigma} $\forall \sigma.$ From Theorems \ref{Theorem_4} and \ref{Theorem_5}, we deduce that $\norm{ \bar{u} {-} \bar{u}^{}_{\sigma} }^{}_{L_2(I;L_2(\Omega))} \rightarrow 0.$ Moreover, let $\bar{y}$ and $\bar{\varphi}$ be the state and adjoint state associated to $\bar{u}$ while $\bar{y}^{}_{\sigma}$ and $\bar{\varphi}^{}_{\sigma}$ the discrete state and adjoint state associated to $\bar{u}^{}_{\sigma}$.

 \begin{definition}
 We denote by $u^{}_{\sigma}$ the projection of $\bar{u}$ into $U^{}_{\sigma}$ define through,
 \begin{align}\label{projection}
    u^{}_{\sigma} = \sum^{N}_{n=1}\sum^{}_{\tau \in \mathcal{T}^{}_{h}} u^{}_{n,\tau} \mathcal{\chi}^{}_{n} \mathcal{\chi}^{}_{\tau}  \quad \textit{where} \quad u^{}_{n,\tau} = \frac{1}{k^{}_n \abs{\tau}} \intJn{ \int^{}_{\tau} \abs{ \bar{u}(t,x) } \diff{x} }.
 \end{align}
 \end{definition}
 \begin{lemma}
 Let $\bar{u} \in H^{1}_{}(\Omega^{}_{T}).$ Then, there exists a constant $C>0$ independent of $\sigma=(k,h)$ and $\epsilon$ such that,
 \begin{align}\label{approx_projection_1}
       \norm{ \bar{u}  {-} u^{}_{\sigma} }^{}_{L_2(I;L_2(\Omega))} & \leq C ( \sqrt{k} {+}  h ) \norm{\bar{u}}^{}_{H^{1}_{}(\Omega^{}_{T})},\\ \label{approx_projection_2}
         \norm{ \bar{u}  {-} u^{}_{\sigma} }^{}_{(H^{1}_{}(\Omega^{}_{T}))^*} & \leq C \left( k {+}  h^2 \right) \norm{\bar{u}}^{}_{H^{1}_{}(\Omega^{}_{T})}.
 \end{align}
 \end{lemma}
\proof We refer the reader to \cite[Lemma 4.17]{Casas-Chrysafinos_2012}. \endproof
We are ready to present the main result. Once again we have to distinguish between the two dimensional and three dimensional cases. 
 \begin{theorem}\label{Theorem_6}
 Let the assumptions of  Theorems \ref{1st-necessary} and \ref{2nd-necessary} and Corollaries \ref{consequence_Theorem_1} and \ref{consequence_Theorem_2} hold. 
Then,  there exist $\mathcal{C}^{}_{1} :=  \sqrt{T} L^{}_2$, $\widetilde{\mathcal{C}}^{}_1 : = L^{}_{1}\mathcal{C}^{}_{1}$ such that for $d=2$ and $r=1$
\begin{align}\label{convergence_of_u_sigma_1}
   & \norm{ \bar{u} {-} \bar{u}^{}_{\sigma}}^{}_{L_2(I;L_2(\Omega))} \leq \frac{\mathcal{C}^{}_1}{\epsilon^{7/2}}(\sqrt{k} {+} h)  \norm{\bar{u}}^{}_{H^{1}_{}(\Omega^{}_{T})}, \\ \label{convergence_of_y_sigma_1}
   & \norm{ \bar{y} {-} \bar{y}^{}_{\sigma}}^{}_{L_{\infty}(I;L_2(\Omega))} + \epsilon \norm{ \bar{y} {-} \bar{y}^{}_{\sigma}}^{}_{L^{}_2(I;H^{1}(\Omega))}  \leq  \frac{\widetilde{\mathcal{C}}^{}_1}{\epsilon^{7/2}}(\sqrt{k} {+} h)  \norm{\bar{u}}^{}_{H^{1}_{}(\Omega^{}_{T})},
\end{align}
while for $r=2$ there exist $\mathcal{C}^{}_{2} :=  \sqrt{T} \Hat{C}_2$, where $\Hat{C}_2$ is defined by  $\widetilde{\mathcal{C}}^{}_2 : = L^{}_{1}\mathcal{C}^{}_{2}$ such that
\begin{align}\label{convergence_of_u_sigma_1_1}
   & \norm{ \bar{u} {-} \bar{u}^{}_{\sigma}}^{}_{L_2(I;L_2(\Omega))} \leq \frac{\mathcal{C}^{}_2}{\epsilon^{6}}(\sqrt{k} {+} h), \\ \label{convergence_of_y_sigma_1_1}
   & \norm{ \bar{y} {-} \bar{y}^{}_{\sigma}}^{}_{L_{\infty}(I;L_2(\Omega))} + \epsilon \norm{ \bar{y} {-} \bar{y}^{}_{\sigma}}^{}_{L^{}_2(I;H^{1}(\Omega))}  \leq  \frac{\widetilde{\mathcal{C}}^{}_2}{\epsilon^{6}}(\sqrt{k} {+} h).
\end{align}
In addition,  for $d=3$ and  $r=1$  there exist $\mathcal{C}^{}_{3}:= \sqrt{T} L^{}_{3},$ $\widetilde{\mathcal{C}}^{}_3 : = L^{}_{1}\mathcal{C}^{}_{3}$ such that
\begin{align}\label{convergence_of_u_sigma_2}
   & \norm{ \bar{u} {-} \bar{u}^{}_{\sigma}}^{}_{L_2(I;L_2(\Omega))} \leq \frac{\mathcal{C}^{}_3}{\epsilon^{15/4}}(\sqrt{k} {+} h)  \norm{\bar{u}}^{}_{H^{1}_{}(\Omega^{}_{T})}, \\ \label{convergence_of_y_sigma_2}
   & \norm{ \bar{y} {-} \bar{y}^{}_{\sigma}}^{}_{L_{\infty}(I;L_2(\Omega))} + \epsilon \norm{ \bar{y} {-} \bar{y}^{}_{\sigma}}^{}_{L^{}_2(I;H^{1}(\Omega))}  \leq  \frac{\widetilde{\mathcal{C}}^{}_3}{\epsilon^{15/4}}(\sqrt{k} {+} h)  \norm{\bar{u}}^{}_{H^{1}_{}(\Omega^{}_{T})}.
\end{align}
Finally, for $d=3$ and $r=2$ there exist $\mathcal{C}^{}_{4}= \sqrt{T} \Hat{C}_3$,  $\widetilde{\mathcal{C}}^{}_4 : = L^{}_{1}\mathcal{C}^{}_{4}$ such that
\begin{align}\label{convergence_of_u_sigma_2_1}
   & \norm{ \bar{u} {-} \bar{u}^{}_{\sigma}}^{}_{L_2(I;L_2(\Omega))} \leq \frac{\mathcal{C}^{}_{4}}{\epsilon^{4+13/6}} (\sqrt{k} {+} h)
   , \\ \label{convergence_of_y_sigma_2_1}
   & \norm{ \bar{y} {-} \bar{y}^{}_{\sigma}}^{}_{L_{\infty}(I;L_2(\Omega))} + \epsilon \norm{ \bar{y} {-} \bar{y}^{}_{\sigma}}^{}_{L^{}_2(I;H^{1}(\Omega))}  \leq  \frac{\widetilde{\mathcal{C}}^{}_4}{\epsilon^{{4+13/6}}} (\sqrt{k} {+} h).
\end{align}
In the above $\Hat C_d$ are defined in Theorem \ref{Theorem_2} (see also Remark \ref{final}), where $\Hat{C}_2$ is bounded independent of $\epsilon$, and $\Hat{C}_3 \approx 1/\epsilon$.  
\end{theorem}
\begin{proof} ({\it Sketch}:) The proof follows \cite[Section 4]{Casas-Chrysafinos_2012}.
  We  begin with  $d=3$ and $r=1$. Note that \eqref{convergence_of_y_sigma_2}   is a consequence of \eqref{convergence_of_u_sigma_2} combined with \eqref{es_4_2}. Thus,  the main subject is the proof of \eqref{convergence_of_u_sigma_2}. Suppose that  \eqref{convergence_of_u_sigma_2}  is false. More specifically,  this assumption implies that
 \begin{align*}
    \lim^{}_{\sigma \rightarrow 0} \sup \frac{\epsilon^{15/4} \norm{\bar{u}}^{-1}_{H^{1}_{}(\Omega^{}_{T})}}{\mathcal{C}^{}_{3}(\sqrt{k}+ h)} \norm{ \bar{u} - \bar{u}^{}_{\sigma}}^{}_{L_2(I;L_2(\Omega))} =  + \infty.
 \end{align*}
 This means that  there exists a sequence of $\sigma$ such that
 \begin{align}\label{false_assumption}
    \lim^{}_{\sigma \rightarrow 0}  \frac{\epsilon^{15/4} \norm{\bar{u}}^{-1}_{H^{1}_{}(\Omega^{}_{T})}}{\mathcal{C}^{}_{3}(\sqrt{k}+ h)} \norm{ \bar{u} - \bar{u}^{}_{\sigma}}^{}_{L_2(I;L_2(\Omega))} =  + \infty.
 \end{align}
 We aim to  conclude to a contradiction. First of all, the fact that  $\bar{u}^{}_{\sigma}$ is a local minimum of \eqref{P_sigma} combined with the property that  $J^{}_{\sigma}$ is of class $C^{\infty}_{}$ and $u^{}_{\sigma} \in U^{}_{\sigma,ad}$ imply that $J'^{}_{\sigma}(\bar{u}^{}_{\sigma})(u^{}_{\sigma} - \bar{u}^{}_{\sigma}) \geq 0 $. After basic manipulations, we have that
\begin{align*}
    & J'(\bar{u}^{}_{\sigma}) (\bar{u} - \bar{u}^{}_{\sigma}) + [ J'^{}_{\sigma}(\bar{u}^{}_{\sigma}) - J'^{}(\bar{u}^{}_{\sigma})] (\bar{u} - \bar{u}^{}_{\sigma}) \\
    & + [J'^{}_{\sigma}(\bar{u}^{}_{\sigma}) - J'^{}(\bar{u})] (  u^{}_{\sigma} - \bar{u}) + J'( \bar{u} ) ( u^{}_{\sigma} - \bar{u} ) \geq 0.
\end{align*}
Since  $\bar{u}$ is an local minimum of \eqref{P} and $\bar{u}^{}_{\sigma} \in U^{}_{ad}$, then $J'(\bar{u})( \bar{u}^{}_{\sigma} - \bar{u} ) \geq 0$. Adding this nonnegative term to the above inequality, it yields that
\begin{equation}\label{crucial_inequality}
    \begin{aligned}
     & [ J'(\bar{u}^{}_{\sigma}) - J'( \bar{u} ) ] (\bar{u}^{}_{\sigma} - \bar{u}) \leq  [ J'^{}_{\sigma}(\bar{u}^{}_{\sigma}) - J'^{}(\bar{u}^{}_{\sigma})] (\bar{u} - \bar{u}^{}_{\sigma})\\
     & + [J'^{}_{\sigma}(\bar{u}^{}_{\sigma}) - J'^{}(\bar{u})] (  u^{}_{\sigma} - \bar{u}) + J'( \bar{u} ) ( u^{}_{\sigma} - \bar{u} ).
    \end{aligned}
\end{equation}
Using the second order necessary and sufficient condition of Theorem \ref{2nd-necessary} we derive an estimate from below for the left-hand side and then an upper bound for the three terms on the right-hand side  through the estimates of Sections 3.1, 3.2, 4.1 and 4.2. Indeed, working identically to \cite[Lemma 4.18]{Casas-Chrysafinos_2012} we obtain, for a suitable $\sigma_0$ independent of $\epsilon$,
\begin{equation*}
    \frac{1}{2} \min \lbrace \delta, \mu \rbrace  \norm{ \bar{u}^{}_{\sigma} - \bar{u} }^{2}_{L_2(I;L_2(\Omega))}  \leq   [ J'(\bar{u}^{}_{\sigma}) - J'( \bar{u} ) ] (\bar{u}^{}_{\sigma} - \bar{u})  \quad \text{for} \quad \abs{\sigma} \leq \abs{\sigma^{}_0}.
\end{equation*}
Inserting the  above lower bound  in \eqref{crucial_inequality} we obtain that,
\begin{equation}\label{crucial_inequality_update}
\begin{aligned}
     & \frac{1}{2} \min \lbrace \delta, \lambda \rbrace  \norm{ \bar{u}^{}_{\sigma} - \bar{u} }^{2}_{L_2(I;L_2(\Omega))} \leq  [ J'^{}_{\sigma}(\bar{u}^{}_{\sigma}) - J'^{}(\bar{u}^{}_{\sigma})] (\bar{u} - \bar{u}^{}_{\sigma})\\
     & + [J'^{}_{\sigma}(\bar{u}^{}_{\sigma}) - J'^{}(\bar{u})] (  u^{}_{\sigma} - \bar{u}) + J'( \bar{u} ) ( u^{}_{\sigma} - \bar{u} ).
\end{aligned}
\end{equation}
We shall estimate from above the three terms on the right-hand side. More specifically, from  \eqref{1-deriv_J} and \eqref{1-deriv_J_sigma_alternate}
\begin{equation}\label{1st_term}
\begin{aligned}
    & [ J'^{}_{\sigma}(\bar{u}^{}_{\sigma}) - J'^{}(\bar{u}^{}_{\sigma})] (\bar{u} - \bar{u}^{}_{\sigma}) \leq \norm{  \bar{\varphi}^{}_{\sigma} - \varphi^{}_{\bar{u}^{}_{\sigma}}}^{}_{L_2(I;L_2(\Omega))}  \norm{  \bar{u} - \bar{u}^{}_{\sigma}  }^{}_{L_2(I;L_2(\Omega))}\\
     & \ \leq \sqrt{T} \norm{ \bar{\varphi}^{}_{\sigma} - \varphi^{}_{\bar{u}^{}_{\sigma}}}^{}_{L_{\infty}(I;L_2(\Omega))} \norm{\bar{u} - \bar{u}^{}_{\sigma}  }^{}_{L_2(I;L_2(\Omega))} \\
    & \ \leq \frac{\sqrt{T} \Hat{C}_3 }{\epsilon^{4+(7r-1)/6}}   ( \sqrt{k} + h )\norm{\bar{u} - \bar{u}^{}_{\sigma}  }^{}_{L_2(I;L_2(\Omega))},
\end{aligned}
\end{equation}
 we used the estimate \eqref{es_11} for $u=\varv=\bar{u}^{}_{\sigma}.$
For the second term on the right-hand side of \eqref{crucial_inequality_update}, we recall again \eqref{es_11} for $u=\bar{u}$ and $\varv = \bar{u}^{}_{\sigma}$ and \eqref{approx_projection_1}
\begin{equation}\label{2nd_term}
\begin{aligned}
    & [J'^{}_{\sigma}(\bar{u}^{}_{\sigma}) - J'^{}(\bar{u})] (  u^{}_{\sigma} - \bar{u}) \\
    &  \leq \left( \norm{ \bar{\varphi}^{}_{\sigma} - \bar{\varphi} }^{}_{L_2(I;L_2(\Omega))} + \mu \norm{\bar{u} - \bar{u}^{}_{\sigma}  }^{}_{L_2(I;L_2(\Omega))} \right) \norm{ u^{}_{\sigma} - \bar{u}  }^{}_{L_2(I;L_2(\Omega))}\\
    &  \leq  C \Bigg \lbrace  \Big( \frac{\sqrt{T} L^{}_3}{\epsilon^{15/4}} + \mu \Big) \norm{\bar{u} - \bar{u}^{}_{\sigma}  }^{}_{L_2(I;L_2(\Omega))} {+} \frac{\sqrt{T}  \Hat{C}_3 }{\epsilon^{4+(7r-1)/6}}  ( \sqrt{k} {+} h )
     \Bigg \rbrace ( \sqrt{k} {+} h) \norm{\bar{u}}^{}_{H^{1}_{}(\Omega^{}_T)}.
\end{aligned}
\end{equation}
Finally, using \eqref{projectionformula} to bound $\|\bar{\varphi}{+} \mu \bar u\|_{H^1(\Omega_T)} \leq C \| \bar{\varphi} \|_{H^{1}(\Omega_T)}$  and \eqref{approx_projection_2} we have
\begin{equation}\label{3rd_term}
    \begin{aligned}
    J'( \bar{u} ) ( u^{}_{\sigma} {-} \bar{u} ) & \leq  \norm{\bar{\varphi} + \mu \bar{u}}^{}_{H^{1}_{}(\Omega^{}_{T})} \norm{ u^{}_{\sigma} {-} \bar{u}} ^{}_{ (H^{1}_{}(\Omega^{}_{T}))^* }  \leq  C \left( k {+} h^2 \right) \norm{\bar{u}}^{}_{H^{1}_{}(\Omega^{}_T)} \norm{\bar{\varphi}}^{}_{H^{1}_{}(\Omega^{}_T)}.
    \end{aligned}
\end{equation}
Applying Young's inequality on the right-hand side of \eqref{1st_term}, \eqref{2nd_term}, \eqref{3rd_term} and collecting the terms properly in \eqref{crucial_inequality_update}, we derive
\begin{align*}
    & \norm{\bar{u} {-} \bar{u}^{}_{\sigma}  }^{}_{L_2(I;L_2(\Omega))}  {\leq}  \max \Bigg \lbrace \frac{\sqrt{T} L^{}_3}{\epsilon^{15/4}} \norm{\bar{u}}^{}_{H^{1}_{}(\Omega^{}_T)}, \frac{\sqrt{T} \Hat{C}_3 }{\epsilon^{4+(7r-1)/6}} \Bigg \rbrace ( \sqrt{k} + h).
\end{align*}
We note that from \eqref{projectionformula} and Theorem \ref{1st-necessary} we have that $\| \bar u\|_{H^{1}(\Omega_T)} {\sim} \epsilon^{-2}$ hence for $r=1$ the first term dominates the maximum. Hence, we deduce the desired estimate for $\mathcal{C}^{}_{3} {=} \sqrt{T} L^{}_3$ \eqref{convergence_of_u_sigma_2}. For the proof of \eqref{convergence_of_y_sigma_2}, we recall \eqref{es_4_2} and \eqref{es_4_3} for $u {=} \bar{u}$, $\varv {=}\bar{u}^{}_{\sigma}$ combined with  \eqref{convergence_of_u_sigma_2}.
Working identically for $d{=}3$ and $r{=}2$, we deduce \eqref{convergence_of_u_sigma_2_1} and \eqref{convergence_of_y_sigma_2_1}, respectively. We proceed similarly for $d=2$-case.
\end{proof}

\section{The spectral estimate for the nonhomogeneous Allen-Cahn equation}
\setcounter{equation}{0}
We discuss a simple generalization of the spectral estimate in case of the nonhomogeneous Allen-Cahn equation.
We observe that in case of our optimal control setting, since right hand side $u \in {U}_{ad}$ we deduce from the regularity estimates of Lemma 2.1 that $y_u \in H^{2,1}(\Omega_T)$ with bounds that depend polynomially upon $1/ \epsilon$. In particular, we note that the stability bounds exhibit similar dependence upon $1/ \epsilon$ to homogeneous case (see e.g.  \cite{Feng-Prohl_2003}, \cite{Bartels-Muller-Ortner_2011}).
Throughout this work, and similar to the above mentioned works, we assume that (a uniform)  bound for the spectral estimate holds for the linearized operator associated to the solution $v$ of the homogeneous (uncontrolled) Allen-Cahn equation. Indeed, given initial data $v_0$, and with zero forcing term, $v$ is the solution of,
the following problem:
\begin{equation}\label{app1}
\begin{aligned}
   v_{t}^{} - \Delta v + \frac{1}{{\epsilon}^2}(v^3 -v) & = 0  & \text{in} & \ \Omega^{}_{T}=\Omega \times (0,T), \\
 \frac{\partial v}{\partial n}   &= 0 \quad\text{or}\quad v =0    & \text{on} & \ \Sigma^{}_{T}=\partial\Omega \times (0,T), \\
   v(\cdot,0)&= v_0^{}   & \text{in} & \ \Omega.
\end{aligned}
\end{equation}
The spectral estimate of the associated linearized operator concerns the quantity:
\begin{equation}\label{app2}
-\Lambda^{}(t) : = {\displaystyle\inf_{\varv\in H^{1}(\Omega) \setminus \lbrace{0 \rbrace}}} \frac{\norm{\nabla \varv}^2_{L^{}_2(\Omega)} + \epsilon^{-2} \left(( 3v^2 -1) \varv,\varv \right)}{\norm{\varv}^2_{L^{}_2(\Omega)}}.
\end{equation}
In \cite{Chen_1994,Mottoni-Schatzman_1995,Alikakos-Fusco_1993}  it is showed that $\Lambda \in L^{}_{\infty}(I)$ can be bounded independently of $\epsilon$ for the case of smooth, evolved interfaces.
We note that a careful inspection of the proof of \cite{Chen_1994} reveals that the uniform bound is a property related to the interface profile generated by the initial data.
\begin{lemma} \label{spect}
Suppose that the initial data profile $v_0$ is such that the spectral estimate \eqref{app2} for the linearized operator associated to the solution $v$ of \eqref{app1},  with $\| \Lambda \|_{L_{\infty}(I)} \leq C$ holds, where $C$ denotes a constant independent of $\epsilon$.
Then, for $y_0$ such that $\|y_0\|_{L_{\infty}(\Omega)} \leq C \epsilon^{-1}$ the spectral estimate for the linearized operator related to the solution $y_u$ of \eqref{state} with $u \in {U}_{ad}$ holds with uniform in time bound, i.e.,
\begin{equation}\label{app2b}
-\lambda^{}(t) : = {\displaystyle\inf_{\varv\in H^{1}(\Omega) \setminus \lbrace{0 \rbrace}}} \frac{\norm{\nabla \varv}^2_{L^{}_2(\Omega)} + \epsilon^{-2} \left(( 3y^2_{u}-1) \varv,\varv \right)}{\norm{\varv}^2_{L^{}_2(\Omega)}}.
\end{equation}
 where $\| \lambda \|_{L_{\infty}(I)} \leq \tilde C$, with $\tilde C$ a constant depending on $C$ and independent of $\epsilon$.
\end{lemma}
\begin{proof} 
Given any control $u \in {U}_{ad}$ and its corresponding state $y_u$, we define $w$ as the solution of the following auxiliary problem:
\begin{equation}\label{app3}
\begin{aligned}
   w_{t}^{} - \Delta w + \frac{1}{{\epsilon}^2}(w^3 -w) & = \frac{1}{\epsilon^2} (3y^{}_u w^2-3w y^2_{u}) - u  & \text{in} & \ \Omega^{}_{T}=\Omega \times (0,T), \\
 \frac{\partial w}{\partial n}   &= 0 \quad\text{or}\quad w =0    & \text{on} & \ \Sigma^{}_{T}=\partial\Omega \times (0,T), \\
   w(\cdot,0)&= (1-\epsilon^2)y_0   & \text{in} & \ \Omega.
\end{aligned}
\end{equation}
Due to the regularity of $y^{}_u$, it is evident that there exists solution $w \in W(I)$ (at least) of \eqref{app3}.
We note that $v:= y^{}_u-w$ satisfies \eqref{app2}, with $v_0 = \epsilon^2 y_0.$ In addition, a straightforward application of the maximum principle implies that $\|v\|_{_{\infty}} \leq \epsilon^2 \|y_0 \|_{L_{\infty}(\Omega)}.$ Let $\varphi \in H^1(\Omega)$, with $\varphi \neq 0.$ Then, substituting $y^{}_u = v+w$, using the \eqref{app1} for $v$, and the inequality $ab \geq - \frac{1}{2} (a^2 + b^2)$
\begin{align*}
& \| \nabla \varphi \|^2_{L_2(\Omega)} + \frac{1}{\epsilon^2} \left ( (3y^2_{u}-1) \varphi, \varphi \right )  = \| \nabla \varphi \|^2_{L_2(\Omega)} + \frac{1}{\epsilon^2} \left ( (3v^2 -1) \varphi, \varphi \right ) \\
& \quad +  \frac{1}{\epsilon^2} \left ( 3w^2  \varphi, \varphi \right ) + \frac{1}{\epsilon^2} \left (  6v w \varphi, \varphi \right ) \\
& \geq - \Lambda(t) \| \varphi \|^2_{L_2(\Omega)} - \frac{1}{\epsilon^2} \left ( 3v^2 \varphi, \varphi \right ) 
 \geq - \Lambda(t) \| \varphi \|^2_{L_2(\Omega)} - \frac{3}{\epsilon^2} \|v\|^2_{_{\infty}} \| \varphi \|^2_{L_2(\Omega)} \\
& \geq - \left ( \Lambda(t) + \epsilon^2 \|y_0\|^2_{L^{}_{\infty}(\Omega)} \right ) \|\varphi \|^2_{L_2(\Omega)} := - \lambda (t) \| \varphi \|^2_{L_2(\Omega)},
\end{align*}
where at the last step we have used the estimate $\|v\|_{_{\infty}} \leq \epsilon^2 \|y_0 \|_{L_{\infty}(\Omega)}.$ Note that if $\|y_0\|_{L_{\infty}(\Omega)} \leq C \epsilon^{-1}$, then $\|\lambda \|_{L_{\infty}(I)}$ is bounded independently of $\epsilon$.
\end{proof}

\section*{Acknowledgements}
This research work was supported by the Hellenic Foundation for Research and Innovation (H.F.R.I.) under 1) the ``First Call for H.F.R.I. Research Projects to support Faculty members and Researchers and the procurement of high-cost research equipment grant" (Project Number: 3270, support for KC) and 2) a ``H.F.R.I. PhD Fellowship grant" (Fellowship Number: 998, support for DP).



\footnotesize
\bibliographystyle{amsplain}
\bibliography{ChrPl}
\footnotesize





\end{document}